\documentclass[12pt]{article}
\usepackage[latin1]{inputenc}

  \usepackage{amssymb}         
  \usepackage{amsmath}          
  \usepackage{amsfonts}           
  \usepackage{amsthm}

  \usepackage[mathscr]{eucal}

  \usepackage{graphicx}        

\def\pq{{p/q}}

 \newcommand*{\bfnf}{{\mbox{\boldmath$\Gamma$}}}
  \renewcommand*{\H}{\mathbb{H}}
  \newcommand*{\inter}{\ensuremath{\operatorname{int}}}

  \newcommand*{\opq}{\omega_{p/q}}
  \newcommand*{\ls}{\lambda,\sigma}
  
  \newcommand{\per}{\operatorname{Per}_1}
  \newcommand*{\R}{\mathbb{R}}

  \newcommand*{\Z}{\mathbb{Z}}
  \newcommand*{\C}{\mathbb{C}}
  \newcommand*{\N}{\mathbb{N}}
  
  \newcommand*{\Q}{\mathbb{Q}}
  \newcommand*{\T}{\mathbb{T}}
  \newcommand*{\D}{\mathbb{D}}
  \newcommand*{\m}[1]{\mathbb{#1}}
  \newcommand*{\CC}{\widehat\C}

\newtheorem{definition}{Definition}
\newtheorem{etheorem}{Theorem}
\newtheorem*{denjoy*}{Denjoy-Wolff Theorem}
\newtheorem*{bers*}{Bers Inequality}
\newtheorem{cor}{Corollary}

\newtheorem{lemma}{Lemma}
\newtheorem{prop}{Proposition}
\newtheorem*{prop*}{Proposition}

\def\wht{\widehat\tau}
\def\wtga{\widetilde\gamma}
\def\wtt{\widetilde\tau}

  \newcommand{\Log}{{\operatorname{Log}}}
  \newcommand{\e}{{\operatorname{e}}}

\usepackage{color} 

\title{Limits of Quadratic Rational Maps:\\
 The Cantor Locus}

\author{Eva Uhre \\
  \small{\texttt{euhre@ruc.dk}} \\
  \small{Department of Science and Environment, Roskilde University}}

\date{\today}

\begin{document}

\maketitle

\begin{abstract}
The \emph{Cantor locus} is the unique hyperbolic component, in the moduli space of quadratic rational maps ${\bf rat}_2$, consisting of maps with totally disconnected Julia sets. Whereas the geometry and dynamics of the Cantor locus is well understood, its boundary and the dynamics of the maps on the boundary are not.
In this paper, we explore the dynamics near the parabolic parts of the boundary. 

We introduce the concept \emph{dynamical marking} of a map $g$, 
relative to the quadratic, parabolic polynomial $\mathrm{P}_{\opq}(z)={\opq} z+z^2$, with $\opq=e^{2\pi ip/q}$. A dynamical marking $(x,\psi)$ of $g$ is a conjugacy $\psi$ between $\mathrm{P}_{\opq}$ (on its parabolic basin of 0) and $g$, which \emph{marks} the dynamical position of the critical values $v_1=\psi(-\frac{\lambda^2}{4})$, $v_2=\psi(x)$ of $g$.

We construct a local parametrization of the Cantor locus, which parametrizes by dynamical marking, and use it to prove a form of \emph{stability} of dynamical marking. That is, for sequences in the Cantor locus, of fixed dynamical marking $x$ and such that the eigenvalue $\lambda_k$ of the attracting fixed point tends to $\opq$ \emph{subhorocyclicly}, either the sequence converges to the unique parabolic parameter in the boundary, which has a fixed point eigenvalue $\opq$ and which is marked by $x$ relative to $\mathrm{P}_{\opq}$. Or, the sequence tends to infinity in ${\bf rat}_2$, and certain representatives $G_{\lambda_k,a_k}$ have \emph{rescaled} limits in the boundary of the Cantor locus within ${\bf rat}_2$.
\end{abstract}

\setcounter{tocdepth}{1}
\tableofcontents

\section{Introduction}
The moduli space ${\bf rat}_2$ of all quadratic rational maps modulo M\"{o}bius conjugacy is isomorphic to $\C^2$, with the two first elementary symmetric functions of the eigenvalues at the fixed points as affine coordinates, as described by Milnor in \cite{mil93}. Moreover, the loci
\[
\per(\lambda )=\{[f]\in{\bf rat}_2:\text{some fixed point
of }f\text{ has eigenvalue }\lambda \} 
\]
are lines in the affine structure given by this coordinate system. A conjugacy class $[f]\in{\bf rat}_2$ is uniquely determined by the three fixed point eigenvalues $\lambda, \mu, \nu$ of $f$. For any ordered pair $(\lambda,\sigma)\in\C^2$, there is a unique conjugacy class ${\bfnf}_{\lambda,\sigma}$ consisting of maps such that some fixed point
has eigenvalue $\lambda$ and such that  the product of the eigenvalues of the
two other fixed points is $\sigma$. 
For fixed $\lambda \in \C$ the map
\[\C\ni \sigma\,\mapsto \,{\bfnf}_{\lambda,\sigma}\in \per(\lambda)\]
is an isomorphism, in fact it is an affine parametrization of the line $\per(\lambda)$ \cite[Remark 6.9]{mil93}. 

A rational map is \emph{hyperbolic} if and only if the orbit of every critical point converges to some attracting periodic orbit. The subset of moduli space consisting of hyperbolic maps is open and its connected components are called \emph{hyperbolic components}. The classification and description of hyperbolic components of ${\bf rat}_2$ was pioneered by Rees \cite{rees90} and Milnor \cite{mil93}, these groundbreaking studies also supplied an overview of the landscape of hyperbolic components. This landscape has been further described by others; compactness questions have been studied, for example, by Epstein \cite{eps00}.

The moduli space ${\bf poly}_2$ (canonically isomorphic to $\per(0)$) of all quadratic polynomials modulo affine conjugacy is isomorphic to $\C$, and can be parametrized by the family $\mathrm{Q}_{c}(z)=z^{2}+c$, $c\in \C$, in which case $\sigma=4c$. By contrast, there is no natural global parametrization of ${\bf rat}_2$. However, the two-parameter family
  \begin{equation}
	G_{\lambda,a}(z)=\frac{1}{\lambda}\left(z+a+\frac{1}{z}\right),
\end{equation}
provides a convenient parametrization of ${\bf rat}_2\setminus \per(0)$. That is, for fixed $\lambda\neq 0$ the family $(G_{\lambda,a})_{a\in \C}$ is a 2:1 branched parametrization of $\per(\lambda)$, with branch point $a=0$ corresponding to the unique class $[f]\in \per(\lambda)$ in which $f$ has a non-trivial automorphism (i.e. $[f]$ belongs to the symmetry locus). The relation between $(\lambda,a)$ and $(\lambda,\sigma)$ is given by $a^2=(\lambda-2)^2-\sigma\lambda^2$.

 Let $J(f)$ denote the Julia set of the rational map $f$ and let $M$ denote the well-known Mandelbrot \nolinebreak set 
\[
M=\{c\in \C:J(\mathrm{Q}_{c})\text{ is connected}\}.
\]
The parameter space of the family $(\mathrm{Q}_{c})_{c\in\C}$ contains infinitely many bounded hyperbolic components, all in the interior of $M$, which are naturally parametrized by the eigenvalue of the unique finite attracting cycle, and one unbounded hyperbolic component, $\C\setminus M$. 

For a parameter $c\in \C\setminus M$, 
the critical value $c$ of $\mathrm{Q}_{c}$ is in the basin of the super-attracting fixed point at infinity. The B\"ottcher coordinate $\Phi_c$ (which is asymptotic to the identity at infinity) is a biholomorphic conjugacy between $Q_c$ and $z^2$ in a neighborhood of $\infty$ and the domain of $\Phi_c$ can be extended naturally to contain $c$. Thus, $\Phi_c(c)\in \C\setminus \overline\D$ and the map $c\mapsto \Phi_c(c)$ provides a natural isomorphism between $\C\setminus M$ and $\C\setminus \overline\D$, as pioneered by Douady and Hubbard \cite{douhub}.
This is the most fundamental example of what we here call a \emph{dynamical marking}: $x=\Phi_c(c)$ is a dynamical marking of $Q_c$ relative to $z^2$ (or of the class in $\per(0)$ corresponding to $\sigma=4c$), and the map $x\mapsto c$ is a parametrization of $\C\setminus M$ by dynamical marking.

For a parameter $c\in \C\setminus M$, the Julia set is totally disconnected. This dichotomy holds in general for quadratic rational maps: the Julia set of a quadratic rational map is either connected or totally disconnected, and it is totally disconnected if and only if the eigenvalue of one of the fixed points is in $\D\cup\{1\}$ and the two critical points are in the same Fatou component \cite[Lemma 8.2]{mil93}, see also \cite{rees90}. 
Moreover, a quadratic rational map with totally disconnected Julia set is conjugate on the Julia set to the one-sided shift on two symbols, and the Julia set is homeomorphic to the standard Cantor set. See \cite{gold90}, \cite{mil93}, \cite{rees90}.

The hyperbolic quadratic rational maps with totally disconnected Julia sets form one hyperbolic component in ${\bf rat}_2$. 
This component has various names, such as \emph{type I}, the \emph{hyperbolic escape locus} or the \emph{shift locus}. Here we will call it the \emph{(hyperbolic) Cantor locus}.
This component is well-understood: it is unbounded and homeomorphic to $\D\times (\C\setminus \bar\D)$, with the central slice ($\lambda=0$) corresponding to $\C\setminus M$. See \cite{gold90}, \cite{mil93}.

In the Cantor locus it is convenient to work with the coordinates $(\lambda,\sigma)$, where $\lambda$ is the eigenvalue of the unique attracting fixed point. The lines $\per(\lambda )$, $\lambda\in \D$, intersect within ${\bf rat}_2$,  but they do not intersect in the Cantor locus, since maps here have only one attracting fixed point. 

For each $\lambda\in\D\cup \{1\}$ let 
\[
M^{\lambda }=\left\{ \sigma \in \C: \mbox{ any map in } {\bfnf}_{\lambda,\sigma} \mbox{ has a connected Julia set}\right\}. 
\]
For $\lambda\in\m{S}^{1}$ we will use the notation $\omega _{\theta}=e^{2\pi i\theta }$ for any value of $\theta $, with the convention that $\omega _{p/q}=e^{2\pi i p/q}$ always denotes a
primitive $q$-th root of unity.
For $\lambda\in\D\cup\bigcup_{p/q\in \Q/\Z}\{\opq\}$ we consider also the following subset of $\C$:

\begin{itemize}
\item $\mathcal{R}^{\lambda}$ consisting of all $\sigma\in\C$ such that both critical orbits of any $f\in{\bfnf}_{\lambda,\sigma}$ converge to the fixed point with eigenvalue $\lambda$.
\end{itemize}

And for $\omega=\opq$ the subset
\begin{itemize}
\item $\mathcal{D}^{\omega}$ consisting of all $\sigma\in\C$ such that any $f\in{\bfnf}_{\omega,\sigma}$ 
has a \emph{degenerate} fixed point of eigenvalue $\omega$, i.e. the fixed point has multiplicity $2q+1$ as a fixed point of $f^q$.
\end{itemize}

We will call $\mathcal{R}^{\lambda}$ the \emph{relatedness locus} for $\per(\lambda)$.
For $\lambda\in\D$, $\mathcal{R}^{\lambda}$ is the open set
$\C\setminus M^\lambda$ and for $\lambda=1$, $\mathcal{R}^{1}$ is the complement of $M^1$, together with the set of parameters $\sigma\in\C$ so that every $f\in{\bfnf}_{1,\sigma}$ has a critical value which is prefixed, under some iterate of $f$, to the parabolic fixed point of eigenvalue 1. Thus, $\mathcal{R}^{1}$ is neither open nor closed.

In $\C^2$ consider the following set, the \emph{hyperbolic relatedness locus} $\mathcal{R}$
\begin{itemize}
\item  $\mathcal{R}$ consisting of all $(\lambda,\sigma)\in\C^2$ such that both critical orbits of any $f\in{\bfnf}_{\lambda,\sigma}$ converge to a common attracting fixed point with eigenvalue $\lambda$. Note that $\mathcal{R}=\left \{(\lambda,\sigma): \lambda\in\D \mbox{ and } \sigma\in\mathcal{R}^{\lambda}\right\}=\bigcup_{\lambda\in\D} \{\lambda\} \times \mathcal{R}^{\lambda}=\bigcup_{\lambda\in\D} \{\lambda\} \times \C\setminus M^\lambda$.
\end{itemize}
Since the lines $\per(\lambda)$ do not intersect in the Cantor locus, the set $\mathcal{R}$ is isomorphic to the Cantor locus, and we will use the names synonymously.

The objective of this paper is to explore the relation between the relatedness loci $\mathcal{R}^{\opq}$, pertaining to the parabolic lines $\per(\opq)$, and the boundary of the Cantor locus $\mathcal{R}$. In order to do this, we introduce the concept \emph{dynamical marking} relative to the polynomial
\[\mathrm{P}_{\opq}(z)=\opq z+z^2.\]

\paragraph{Dynamical markings}
We will introduce the notion of \emph{dynamical marking} of a quadratic rational map $g$; a homeomorphic (and to some extent holomorphic) conjugacy between a fixed quadratic polynomial $\mathrm{P}_{\lambda}(z)=\lambda z+z^2$, for some 
$\lambda\in\D^*\cup\bigcup_{p/q\in \Q/\Z}\{\opq\}$, and the given map $g$. This conjugacy is defined on a subset of the attracting or parabolic basin of 0 for the polynomial $\mathrm{P}_{\lambda}$, marked with a point $x$, and such that $x$ and the critical value $-\lambda^2/4$ of $\mathrm{P}_{\lambda}$ are mapped to the critical values of $g$. Our main application will be dynamical marking relative to $\mathrm{P}_{\opq}$. A precise definition will be given in Section \ref{sectdynmarkrelP}.

\subsection{Short summary of results}
Suppose a $p/q$ is fixed, with $(p,q)=1$ and $q\geq 2$.
The concept of dynamical marking gives rise to a partial parametrization of $\mathcal{R}$ by the parabolic basin of 0 for $\mathrm{P}_{\opq}$, see Definition \ref{defParameter}. This parabolic basin constitutes a pre-model for $\mathcal{R}^{\opq}$, a model which records dynamical markings relative to $\mathrm{P}_{\opq}$, thus relating $\mathcal{R}$ to $\mathcal{R}^{\opq}$. From this pre-model, a faithful bijective model can be constructed, see \cite{u10}.

We characterize the limiting behavior of sequences $(\lambda_k,\sigma_k)\in \mathcal{R}$, where $\lambda_k\to\opq$ subhorocyclicly and where $\sigma_k$ has a fixed dynamical marking $x$ relative to $\mathrm{P}_{\opq}$. We use the notion \emph{subhorocyclic} convergence introduced by Epstein \cite{eu}, and say that a sequence $\lambda_{k}\in \D$ converges to $\omega\in\mathbb{S}^{1}$ $\emph{subhorocyclicly}$  if $\Im (\overline{\omega}\lambda _{k}) = o\left( \sqrt{1-|\lambda _{k}|^{2}}\right) $. Geometrically, this condition states that $\lambda_k$ is eventually contained in any $\D$-horodisc at $\omega$. 
We show that every "admissible" $x$ in the parabolic basin of 0 for $\mathrm{P}_{\opq}$ (for the precise definition see Equation (\ref{domainadmissiblex}) and the discussion in Section \ref{results}) is either a dynamical marking of a unique $\sigma \in \mathcal{R}^{\opq}\setminus \mathcal{D}^{\opq}$, which is the limit of a sequence $(\lambda_k,\sigma_k)\in \mathcal{R}$, where $\sigma_k$ has dynamical marking $x$ relative to $\mathrm{P}_{\opq}$ (for details see Theorem \ref{thm:convergence}), or $x$ corresponds to an unbounded sequence $(\lambda_k,\sigma_k)\in \mathcal{R}\subset\C^2$, where $\sigma_k$ converges to $\infty$. In the last case, choosing the normal form $G_{\lambda_k,a_k}(z)=\frac{1}{\lambda_k}(z+a_k+\frac{1}{z})$ representing ${\bfnf}_{\lambda_k,\sigma_k}$, the sequence of $q$th iterates $G_{\lambda_k,a_k}^q$ converges locally uniformly on $\C^*$ to a quadratic rational map $G_T:=G_{1,T}\in {\bfnf}_{1,\sigma}$, with $1-T^2=\sigma \in \mathcal{R}^{1}$, of dynamical marking relative to $\mathrm{P}_{1}$ determined by $x$ (for details, see Theorem \ref{thm:convergenceunbound}). This is sometimes referred to as \emph{rescaling}.  

In the first case, it follows from Theorem \ref{thm:convergence} and results from \cite{u10} (see Proposition \ref{prop:parabolicmarkings}), that every $\sigma \in \mathcal{R}^{\opq}\setminus \mathcal{D}^{\opq}$ is accessible from the Cantor locus along an analytic curve of the same dynamical marking as $\sigma$ (see Corollary \ref{cor:shootinglemma}).

\section{Statement of results}\label{results}
We will now give a more detailed introduction, leading to the formulation of the statements.

\paragraph{The polynomial $\mathrm{P}_{\opq}$.} Consider the polynomial $\mathrm{P}_{\opq}:\CC\to\CC$, given by:
\[\mathrm{P}_{\opq}(z)={\opq} z+z^2.\] 
For the discussion in this paper we consider a fixed $p/q$, with $(p,q)=1$ and $q\geq 2$, and in order to ease the notation we will everywhere suppress the dependence on $p/q$, so that $\omega:=\opq$. In particular, $\mathrm{P}$ will be used in place of $\mathrm{P}_{\opq}$.

$\mathrm{P}$ has a parabolic fixed point at 0 with eigenvalue $\omega$ and a critical point at $-\omega/2$, with critical value $-\omega^2/4$. Let $A$ denote the parabolic basin of 0. The parabolic basin $A$ has a $q$-cycle of components, whose closures meet only at the parabolic fixed point 0. These $q$ components perform a $p/q$ rotation around the fixed point 0, under iteration by $\mathrm{P}$. They will be indexed counterclockwise with respect to their cyclic ordering $B^0,..., B^{q-1}$, where $B^0$ is the component that contains the critical point $-\omega/2$. See Figure \ref{Fatrabbit}.

 Let $\phi:A\to\C$ denote an extended Fatou coordinate, which is a surjective and holomorphic map that satisfies the Abel functional equation $\phi\circ\mathrm{P}^q=1+\phi$. It is patched together between components $B^j$ to satisfy the functional equation:
 \begin{equation}  
 	\phi\circ\mathrm{P}=1/q+\phi.
 \end{equation}
 The Fatou coordinate will be further normalized by setting $\phi(-\omega/2)=0$.

Each $B^j$ contains a unique petal $\mathcal{P}_\Delta^j\subseteq B^j$ for $\mathrm{P}^q$ at 0, which is mapped univalently by $\phi$  onto some right half-plane $\H_\Delta=\{z\in\C \: | \: \Re(z) >\Delta\}$ and which is maximal in this sense, whence it has a critical point of $\mathrm{P}^q$ on its boundary. Such a petal is called a maximal attracting petal. In general, denote by $\mathcal{P}_\delta^j\subseteq B^j$ an attracting petal for $\mathrm{P}^q$ at 0, whose (univalent) image under $\phi$ is precisely the right half-plane $\H_\delta=\{z\in\C \: | \: \Re(z) >\delta\}$. Note that by our choice of normalization $\mathcal{P}_\Delta^0=\mathcal{P}_0^0$. We will say that a petal $\mathcal{P}_\delta^j$ is of level $\delta$ and call the union $\mathcal{P}_\delta= \bigcup_{j\in \Z_q} \mathcal{P}_\delta^j$ an \emph{attracting flower} for $\mathrm{P}$ at 0 of level $\delta$. Let $\widehat{\mathcal{P}}_{\delta}=\mathcal{P}_{\delta}\cup \{0\}$.

The set $\overline A\setminus \{0\}$ has $q$ connected components. Let $S^j$ denote the component so that $B^j\subset S^j$, whence $S^0$ is the component containing the critical point $-\omega/2$ and $S^p$ the component containing the critical value $-\omega^2/4$.

\begin{figure}
\begin{center}
\begin{tabular}{cc}
    \includegraphics[width=5.5cm]{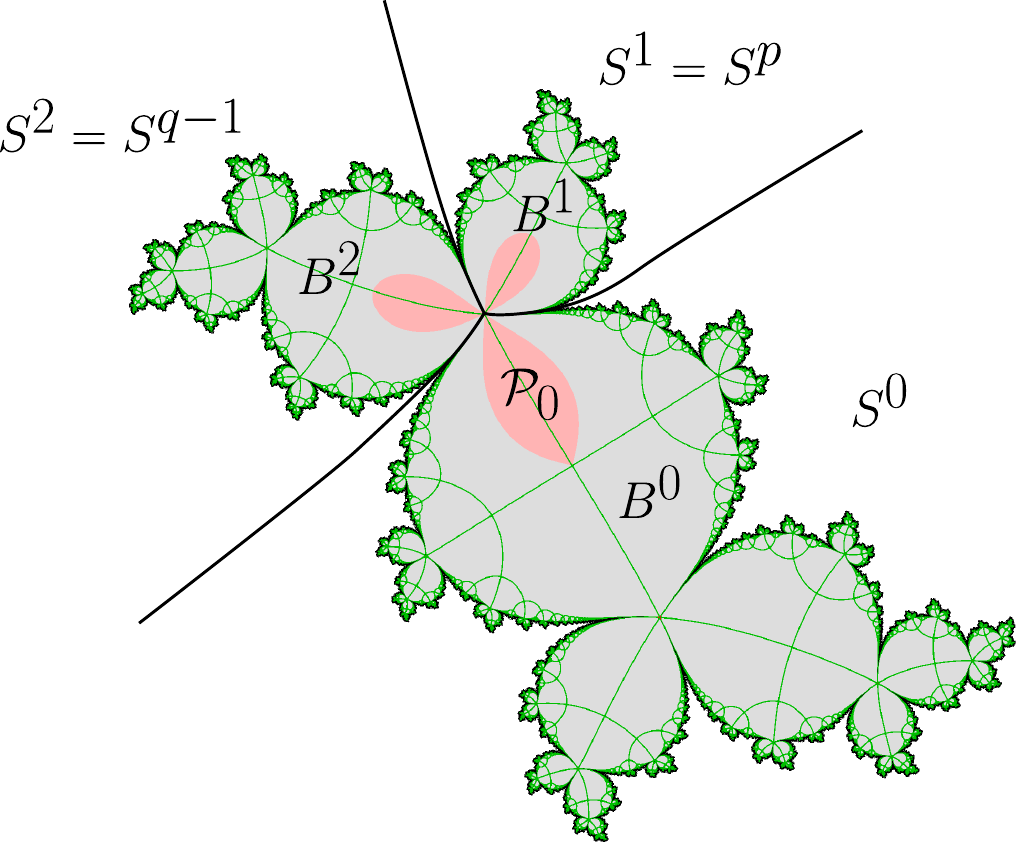}
   \includegraphics[width=5cm]{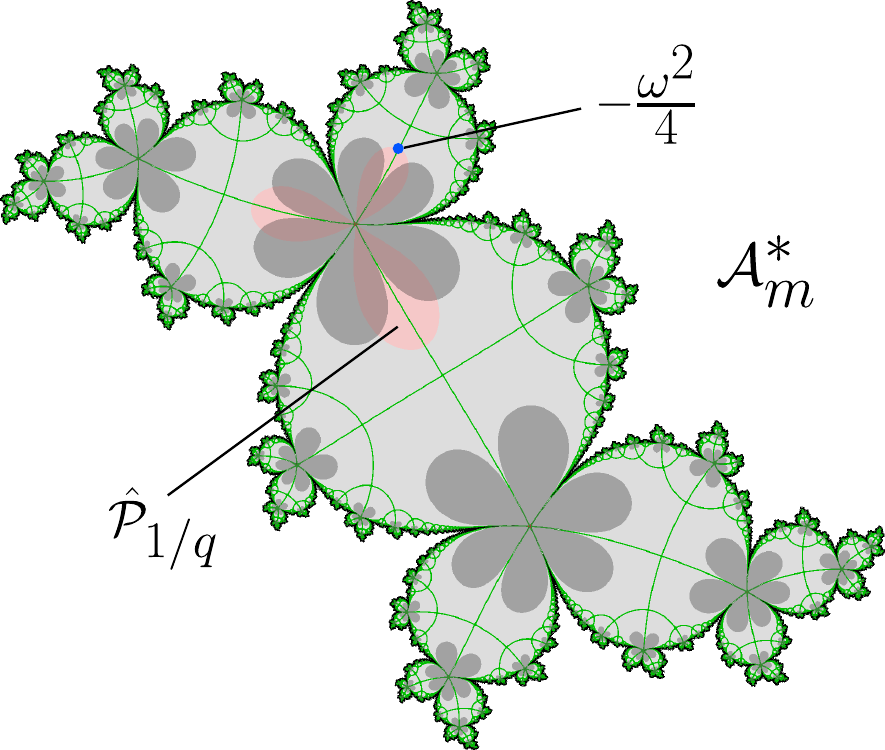} 
\end{tabular}
\caption{The left picture shows $A$, the attracting basin of 0, for $\mathrm{P}=\mathrm{P}_{1/3}$. The attracting flower $\mathcal{P}_{0}$ is shown in pink. 
The figure also illustrates components $B^j$ and sectors $S^j$. The right picture shows $\mathcal{A}_m^*$, with the removed sepals shown in dark grey, and the removed (extended) flower $\widehat{\mathcal{P}}_{1/q}$ shown in pink. Dynamical pictures are made by Arnaud Ch{\'e}ritat.}
\label{Fatrabbit}
\end{center}

\end{figure}

Let
\begin{equation}
	\mathcal{A}=A \cup \bigcup_{n\geq 0}\mathrm{P}^{-n}(0)\subset \overline A, 
\end{equation}
and for $m>0$ let
\begin{equation}\label{def:Xim}
 \mathcal{A}_{m}=\mathcal{A}(m)=\left\{z\in A:   |\Im(\phi(z))|<m/2\right\} \cup \bigcup_{n\geq 0}\mathrm{P}^{-n}(0)\subset \mathcal{A}.  
\end{equation}
We will also allow $m=\infty$, in which case $\mathcal{A}_{m}=\mathcal{A}$. 
The set $\mathcal{A}$ is the parabolic basin, augmented with 0 and all its pre-images, and $\mathcal{A}_{m}$ has all so-called \emph{sepals} of height $m/2$, and all their pre-images, removed.

Recall that $\mathcal{P}_{1/q}$ denotes the attracting flower for $\mathrm{P}$ at 0 of level $1/q$, and let  
\begin{equation}\label{domainadmissiblex}
	\mathcal{A}^*=\mathcal{A}\setminus (\widehat{\mathcal{P}}_{1/q}\cup\{-\omega^2/4\}) \quad \text{and} 
\end{equation}
\begin{equation}
 \mathcal{A}^*_{m}=\mathcal{A}_m\setminus (\widehat{\mathcal{P}}_{1/q}\cup\{-\omega^2/4\}). 
\end{equation}
See Figure \ref{Fatrabbit}.
We are now ready to define dynamical marking relative to $\mathrm{P}$.

\subsection{Dynamical marking relative to $\mathrm{P}$}\label{sectdynmarkrelP}
For $x\in \mathcal{A}^*$, let $m>0$ be such that $x\in \mathcal{A}_m^*$ and define 
\[\mathcal{U}=\mathcal{U}_m(x):=\mathrm{P}^{-n}(\widehat{\mathcal{P}}_{0})\cap \mathcal{A}_{m},\]for $n=n(x)\geq 0$ minimal so that $\mathrm{P}^{n}(x)\in \widehat{\mathcal{P}}_{0}$. 

Note that $\mathcal{U}_m(x)$ is forward invariant and connected, but neither open nor closed, and $0, \frac{-\omega^2}{4}, x \in \mathcal{U}_m(x)$, while $\mathrm{P}^{-1}(x) \cap \mathcal{U}_m(x)=\emptyset$. To alleviate the notation, we use $\mathcal{U}=\mathcal{U}_m(x)$ if both $x$ and $m$ are clear from the context. 

We define slit planes $\C_{-\omega}=\C\setminus \{-\omega \cdot r \; | \; r\geq 0\}$ and disks $\D_{-\omega}=\D\setminus [0,-\omega)$. For $\lambda\in \D_{-\omega}$ let $L=q\Log(\lambda\bar\omega)=q\log\lambda-p2\pi i$, where $\Log$ is the principal logarithm, whence $\log$ is the branch of the logarithm from the cut plane $\C_{-\omega}$ to the strip $\{z\in\C \, |  \; 2\pi p/q - \pi < \Im(z) < 2\pi p/q + \pi\}$.
For $\lambda\in \D_{-\omega}\cup \{\omega\}$ let

\begin{equation}\label{eqn_mlambdaattrpar}
    m(\lambda)=\begin{cases}
\frac{2\pi\sin\theta}{q|L|} & \text{ for } \lambda\neq \omega\\
\infty  & \text{ for } \lambda= \omega,
\end{cases}
\end{equation}
where $\theta$ is the angle between $L$ and $2\pi i$. For $\lambda\neq \omega$ the number $r_\lambda=\frac{\pi}{q^2m(\lambda)}$ is the radius of the circle through $\log\lambda$, and tangent to the imaginary axis at $2\pi ip/q$, see also Section \ref{sec:star}.
\begin{definition}[Dynamical marking relative to $\mathrm{P}$]\label{def:dynmarkingpara}
Let $\lambda\in \D_{-\omega}\cup \{\omega\}$ and $m=m(\lambda)$, let $\sigma \in \mathcal{R}^\lambda$ and let $g\in \Gamma_{\lambda,\sigma}$ with a fixed point at $z_0$ of eigenvalue $\lambda$. A \emph{dynamical marking} of $g$ \emph{relative to} $\mathrm{P}$ (written rel $\mathrm{P}$) is a pair $(x,\psi)$ where $x\in \mathcal{A}^*_{m(\lambda)}$ and $\psi: \mathcal{U}\to V$ is a 
homeomorphism obeying 
\begin{equation}
   \psi\circ \mathrm{P}=g\circ \psi, 
\end{equation} 
which is holomorphic on $\inter(\mathcal{U})$, where $\psi(-\frac{\omega^2}{4})$ and $\psi(x)$ are the critical values of $g$, and where $\mathcal{U}=\mathcal{U}_{m(\lambda)}(x)$ as defined above.
\end{definition}
\begin{figure}[htbp]
  \begin{center}
    \includegraphics[width=11cm]{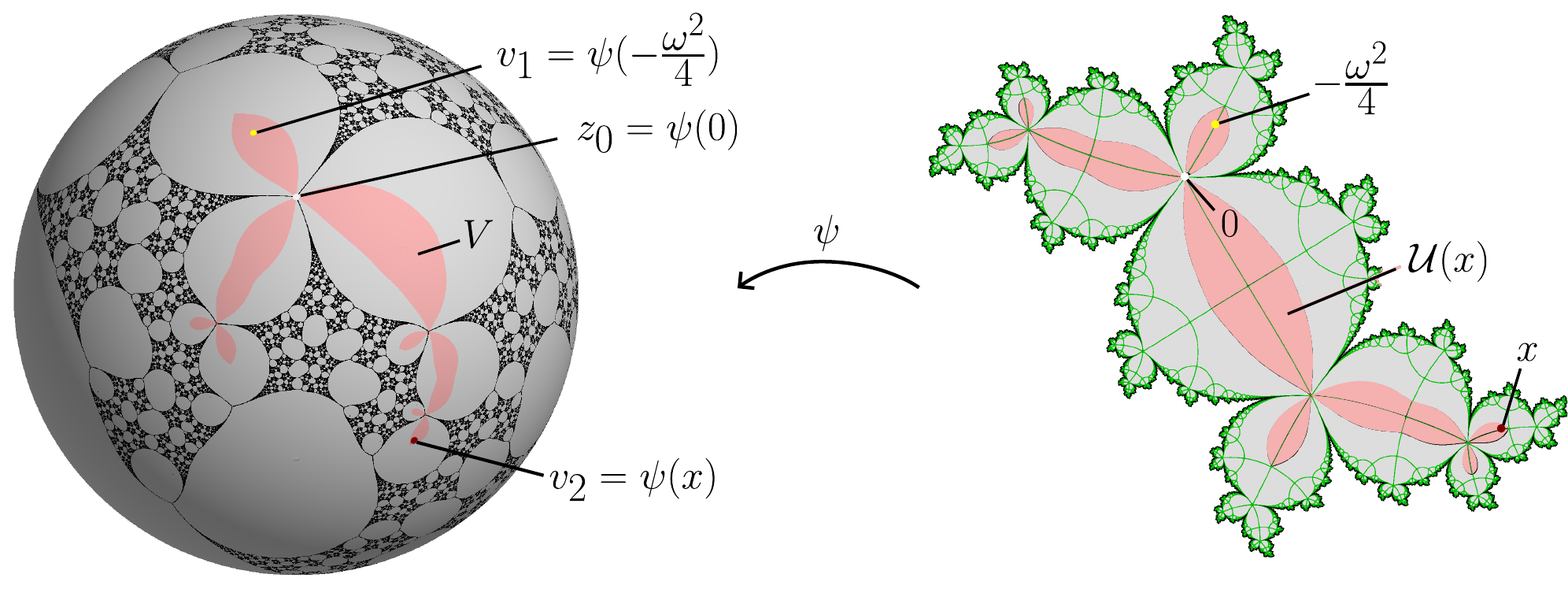}
  \end{center}
  \caption{\label{fig:illustrationdynmarking} Illustrates dynamical marking rel $\mathrm{P}$ for $p/q=1/3$. The left shows the dynamical plane of a $g\in \Gamma_{\omega_{1/3},\sigma}$ (on the Riemann sphere), and the filled-in Julia set of $\mathrm{P}$ is on the right. The figure shows the domain $\mathcal{U}(x)$ and range $V$ of $\psi$ in pink, the marking $x\in\mathcal{A}^*$ of $g$ and the two critical values $v_1$ and $v_2$ of $g$. Dynamical pictures are made by Arnaud Ch{\'e}ritat.}
\end{figure}

When $\mathrm{P}$ is understood from the context, we will omit "rel $\mathrm{P}$". We will say that $g$ is \emph{dynamically marked} by $(x,\psi)$ if $(x,\psi)$ is a dynamical marking of $g$. We will say that $g$ is dynamically marked by $x$, if there exists a \emph{marker map} $\psi$, so that $g$ is dynamically marked by $(x,\psi)$. A dynamical marking induces a labeling  $v_1=\psi(-\frac{\omega^2}{4})$ and $v_2=\psi(x)$ of critical values of $g$. We will refer to this labeling as a dynamical labeling of $g$.

For $\lambda\in \D_{-\omega}\cup \{\omega\}$ and $\sigma \in \mathcal{R}^\lambda$, if $g\in \Gamma_{\lambda,\sigma}$ is dynamically marked by $(x,\psi)$ rel $\mathrm{P}$, then $x$ is a conformal invariant of $\Gamma_{\lambda,\sigma}$. Moreover, we will show that $\sigma$ is the unique parameter in $\mathcal{R}^\lambda$, which is dynamically marked by $x$, and that the dynamical marker map $\psi$ is unique up to automorphisms of $g$, see Proposition \ref{prop:markinglambdarelP} (and its proof in Section \ref{markingRbyP}) and Proposition \ref{thm:injective1}
in Section \ref{sectdynmarkrelPonRomega}. In view of these properties, we will speak of $x$ as a dynamical marking of the equivalence class $\Gamma_{\lambda,\sigma}$, or of its parameter $\sigma\in \mathcal{R}^\lambda$.

\subsection{Parametrization of $\mathcal{R}$ by dynamical marking rel $\mathrm{P}$}
Consider an $M>1/q^2$ and let 
\[D_M=\{\lambda\in\D_{-\omega}: m(\lambda)> M\}.\]
$D_M$ is the $\D^*$ horodisk $\exp(\D(-r+2\pi ip/q,r))$ at $\omega$, where $r=\frac{\pi}{q^2M}<\pi$, see Figure \ref{DMrho}. 
\begin{figure}
\begin{center}
   \includegraphics[width=11cm]{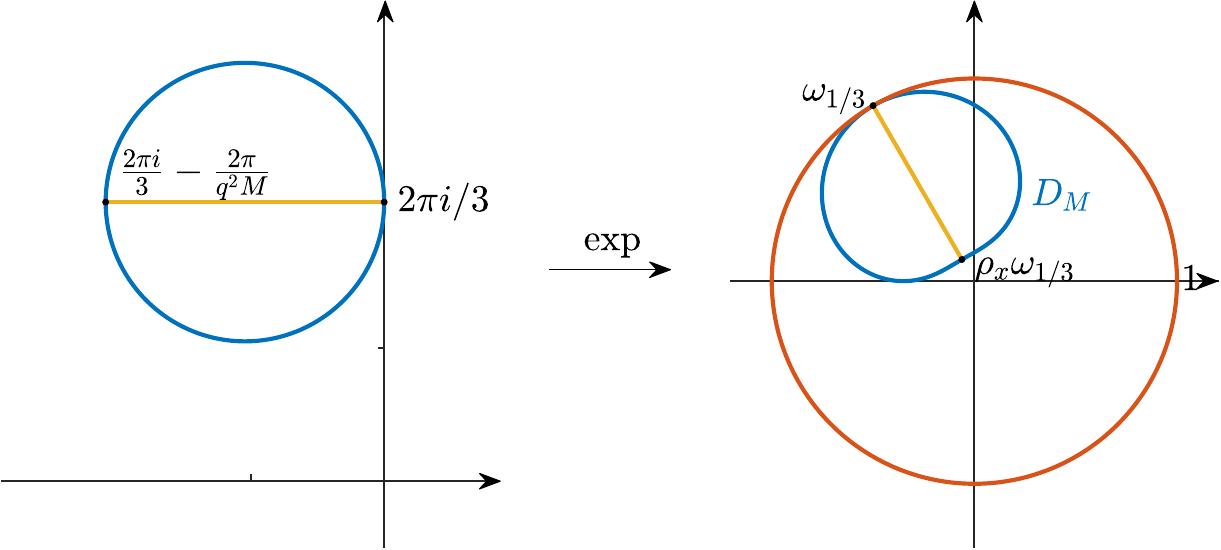} 
\caption{The figure illustrates $\D(-r+2\pi ip/q,r)$, $D_M$ and $\rho_x$ for $p/q=1/3$ and $M=1/3$.}
\label{DMrho}
\end{center}
\end{figure}

 Suppose $\lambda\in \D_{-\omega}$ is given. If $x$ is a dynamical marking of a $\sigma\in \mathcal{R}^{\lambda}$, then $x$ is necessarily in $\mathcal{A}_{m(\lambda)}^*$, thus $\mathcal{A}_{m(\lambda)}^*$ is the set of admissible $x$-values for $\lambda$. On the other hand, suppose $x \in \mathcal{A}^*$ is fixed and let $M >\max\{1/q^2, 2|\Im(\phi(x))|\}$, where it is understood (whenever it appears) that the maximum is equal to $\frac{1}{q^2}$ when $x$ is not in the domain of $\phi$. Then $x \in \mathcal{A}_{M}^*$ 
 and the corresponding horodisk $D_M$ consists of admissible values of $\lambda$.
 Fixing an $M$ determines a domain $D_M\times \mathcal{A}_{M}^*$ of admissible pairs $(\lambda,x)$ for dynamical markings in $\mathcal{R}$, since $(\lambda,x)\in D_M\times \mathcal{A}_{M}^* \Rightarrow x\in \mathcal{A}^*_M\subset \mathcal{A}_{m(\lambda)}^*$.

Not every $\sigma \in \mathcal{R}^{\lambda}$ is dynamically marked relative to $\mathrm{P}$ (see Section \ref{markingRbyP}), 
but every $x \in \mathcal{A}_{M}^*$ is the dynamical marking of exactly one $\sigma\in \mathcal{R}^{\lambda}$, for each $\lambda\in D_M$. 

\begin{prop}\label{prop:markinglambdarelP}
	Suppose $M>1/q^2$. For any 
    $(\lambda,x)\in D_M\times \mathcal{A}_{M}^*$, there exists a unique parameter $\sigma\in \mathcal{R}^{\lambda}$ such that $\Gamma_{\lambda,\sigma}$ is dynamically marked by $x$ rel $\mathrm{P}$. 
    In particular, this defines a partial parametrization by dynamical marking, $\Phi^\lambda: \mathcal{A}^*_{M} \to \mathcal{R}^{\lambda}$. The map $\Phi^\lambda$ is holomorphic in the interior of its domain, continuous when restricted to $\mathrm{P}^{-n}(\widehat{\mathcal{P}}_{0})\cap \mathcal{A}^*_{M}$ for every $n$, but not in the entire domain $\mathcal{A}^*_{M}$, and depends analytically on the parameter $\lambda\in D_M$.
    
     Moreover, if
	\begin{itemize}
		\item $x\notin \partial\mathcal{P}_{1/q}$ then $x$ is the unique marking of $\Gamma_{\lambda,\sigma}$.
		\item if $x\in \partial\mathcal{P}_{1/q}^{p+j}$ then $x'\in \partial\mathcal{P}_{1/q}^{p-j}$ with $\Im(\phi(x'))=-\Im(\phi(x))$, both $p+j$ and $p-j$ taken modulo $q$, is the only other marking of $\Gamma_{\lambda,\sigma}$.
	\end{itemize}
\end{prop}
The proposition will be proved in Section \ref{markingRbyP}. We will use this result to define a partial parameter $\mathbf{\Phi}$ on $\mathcal{R}$, which parametrizes by dynamical marking relative to $\mathrm{P}$.

\begin{definition}[Definition of a partial parameter $\mathbf{\Phi}$ on $\mathcal{R}$] \label{defParameter}
For every $\pq$, with $(p,q)=1$ and $q\geq 2$, and for $M>\frac{1}{q^2}$ we define  a map
\[\mathbf{\Phi}:D_M\times \mathcal{A}_{M}^*\to (D_M\times \C)\cap \mathcal{R} \quad \text{of the form}\]
\[\mathbf{\Phi}(\lambda,x)=(\lambda, \Phi(\lambda,x))=(\lambda, \Phi^\lambda(x))=(\lambda, \Phi_x(\lambda)),\]
by letting $\Phi^\lambda: \mathcal{A}^*_{M} \to \mathcal{R}^{\lambda}$ be the parametrization from Proposition \ref{prop:markinglambdarelP}.
That is, $\Phi^\lambda(x)=\sigma$, where $\sigma$ is the unique parameter in $\mathcal{R}^{\lambda}$ so that $\Gamma_{\lambda,\sigma}$ is dynamically marked by $x$ rel  $\mathrm{P}$. 
\end{definition}
From Proposition \ref{prop:markinglambdarelP} it follows that $\Phi_x$ is analytic in $D_M$ for every $x\in \mathcal{A}^*_{M}$, we say that $\mathbf{\Phi}$ is \emph{horizontally analytic}. For a fixed $x\in \mathcal{A}^*_{M}$, and for a curve $\delta\subset D_M$ or a sequence $(\lambda_k)_k\subset D_M$, we will call $\Phi_x(\delta)$ an $x$-horizontal curve and $(\Phi_x(\lambda_k))_k$ an $x$-horizontal sequence. Every parameter in such an $x$-horizontal curve or sequence has dynamical marking $x$. For any $x \in \mathcal{A}^*$ if we let $M>\max\{1/q^2, 2|\Im(\phi(x))|\}$ then $x \in \mathcal{A}_{M}^*$. 
Thus, every $x \in \mathcal{A}^*$ and every sequence $(\lambda_k)_k\subset D_M$ give rise to sequences $(\lambda_k,\sigma_k)=(\lambda_k,\Phi_{x}(\lambda_k))\in D_M\times\C$, where each $\sigma_k\in \mathcal{R}^{\lambda_k}$ and $(\sigma_k)_k$ is an $x$-horizontal sequence. 

\subsection*{Main result 1: Convergence to the parabolic locus $\mathcal{R}^{\omega}$}
The following Theorem is a main result of this paper. It is concerned with the limiting behavior of such $x$-horizontal sequences $(\lambda_k,\sigma_k)$ in $\mathcal{R}\subset\C^2$ as $\lambda_k\to\omega$ subhorocyclicly.

\begin{etheorem}\label{thm:convergence}
For $x \in \mathcal{A}^*\setminus S^p$, let $M>\max\{1/q^2, 2|\Im(\phi(x))|\}$ and let $(\lambda_k)_k\subset D_M$ be a sequence converging to $\omega=\opq$ subhorocyclicly. Then the $x$-horizontal sequence $(\sigma_k)_k=(\Phi_x(\lambda_k))_k$ 
converges to the unique $\sigma\in\mathcal{R}^\omega\setminus\mathcal{D}^\omega$, which is dynamically marked by $x$ rel $\mathrm{P}$. In particular, every $x \in \mathcal{A}^*\setminus S^p$ is a dynamical marking for exactly one $\sigma\in\mathcal{R}^{\omega}\setminus\mathcal{D}^\omega$.\end{etheorem}

In the other direction, dynamical marker maps for parameters $\sigma\in\mathcal{R}^\omega\setminus \mathcal{D}^\omega$ can be constructed using Fatou coordinates.
This is done in \cite{u10} (although it is not called dynamical markings). 
\cite[Lemma 3.3]{u10} shows that no $\sigma \in \mathcal{R}^\omega\setminus \mathcal{D}^\omega$ is dynamically marked, rel $\mathrm{P}$, by an $x\in S^p$. For completeness, the next proposition formulates the results from \cite{u10} in the context of dynamical markings, and we recall the construction of dynamical marker maps for $g\in\Gamma_{\omega,\sigma}$ with $\sigma\in\mathcal{R}^\omega\setminus \mathcal{D}^\omega$ in its proof in Section \ref{sectdynmarkrelPonRomega}.

\begin{prop}\label{prop:parabolicmarkings}
	Any $\sigma\in \mathcal{R}^\omega\setminus \mathcal{D}^\omega$ is dynamically marked rel $\mathrm{P}$ by an $x\in \mathcal{A}^*\setminus S^p$.    
         If $x\notin \partial\mathcal{P}_{1/q}$ then $x$ is the unique marking of $\sigma$. On the other hand if $x\in \partial\mathcal{P}_{1/q}^{p+j}$ then $x'\in \partial\mathcal{P}_{1/q}^{p-j}$ with $\Im(\phi(x'))=-\Im(\phi(x))$, both $p+j$ and $p-j$ taken modulo $q$, is the only other marking of $\sigma$.
\end{prop}
Proposition \ref{prop:parabolicmarkings} and Theorem \ref{thm:convergence} together show that dynamical marking rel $\mathrm{P}$ gives rise to a surjective and almost injective parameter from $\mathcal{A}^*\setminus S^p$ onto $\mathcal{R}^\omega\setminus \mathcal{D}^\omega$, and they imply the following corollary. 

For $x\in \mathcal{A}^*\setminus S^p$ let $M=\max\{1/q^2, 2|\Im(\phi(x))|\}$ and let 
\[\rho_x=e^\frac{-2\pi}{q^2 M},\]
it is the "diameter" of the $\D^*$ horodisk $D_M$, see Figure \ref{DMrho}.

\begin{cor}\label{cor:shootinglemma}
Every $\sigma \in \mathcal{R}^\omega\setminus \mathcal{D}^\omega$ is accessible from the Cantor locus $\mathcal{R}$. Let $x$ be a dynamical marking rel $\mathrm{P}$ of $\sigma$. Then the $x$-horizontal analytic curve $\delta_x=(\Phi_x: \omega]\rho_x,1[\to\C)$ \emph{lands} at $\sigma$, that is $\delta_x$ extends continuously to 1 with $\delta_x(1)=\sigma$. 
\end{cor}

\subsection{Main result 2: Divergence and rescaling}\label{mainresult2rescaling}
As described above, no $\sigma \in \mathcal{R}^\omega\setminus \mathcal{D}^\omega$ is dynamically marked rel $\mathrm{P}$ by an $x\in S^p$, so what happens for sequences $(\lambda_k,\sigma_k=\Phi_x(\lambda_k))_k$, for $x\in S^p$, as $\lambda_k$ tends to $\opq$? That is the topic of the next theorem, but first we need a little more background.

The definition of dynamical marking relative to $\mathrm{P}_{\omega}$ also makes sense for the case $\omega=1$. We let $A_1$ denote the parabolic basin of 0 for $\mathrm{P}_1$, and let $\phi_1$ denote the extended Fatou coordinate, normalized so that $\phi_1(-1/4)=1$. Petals $\mathcal{P}_\delta$ are defined in the same way as for $q\geq 2$, and we let   
\begin{equation}
 \widetilde{A}_1=A_1 \cup \bigcup_{n\geq 0}\mathrm{P}_1^{-n}(0)  
\end{equation}
and
\begin{equation}
 A_1^*=\widetilde{A}_1\setminus (\mathcal{P}_{1}\cup\{0,-1/4\}). 
\end{equation}
Dynamical markings $x\in A_1^*$, the corresponding marker maps $\psi$, and their domains are defined similarly to the case $q\geq 2$.

Recall that the set $\overline{A}\setminus \{0\}$ (for $\mathrm{P}$) has $q$ connected components and that $S^0$ is the component containing the critical point $-\omega/2$, see again Figure \ref{Fatrabbit}.
\begin{figure}
\begin{center}
   \includegraphics[width=11cm]{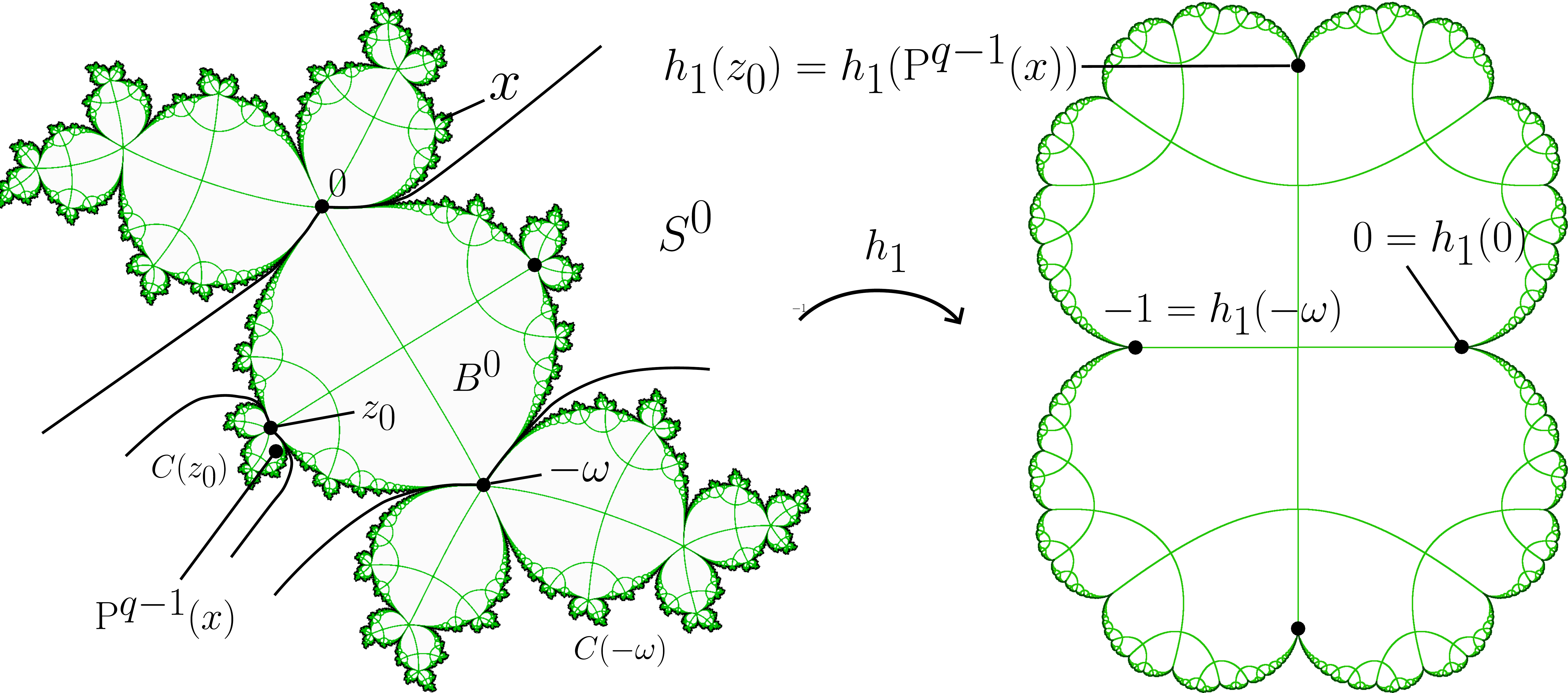} 
\caption{The figure illustrates the map $h_1$ (for $p/q=1/3$), with the augmented basin $\mathcal{A}$ of $\mathrm{P}$ on the left and the augmented basin $\widetilde{A}_1$ of $\mathrm{P}_1$ on the right (to the level $n=2$ for $\widetilde{A}_1$). $h_1$ maps the sets $C(z_0)$ to $h_1(z_0)$. The figure also illustrates $h_1(\mathrm{P}^{q-1}(x))$, cf. Theorem \ref{thm:convergenceunbound}. Dynamical pictures by Arnaud Ch{\'e}ritat.}
\label{fig_h1}
\end{center}
\end{figure}
\paragraph{Connecting $\mathrm{P}^q$ on $\overline{S^0}\cap \mathcal{A}$ to $\mathrm{P}_{1}$ on $\widetilde{A}_1$}
We introduce a surjective, continuous map $h_1: \overline{S^0}\cap \mathcal{A} \to \widetilde{A}_1$, with $h_1(0)=0$ and $h_1(-\omega/2)=-1/2$, such that 
\begin{equation}\label{functeqnh1}
h_1\circ\mathrm{P}^q = \mathrm{P}_{1}\circ h_1 \quad \text{on } \overline{B^0}\cap \mathcal{A}.
\end{equation}
On $B^0$ it is the biholomorphic conjugacy $h_1:B^0\to \widetilde{A}_1$ of $\mathrm{P}^q$ to $\mathrm{P}_1$, constructed from Fatou coordinates for $\mathrm{P}^q$ and $\mathrm{P}_{1}$, such that $h_1(-\omega/2)=-1/2$. It is extended to the fixed point 0 and all its pre-images by continuity, so that, in particular, $h_1(0)=0$. This defines $h_1$ as a homeomorphic conjugacy from $\overline{B^0}\cap \mathcal{A}$ onto $\widetilde{A}_1$. The next step is to  extend $h_1$ to the full domain. Let $z_0$ denote an iterated pre-image of 0 in $\overline{B^0}$, every such pre-image of 0 separates $S^0$ into $q$ connected components, let $C(z_0)$ denote the union of components not containing $-\omega/2$, and define $h_1(z)=h_1(z_0)$ for all $z\in C(z_0)$, see Figure \ref{fig_h1}. Thus, informally speaking, $h_1$ "collapses" everything attached to $B^0$ at $z_0$ to the image of $z_0$. With this extension $h_1$ is continuous and surjective, but of course no longer a bijection, and it does not obey the functional equation (\ref{functeqnh1}) everywhere. 
\\ \\

For $\omega=1$ we use the normal form $G_T(z):=G_{1,T}(z)=z+T+\frac{1}{z}$. The next Theorem is the analogue of Theorem \ref{thm:convergence} for $x\in S^p$.
\begin{etheorem}\label{thm:convergenceunbound}
For every $x \in \mathcal{A}^*\cap S^p$ let $M>\max\{1/q^2, 2|\Im(\phi(x))|\}$ and let $(\lambda_k)_k\subset D_M$ be a sequence converging to $\omega=\opq$ subhorocyclicly. Then the $x$-horizontal sequence $\sigma_k=\Phi_x(\lambda_k)$ tends to $\infty\in \CC$. 

Moreover, for representatives $G_{\lambda_k, a_k} \in {\bfnf}_{\lambda_k,\sigma_k}$ with dynamical labeling $\psi_k\left(-\frac{\lambda_k^2}{4}\right)=G_{\lambda_k, a_k}(1)$ and $\psi_k(x)=G_{\lambda_k, a_k}(-1)$ rel $\mathrm{P}$, it holds that

\begin{itemize}
\item $G_{\lambda_k, a_k}\to \infty$ locally uniformly on $\C^*$
\item $G_{\lambda_k, a_k}^q\to G_T$ locally uniformly on $\C^*$, where $T\in \C$, $\sigma=1-T^2 \in\mathcal{R}^{1}$ and $\sigma$ is dynamically marked by $h_1(\mathrm{P}^{q-1}(x))$ rel $\mathrm{P}_1$. 

In particular, if $x\in B^p$, then both critical points $\pm1$ of $G_T$ are in the same component of the parabolic basin of $\infty$ and $\sigma\in \inter(\mathcal{R}^{1}) =\C\setminus M^1$. And if $x\in S^p\setminus B^p$ then there exists $n\geq 1$ so that $G^n_T(-1)=0$ and $\sigma\in\mathcal{R}^{1}\cap M^1$.
 \end{itemize}
\end{etheorem}

There are three main tools for the proofs. The first is the behavior of sequences that diverge to infinity in ${\bf rat}_2$, which is treated in \cite{eps00}. We give a quick distillation of the results needed for our use in Section \ref{sect:divergenceinfty} and apply them to show new results concerning the Cantor locus in Section \ref{sectMainProofs}. The largely algebraic conditions in \cite{eps00} are refined to our case in Lemma \ref{lemma:dynamicsinfinity} (in Section \ref{sectMainProofs}), using the geometric control coming from the dynamical markings. 

The second is a construction called \emph{stars} in attracting basins, introduced in \cite{pet99}, and treated in Section \ref{sec:star}. We define the construction in Section \ref{sect_stardef}, and use it to introduce dynamical markings rel $\mathrm{P}_\lambda$ (for $\lambda\in\D_{-\omega}$) in Section \ref{sec:dynmarkinglambda}. 
The third tool is a model for $\mathcal{R}^{\lambda}$, which is modified from a model introduced in \cite{gold90}. Our formulation of the model is based on stars and dynamical markings, it is given in Section \ref{chap:gkmodel}. 

In Sections \ref{sect:landingwire} and \ref{sect:modulus} we treat stars in non-simply connected quadratic basins; we show some new results and extend a main result from \cite{pet99} to this case. In Section \ref{sectdynmarkrelPproofs} we prove Proposition \ref{prop:markinglambdarelP} and Proposition \ref{prop:parabolicmarkings}, and show uniqueness of dynamical markings rel $\mathrm{P}$. The paper ends with the proofs of Theorems \ref{thm:convergence} and \ref{thm:convergenceunbound} in Section \ref{sectMainProofs}.

\section{Divergence to infinity in ${\bf rat}_2$}\label{sect:divergenceinfty}

\subsection{Reminder on normal forms} 
For polynomials with an attracting or parabolic fixed point we use 
\[\mathrm{P}_{\lambda}(z)=\lambda z+z^2, \quad \lambda\in\D^*\cup \bigcup_{p/q}\{\opq\}.\] 
$\mathrm{P}_{\lambda}$ has an attracting or parabolic fixed point at 0 with eigenvalue $\lambda$ and a critical point at $\frac{-\lambda}{2}$, with critical value $\frac{-\lambda^2}{4}$.

Every quadratic rational map has three fixed points, counting multiplicity, and two critical points with distinct critical values. Recall that each class ${\bfnf}_{\lambda,\sigma}$ in $\per(\lambda)$ has a representative of the form 
\begin{equation}
G_{\lambda,a}(z)=\frac{1}{\lambda}\left(z+a+\frac{1}{z}\right),
\end{equation}
as long as $\lambda\neq 0$. Maps $G_{\lambda,a}$ have critical points $\pm1$ and a fixed point at $\infty$ with eigenvalue $\lambda$. Note that $[G_{\lambda,a}]=[G_{\lambda,-a}]= {\bfnf}_{\lambda,\sigma}$ where 
\begin{equation}
\sigma=\frac{(\lambda-2)^2-a^2}{\lambda^2}.
\end{equation}
For the case $\lambda=1$ we use the notation:
\begin{equation}
G_T(z):=G_{1,T}(z)=z+T+\frac{1}{z}, \quad G_T \in {\bfnf}_{1,\sigma},
\end{equation}
where $\sigma=1-T^2$. Maps $G_{T}$ have critical points $\pm1$ and a double fixed point at $\infty$ with eigenvalue 1.

\subsection{Dynamics near infinity in ${\bf rat}_2$}
The moduli space ${\bf rat}_2$ contains infinitely many unbounded and infinitely many bounded components. Every hyperbolic component of maps with an attracting fixed point is unbounded. Epstein shows \cite{eps00} that every hyperbolic component of maps with two distinct attracting cycles has compact closure in ${\bf rat}_2$ if and only if neither attractor is a fixed point. One of the tools for this is a detailed understanding of those unbounded sequences in ${\bf rat}_2$, which have a fixed point eigenvalue tending to a $q$th root of unity. \cite{eps00} studies sequences of maps $G_{\lambda_k,a_k}(z)=\frac{1}{\lambda_k}(z+a_k+\frac{1}{z})$, so that  $\lambda_k$ tends to a $q$th root of unity and $[G_{\lambda_k,a_k}]$ diverges to infinity in ${\bf rat}_2$. In particular \cite{eps00} gives algebraic conditions under which the sequence of $q$th iterates of such a sequence converges locally uniformly on $\C^*$ to a quadratic rational map. We refer to this as \emph{rescaling}.

We summarize here some results from \cite{eps00} and \cite{mil93} concerning the dynamical behavior of sequences that diverge to infinity in ${\bf rat}_2$. For the formulation of Lemma \ref{BHClemma}(\cite{eps00}) given here, see also \cite{kabelka} and \cite{eu}. 

A sequence is bounded in ${\bf rat}_2$ if and only if the corresponding sequences of the eigenvalues of the fixed points $\lambda_k, \mu_k, \nu_k$ are bounded in $\C$. The well-known holomorphic index formula, which for $\lambda, \mu \neq 1$ can be written as  
\[\nu=\frac{2-\lambda-\mu}{1-\lambda\mu},\]
shows that if $\lambda_k, \mu_k$ tend to $\lambda, \frac{1}{\lambda}$, for some $\lambda\in \CC\setminus \{1\}$, then $\sigma_k=\mu_k\nu_k\to\infty$.

For the other direction, any unbounded sequence in ${\bf rat}_2$ has a subsequence, such that the fixed point eigenvalues tend to $\lambda, \frac{1}{\lambda}, \infty$, where $\lambda\in \CC$. Here we consider sequences ${\bfnf}_{\lambda_k,\sigma_k}$ where $\lambda_k\to\opq$.
\begin{lemma}[\cite{eps00}]\label{BHClemma}
Let $f_k$ be a sequence of quadratic rational maps, with fixed point eigenvalues $\lambda_k, \mu_k, \nu_k$ tending to $\opq,\bar\omega_{\pq},\infty$, $q\geq 2$. If $f_k= G_{\lambda_k, a_k}$ then
\begin{equation}\nonumber
f_k^i \to \infty \quad \text{ locally uniformly on } \C^* \text{ for } 1\leq i<q. 
\end{equation}
If further $\frac{1-\lambda_k^q}{1-\lambda_k}\cdot a_k\to T\in\CC$, then 
\begin{equation}\nonumber
f_k^q \to G_T   \quad \text{ locally uniformly on } \C^*, 
\end{equation}
where we use the convention $G_\infty=\infty$. Moreover, if $T\in \C$, $\lambda_k\neq \opq$ and $\mu_k\neq \bar\omega_{\pq}$ then there are $q$-cycles $\langle z \rangle_k$ with eigenvalues $\rho_k\to 1-T^2$, while for any other $q$-cycles $\langle \widehat z \rangle_k\neq\langle z \rangle_k$ the eigenvalues tend to infinity.

\end{lemma}

\section{Quadratic attracting basins}\label{sec:star}
In this section we study quadratic rational maps belonging to the Cantor locus, that is, maps $g\in {\bfnf}_{\lambda,\sigma}$ with $\lambda\in\D^*$ and $\sigma \in \mathcal{R}^{\lambda}$. As a tool for this, we introduce the concept of \emph{stars} for attracting basins, in Section \ref{sect_stardef}. The stars are used both for the definition of the parametrization $\mathbf{\Phi}$ from definition \ref{defParameter} and to classify the (un)boundedness of the sequences $\sigma_k=\Phi^{\lambda_k}(x)$ of Theorem \ref{thm:convergence} and \ref{thm:convergenceunbound}. In order to construct the (partial) parametrization $\Phi^\lambda$ of $\mathcal{R}^{\lambda}$ we introduce dynamical markings relative to $\mathrm{P}_{\lambda}$ (Section \ref{sec:dynmarkinglambda} and Definition \ref{def:dynmarkinglambda}), and use them both as a model for the dynamics of maps $g\in{\bfnf}_{\lambda,\sigma}$, as well as to construct a model of the locus $\mathcal{R}^{\lambda}$ (see Proposition \ref{prop:dynmarkinglambdabylambda} and Section \ref{chap:gkmodel}). The classification of the boundedness behavior of sequences $\sigma_k=\Phi^{\lambda_k}(x)$ is built on Yoccoz-type inequalities for certain eigenvalues, which are derived from modulus estimates from the star; see Sections \ref{sect:modulus} and \ref{sect:landingwire}, Proposition \ref{prop:eigenvalueconv} and Corollary \ref{cor:fatcycles} and \ref{cor:extracycles}.

Stars for attracting basins were introduced by Petersen in \cite{pet99}, where the case of simply connected quadratic basins is studied. Our objective is infinitely connected quadratic basins. We will give an overview of the construction and properties of stars in the quadratic case, tailored to our application, and refer to \cite{pet99} for more background. We extend main results from \cite{pet99} to the case of infinitely connected basins of quadratic rational maps. 

Given an irreducible rational number $p/q$, $(p,q)=1$, we define as before for $\lambda\in\C^*\setminus \mathbb{S}^1$ the vector $L=L(p/q)=
q\Log(\lambda\e^{-i2\pi p/q})=q\log\lambda-p2\pi i$; here $\Log:
{\C\setminus(-\infty,0]}\to \C$ is the principal branch of $\log$ and
$\log\lambda$ is an appropriate choice of logarithm of $\lambda$, see also Section \ref{results}, before Definition \ref{def:dynmarkingpara}. We define also the number 
\begin{equation}\label{eqn:rlambda}
r_\lambda=r_{\lambda, p/q}:=\frac{|L|}{2q\sin\theta}=\frac{\pi}{q^2 m(\lambda)},
\end{equation}
where $\theta$ is the angle between $L$ and $2\pi i$ and $m(\lambda)$ is defined in Equation (\ref{eqn_mlambdaattrpar}). The number $r_\lambda$ is the radius of the circle through $\log\lambda$, and tangent to the imaginary axis at $2\pi ip/q$, see Figure \ref{radius}. 
\begin{figure}
\begin{center}
   \includegraphics[width=4cm]{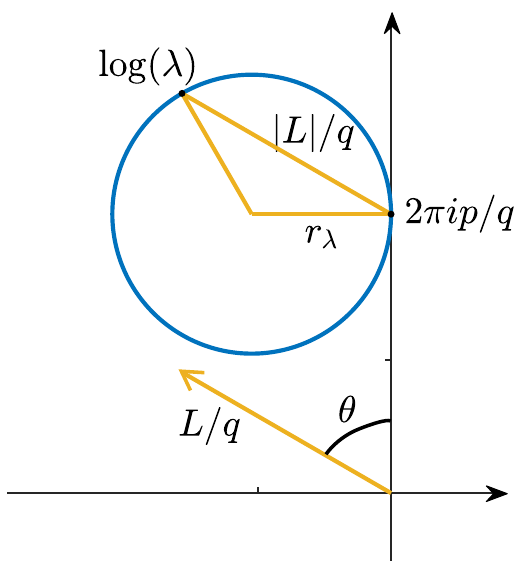} 
\caption{The figure illustrates $r_\lambda$, $\log\lambda$, $L$ and $\theta$.}
    \label{radius}
\end{center}
\end{figure}

We now restrict our attention to $\lambda\in\D_{-\omega}$. For $\lambda\in\D_{-\omega}$ and $\sigma\in \C$, let $g_{\ls}: \CC\to \CC$ denote a quadratic rational map in ${\bfnf}_{\lambda,\sigma}$, with the attracting fixed point of eigenvalue $\lambda$ at $z_0$, so that $g_{\ls}(z_0)=z_0$ and $g_{\ls}'(z_0)=\lambda$. Let
$A_{\lambda,\sigma}$ denote the attracting
basin of $z_0$. Such a map has an extended linearizing coordinate $\phi_{\ls}: A_{\lambda,\sigma} \to \C$, which is a holomorphic and surjective map so that $\phi_{\ls}(z_0)=0$ and $\phi_{\ls}\circ g_{\ls}=\lambda\cdot\phi_{\ls}$. The linearizing coordinate is univalent on some neighborhood of the fixed point and unique up to multiplication by a non-zero constant.
In our notation $g_{\ls}$ denotes a general quadratic rational map and could be a polynomial ($\sigma=0$). When we need to emphasize that the map in question is a polynomial, we use the notation introduced in Section \ref{sect:divergenceinfty} and choose the specific representative $\mathrm{P}_{\lambda}$.

In the following $p/q$ is fixed and $L$ is fixed for a given $\lambda\in\D_{-\omega}$.

\subsection{Definition of the $(\log\lambda,p/q)$-\emph{star} of $A_{\lambda,\sigma}$ for $g_{\ls}$}\label{sect_stardef}

The meromorphic vector field $v(z)\frac{d}{dz}$, given by $v(z)=L\frac{\phi_{\ls}(z)}{\phi_{\ls}'(z)}$, is independent of the choice of normalization of the linearizing coordinate $\phi_{\ls}$ for $g_{\ls}$. The flowlines of the vector field can be thought of as a singular foliation $\mathcal F$ in $A_{\lambda,\sigma}\setminus \bigcup_{n\geq 0} g_{\ls}^{-n}(z_0)$. The following description is convenient to have in mind. The collection of lines parallel with $L$, $\{\R L+iy: y\in \R\}$ defines a foliation of $\C$, and the collection of logarithmic spirals $\{\exp(\R L+iy): y\in \R\}$ defines a foliation of $\C^*$, with leaves $\exp(\R L+iy)$ emanating from infinity and ending at 0. The foliation $\mathcal F$ in $A_{\lambda,\sigma}\setminus \bigcup_{n\geq 0} g_{\ls}^{-n}(z_0)$ is thus equivalently defined as the pullback of the foliation $\{\exp(\R L+iy): y\in \R\}$ in $\C^*$ by a linearizing coordinate $\phi_{\ls}$. See Figure \ref{foliation}.

 Charts for $\mathcal F$ (and linearizing coordinates for the vector field) are of the form $\frac{\log\circ\phi_{\ls}}{L}$, where $\log$ is some branch of the logarithm and $\phi_{\ls}$ as always a linearizing coordinate, whence changes of charts are translations.
In this way, a choice of $L$ defines a unique holomorphic, singular foliation by curves in $A_{\lambda,\sigma}\setminus \bigcup_{n\geq 0} g_{\ls}^{-n}(z_0)$, with singular points precisely at the critical points of $g_{\ls}$ in $A_{\lambda,\sigma}\setminus \bigcup_{n\geq 0} g_{\ls}^{-n}(z_0)$ and their iterated pre-images. 
\begin{figure}
\begin{center}
   \includegraphics[width=12cm]{foliation.pdf} 
\caption{An illustration of leaves of the foliations in $A_{\lambda}\setminus \bigcup_{n\geq 0} \mathrm{P}_{\lambda}^{-n}(0)$, $\C^*$ and $\C$, for $p/q=1/3$ and $\lambda=r \omega_{1/3}$ for an $r>0$.}
    \label{foliation}
\end{center}
\end{figure}

Leaves of $\mathcal F$ either end (in infinite time) at  $z_0$ or one of its iterated pre-images, or connect two singular points of $\mathcal F$. A leaf whose closure contains $z_0$ will be called a \emph{visible} leaf. A visible leaf either begins in finite time at a singular point of $\mathcal F$, or continues in infinite time, both forward and backward. We will say that a point $z\in A_{\lambda,\sigma}$ is \emph{visible}, if $z=z_0$ or $z$ is in the interior of a visible leaf. For a visible point $z$, the time-1 flow of the vector field is equal to the $q$th iterate $g_{\ls}^q(z)$, in particular, a visible leaf is forward invariant under $g_{\ls}^q$. Moreover, visible leaves are grouped in cycles of exact period $q$, each performing a $p/q$ rotation under the map $g_{\ls}$, as can be easily seen from the foliation of $\C$ by $\{\R L+iy: y\in \R\}$. 

There is always at least one visible critical value in $A_{\lambda,\sigma}\setminus \bigcup_{n\geq 0} g_{\ls}^{-n}(z_0)$, and in case $\sigma \in \mathcal{R}^{\lambda}$ there might be two. We will call a visible leaf a \emph{branch}, if it begins at a singular point, every branch is part of a $q$-cycle of branches containing a visible critical value. Let $1\leq k' \leq 2$ denote the number of visible critical values of $g_{\ls}$, let $v_i$, $1\leq i\leq k'$, denote the visible critical values and $c_i$ the corresponding critical points. Let $k$ denote the number of cycles of branches, then $1\leq k\leq k'\leq 2$ with $k<k'$ if and only if $k'=2$ and $v_1$ and $v_2$ belong to the same cycle of branches.

We name branches $\tau_i^j$, $j\in \Z_{q}$ and $1\leq i\leq k$, labelled counterclockwise with respect to their cyclic ordering, and so that $c_i$ is the start point of $\tau_i^0$ (when $k'=2$ there is no preferred choice of ''first'' visible critical value $v_1$).

We now elaborate on the description and properties of the leaves of the foliation. We choose a numbering of the critical values, by letting $v_1$ denote (one of) the visible critical value(s). The linearizing coordinate is then normalized by setting $\phi_{\ls}(v_1)=\lambda$, hence $\phi_{\ls}(c_1)=1$. Let $\zeta_i$ denote a logarithm of $\phi_{\ls}(c_i)$,
where we normalize the logarithm by choosing $\zeta_1$ to be 0. 

Among the lines which are parallel to $L$ we give special attention to those giving rise to branches. Namely, let $\wtt_{i}^j$ denote the lines:
\begin{equation}\label{eqn:tautilde}
\wtt_{i}^j:=tL+\frac {j} {q}2\pi i + \zeta_i \quad \text{for } t\in \R, j\in\Z \text{ and } 1\leq i\leq k,
\end{equation}
and set
\[\wtt_i=\bigcup_{j\in \Z} \wtt_{i}^j \quad \text{ and } \quad \wtt=\bigcup_{1\leq i\leq k} \wtt_{i},\]
where lines in $\wtt$ are named $\wtt^{j}$ according to their horizontal ordering and so that $\wtt^{0}=\wtt_{1}^0$. Note that when $k=2$, $\wtt^{j}\neq\wtt_{i}^j$ in general, since each $\wtt_{1}^j$ is situated between $\wtt_{2}^{j'}$ and $\wtt_{2}^{j'+1}$ for some $j'$ and vice versa. 

The collection of lines $\wtt$ bound open strips $\widetilde U^j$, $j\in \Z$, which are named according to their horizontal ordering and so that $\widetilde U^0$ is bounded below by $\wtt^{0}$. Let $\widetilde \gamma=\bigcup_{j\in\Z} \widetilde \gamma^j$ denote the collection of central straight lines $\widetilde\gamma^j\subset \widetilde U^j$. Each line $\wtt_{i}^j$ and $\widetilde \gamma^j$ and strip $\widetilde U^j$ is invariant under the translation $z\mapsto z+L$, and each $\wtt_{i}^j$ is mapped to $\wtt_{i}^{j+p}$ (and $\wtt^j$, $\widetilde \gamma^j$ and $\widetilde U^j$ to $\wtt^{j+kp}$, $\widetilde \gamma^{j+kp}$ and $\widetilde U^{j+kp}$ respectively) under the translation $z\mapsto z+\log\lambda$.

Each of the collections $\wtt$ and $\widetilde \gamma$ projects by the exponential map to $k$ $q$-cycles, under multiplication by $\lambda$, of disjoint logarithmic spirals and the strips $\widetilde U^j$ to $k$ $q$-cycles of "strips". 
For $j\in \Z_q$ let $\wht_i^j=\exp(\wtt_{i}^j)$. Each $q$-cycle $\wht_i=\exp(\wtt_{i})$ of logarithmic spirals contains the non-zero critical values of $\phi_{\ls}$ coming from $c_i$: $\{\lambda^{-n}\phi_{\ls}(c_i): n\in \Z, n\geq 0\}$. 
The image of the branch $\tau_i^j$ under $\phi_{\ls}$ is contained in $\wht_i^j$, $\phi_{\ls}(\tau_i^j)\subset \wht_i^j$. In fact, since branches $\tau_i^j$ start from points in the backward orbit of the critical value $v_i$, the start point of $\tau_i^j$ corresponds to $t=\frac{-(q-n(j))}{q}$ and $\phi_{\ls}(\tau_i^j)= \widehat{\tau}_i^j|_{t>\frac{-(q-\kappa(j))}{q}}$, where 
\begin{equation}\label{eqn_nj}
\kappa(j)=j/p\in\Z_q \text{ using the representative in } \{1, ..., q\},
\end{equation}
see Figure \ref{fig_star}.

Let $\Upsilon$ denote the collection of visible leaves of $\mathcal F$ (including branches) and for $j\in \Z_{kq}$, let $U^j\subset A_{\lambda,\sigma}$ be the unique connected component of $\phi_{\ls}^{-1}( \exp(\widetilde U^j))$ with $z_0$ on the boundary.

\begin{definition}[The $(\log\lambda,p/q)$-\emph{star} of $A_{\lambda,\sigma}$ for $g_{\ls}$.]
 The $(\log\lambda,p/q)$-\emph{star} of $A_{\lambda,\sigma}$ for $g_{\ls}$, here denoted $\Sigma_{\lambda, \sigma}$, is defined as the following open subset of $A_{\lambda,\sigma}$:
 \[\Sigma_{\lambda, \sigma}=\{z\in A_{\lambda,\sigma} : z \text{ is visible}\}.\]
It follows that
\[\Sigma_{\lambda, \sigma}=\inter\left(\bigcup_{\upsilon\in \Upsilon}\overline \upsilon\right)=\inter(\bigcup_{j\in \Z_{kq}}\overline U^j).\]
\end{definition}

To ease notation, and since $p/q$ is fixed, we have suppressed the dependence on $p/q$.
When $p/q$ and $\log\lambda$ are clear from the context, we will just refer to $\Sigma_{\ls}$ as the \emph{star} for $g_{\ls}$. 
A star can also be described as a maximal domain of univalence for the linearizing coordinate, obtained via analytic extension along logarithmic spirals of direction $L$ in the co-domain. 
That is, the star corresponds to a $kq$-times slit domain in the log-linearizing coordinates, namely slit along the collection of lines $\wtt_{1}$ (and $\wtt_{2}$ for $k=2$), at 0 and $\zeta_2$ and their backward orbits under $z\mapsto z + \log\lambda$, respectively; see Figure \ref{fig_star}.

For each $j\in \Z_{kq}$:
\begin{itemize}
\item the set $U^j$ is called the \emph{$j$th strip} of the star,
\item the leaf $\gamma^j=\phi_{\ls}^{-1}( \exp(\widetilde\gamma^{j}))\cap U^j$ is called the \emph{$j$th wire} of the star,
\item the set $\overline U^j\setminus A_{\lambda,\sigma}$ is called the \emph{$j$th tip} of the star,
\item the set $\tau=\bigcup_{j\in \Z_{q}, 1\leq i\leq k} \tau_i^j \cup \{z_0\}$ is called the \emph{twig} of the star and the set $\tau_i=\bigcup_{j\in \Z_{q}} \tau_i^j \cup \{z_0\}$ the twig associated to the critical value $v_i$. The twig $\tau_i$ contains the forward orbit of $v_i=g_{\ls}(c_i)$. 
\end{itemize}
Sometimes we will need to consider a truncated twig, or truncated branches of the twig, in this case we will use the parametrization from Equation (\ref{eqn:tautilde}), so that, for example, $\tau_i^j|_{t\geq K}$ is defined by $\phi_{\ls}(\tau_i^j|_{t\geq K})=\widehat{\tau}_{i}^j|_{t\geq K}=\exp(\wtt_{i}^j|_{t\geq K})$.

\begin{figure}
\begin{center}
     \includegraphics[width=13.7cm]{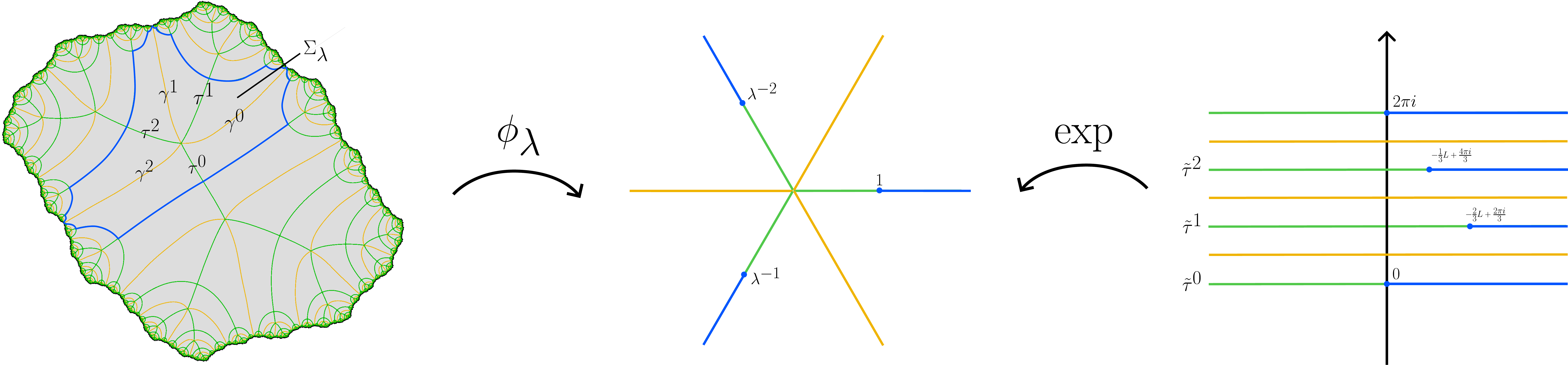}
\caption{A 1/3-star in the basin of 0 for $\mathrm{P}_{\lambda}$, $\lambda=r\omega_{1/3}$, $r<1$. The star is the domain bounded by the blue contour. The figure also shows branches $\tau^j$ (green) and wires $\gamma^j$ (orange).  Dynamical picture by Arnaud Ch{\'e}ritat.}
\label{fig_star}
\end{center}
\end{figure}

We summarize the important properties of the star. 
For each $j\in \Z_{kq}$ the log-linearizing coordinate $\log\circ\, \phi_{\ls}: U^j\to \widetilde U^j$ is biholomorphic, and conjugates $g_{\ls}^q$ to $z\mapsto z+L$,
and the strip $U^j$ and the wire $\gamma^j$ are invariant under $g_{\ls}^q$. The collection of strips $U^j$ and wires $\gamma^j$ are each grouped in $k$ $q$-orbits under $g_{\ls}$, and each orbit performs a $p/q$ rotation around $z_0$. 
Similarly, branches $\tau^j$ are grouped in $k$ twigs: forward invariant $q$-cycles of branches. Note that $g_{\ls}$ maps the star univalently into itself. 

\paragraph{A $q$-cycle of wires.} A $q$-cycle of wires will be denoted $\gamma$, where we will understand $\gamma$ as the union of the wires in the cycle and $z_0$.

\paragraph{Landing of wires.} Let $\gamma$ denote a $q$-cycle of wires of the star. We will say that the wires \emph{land}, if the corresponding tips of the star are one point sets.

\begin{lemma}\cite[Lemma 3.3 and 3.4]{pet99}\label{lemma:wiresland}
For $\lambda\in\D_{-\omega}$ and $\sigma\in \C$ let $g=g_{\ls}\in{\bfnf}_{\lambda,\sigma}$. Any $q$-cycle of wires $\gamma$ in the $(\log\lambda,p/q)$-\emph{star} of $A_{\lambda,\sigma}$ for $g$ lands. Moreover, the wires in the cycle either land together on a fixed point of $g$, or they land separately on a $q$-cycle of $g$.
\end{lemma}
Figure \ref{fig_stars} shows examples of stars, and illustrates cases with 1 or 2 cycles of wires, and landing on fixed points and $q$-cycles respectively. 

Note that by our convention $v_1\in \Sigma_{\ls}$ always and that $k=2$ if and only if $v_2\in \Sigma_{\ls}\setminus \tau_1$. We expand on the case $v_2\in \tau_1$, since it will play a role later.  

\paragraph{Two visible critical values on the same twig.} 
As usual, we label one of the critical values 
$v_1$ and normalize by $\phi_{\ls}(v_1)=\lambda$. Then $v_2$ is on a branch $\tau_1^j$, $j\in \Z_q$. 

When there are two critical values on the same twig, the orbits of the two can not be too far apart: both critical values must be in front of the backward orbit of the other, in order to be visible. 
To be precise, let 
\begin{equation}\label{eqn_lj}
s(j)=j/p\in\Z_q \text{ using the representative in } \{2, ..., q+1\}, 
\end{equation}
then
\begin{equation}\label{eqn_2on1twig}
v_2\in \tau_1^j|_{t<\frac{s(j)}{q}},
\end{equation}
where we again use the parametrization from Equation (\ref{eqn:tautilde}).

\begin{figure}
  \begin{center}
    \begin{tabular}{cc}
      \includegraphics[width=4.7cm]{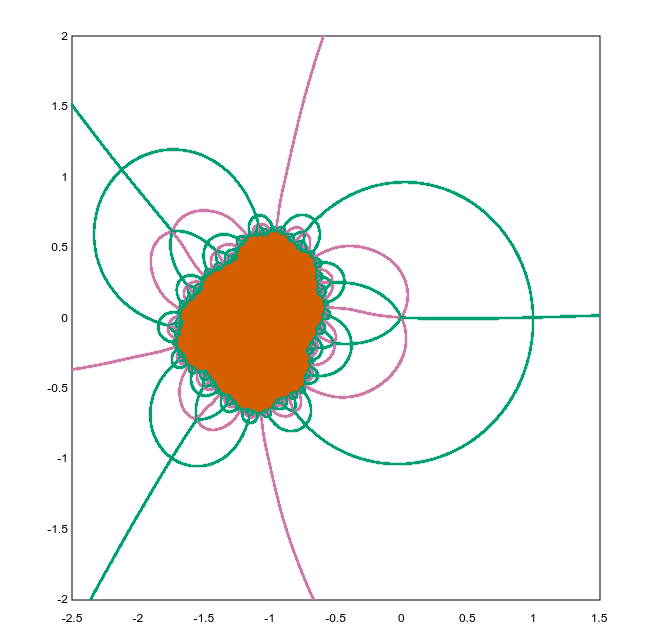} &
    \includegraphics[width=4.7cm]{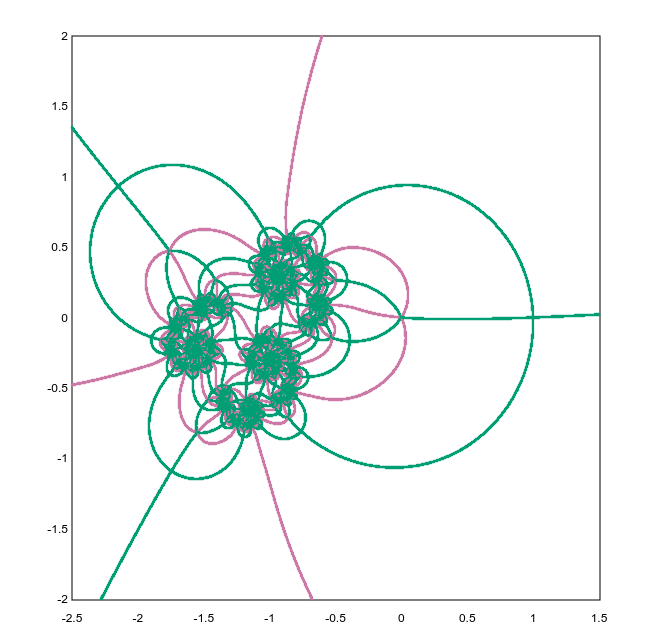}  \\
      \includegraphics[width=4.7cm]{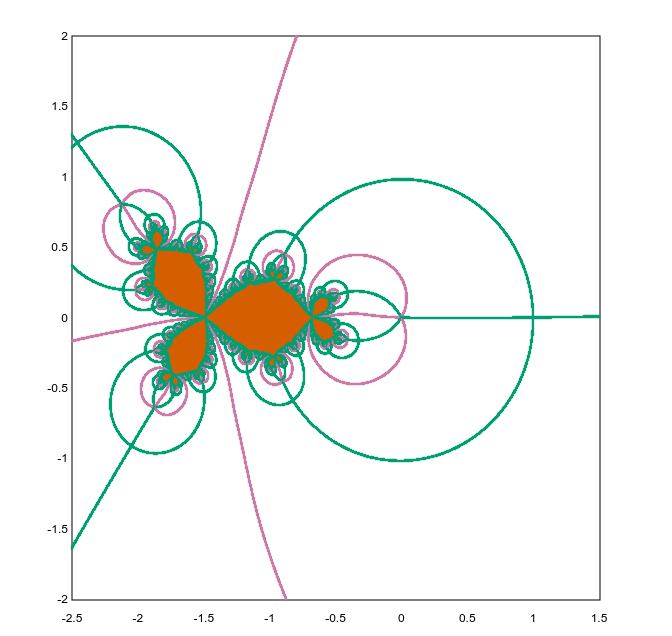} &
    \includegraphics[width=4.7cm]{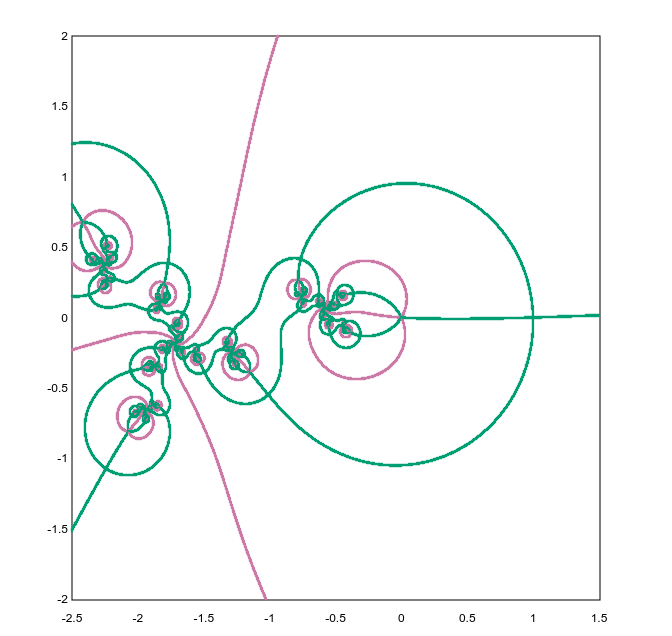}  \\
      \includegraphics[width=4.7cm]{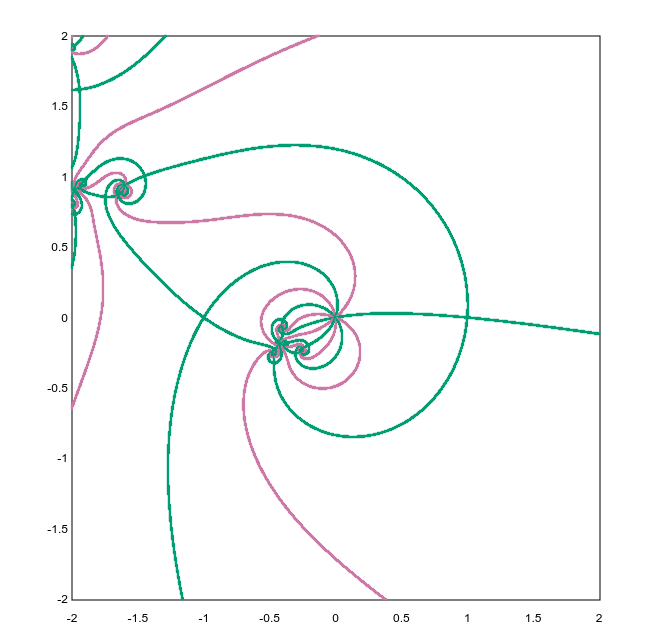} &
    \includegraphics[width=4.7cm]{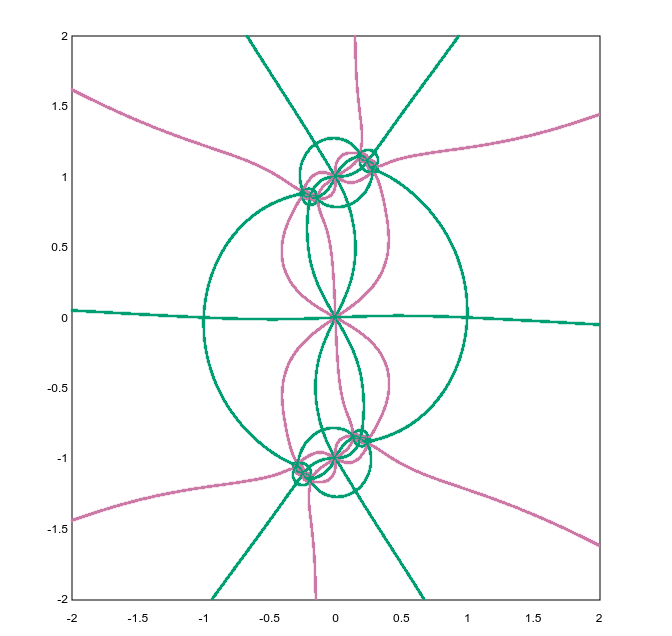}  
    \end{tabular}
  \end{center}
  \caption{Shows stars for $g_{\lambda,a}\in \Gamma_{\lambda,\sigma}$ for $\lambda=\frac{1}{2}\omega_{1/3}$ and different choices of $\sigma$. We conjugate by $z\mapsto \frac{1}{z}$ to put infinity at 0. Branches (and their pre-images) are shown in green, and wires (and their pre-images) are shown in red.
  The top row shows a situation where $k=1$ and the only cycle of wires $\gamma$ lands on a $q$-cycle ($\sigma\in \mathcal{R}^\lambda$ on the right and $\sigma\in M^\lambda$ on the left for comparison). 
    The middle row shows a situation where $k=1$ and $\gamma$ lands on a fixed point ($\sigma\in \mathcal{R}^\lambda$ on the right and $\sigma\in M^\lambda$ on the left for comparison).
    The bottom row shows cases where $k=2$. On the left
    one cycle of wires land on a fixed point and the other on a $q$-cycle. On the right, both cycles of wires land on $q$-cycles. Program by Christian Henriksen. } \label{fig_stars}
\end{figure}

\subsection{Dynamical marking relative to $\mathrm{P}_\lambda$}\label{sec:dynmarkinglambda}
The purpose of this section is to use stars to compare the structure of the basin $A_{\ls}$ to the basin $A_\lambda$ of $\mathrm{P}_\lambda$. In particular, we match the star $\Sigma_{\ls}$ for $g=g_{\ls}$ to the star $\Sigma_{\lambda}$ for $\mathrm{P}_\lambda$, in order to record the position of the critical values of $g$ relative to the $p/q$-star $\Sigma_{\lambda}$ for $\mathrm{P}_\lambda$. This will define a dynamical marking rel $\mathrm{P}_\lambda$. 
First we consider $\mathrm{P}_\lambda$ in a little more detail.

\paragraph{The polynomial $\mathrm{P}_\lambda$.} For $\lambda\in \D_{-\omega}$, consider the polynomial $\mathrm{P}_\lambda:\CC\to\CC$, given by:
\[\mathrm{P}_\lambda(z)=\lambda z+z^2.\] 
$\mathrm{P}_\lambda$ has an attracting fixed point at 0 with eigenvalue $\lambda$ and a critical point at $-\lambda/2$, with critical value $-\lambda^2/4$. Let $A_\lambda$ denote the attracting basin of 0 for $\mathrm{P}_{\lambda}$ and $\phi_\lambda:A_\lambda\to\C$ an extended linearizing coordinate for $\mathrm{P}_{\lambda}$, that is a holomorphic and surjective map so that $\phi_\lambda(0)=0$ and $\phi_\lambda\circ\mathrm{P}_{\lambda} = \lambda \cdot \phi_\lambda$, normalized so that $\phi_\lambda(-\lambda/2)=1$.

We let $\Sigma_\lambda$ denote the $(\log\lambda,p/q)$-star of $A_\lambda$ for $\mathrm{P}_{\lambda}$. There is one critical point in the basin $A_\lambda$, whence $k'=k=1$. Let $\gamma_\lambda$ denote the $q$-cycle of wires,  and $\tau_\lambda$ the twig, for the star $\Sigma_\lambda$ for $\mathrm{P}_{\lambda}$. The Julia set of $\mathrm{P}_{\lambda}$ is a quasi-circle, in particular a Jordan curve, so $\gamma_\lambda$ necessarily lands on a $q$-cycle of $\mathrm{P}_{\lambda}$. The $q$-cycle of wires $\gamma_\lambda$ separates the basin $A_{\lambda}$ into $q$ simply connected components, the component that contains the critical value $-\lambda^2/4$ will be called the \emph{critical value sector given by the $p/q$-star}, and denoted $S^p_\lambda$, see Figure \ref{critvaluesector}.
\begin{figure}
\begin{center}
     \includegraphics[width=4cm]{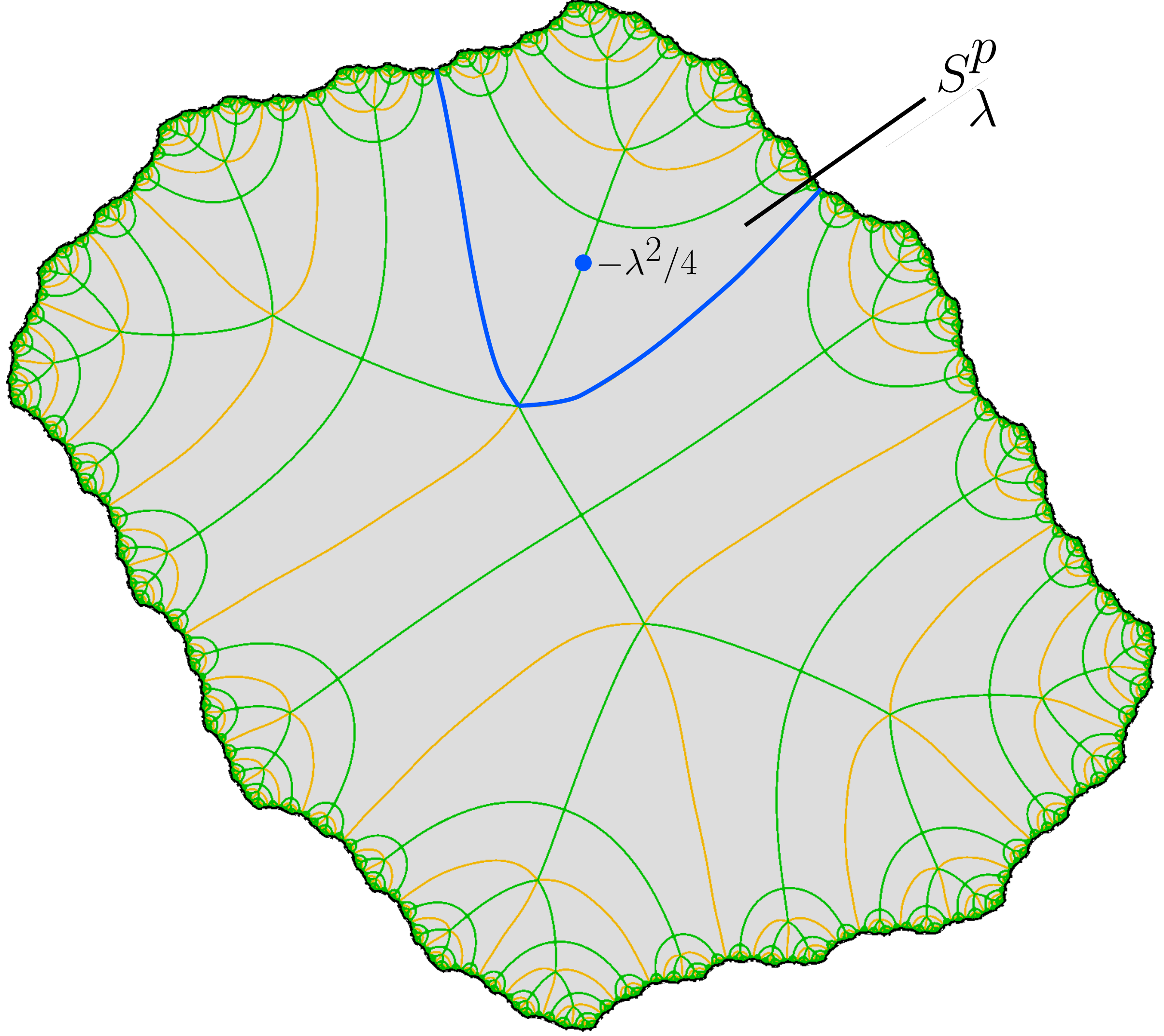}
\caption{\label{critvaluesector} The basin of 0 for $\mathrm{P}_{\lambda}$, $\lambda=r\omega_{1/3}$, $r<1$, showing the critical value sector $S_\lambda^p$, bounded by the wires in blue. Dynamical picture by Arnaud Ch{\'e}ritat.}
\end{center}
\end{figure}

The idea behind the formal definition of dynamical marking relative to $\mathrm{P}_\lambda$ (Definition \ref{def:dynmarkinglambda}) is the following. 
Suppose $\lambda\in\D^*$ and $\sigma\in \mathcal{R}^{\lambda}$ and let $g\in{\bfnf}_{\lambda,\sigma}$. 
Since $\sigma\in \mathcal{R}^{\lambda}$, $g$ has two critical values in the attracting basin $A_{\lambda,\sigma}$.  As always, we choose a labeling of the critical values of $g$, so that $v_1$ is a visible critical value and let the linearizing coordinate $\phi_{\lambda,\sigma}$ be normalized by $\phi_{\lambda,\sigma}(v_1)=\lambda$. Let $\psi$ be defined by $\phi_{\lambda,\sigma}^{-1}\circ\phi_{\lambda}$ on the domain $\phi_\lambda^{-1}(\phi_{\lambda,\sigma}(\Sigma_{\lambda,\sigma}))\cap \Sigma_{\lambda}$, where $\phi_{\lambda,\sigma}^{-1}$ is the inverse of the restriction $\phi_{\lambda,\sigma}|_{\Sigma_{\lambda,\sigma}}$. Thus, $\psi(0)=z_0$, $\psi(-\frac{\lambda^2}{4})=v_1$, $\psi$ is a conformal isomorphism onto its image and obeys the functional equation
$\psi\circ \mathrm{P}_\lambda=g\circ \psi$.

If necessary, $\psi$ is then extended by iterated use of this functional equation until both critical values of $g$ are in the range of $\psi$, and we set $x:=\psi^{-1}(v_2)$. We then say that $(x,\psi)$ is a dynamical marking of $g$. We elaborate on the definition of dynamical marking rel $\mathrm{P}_\lambda$ and the domain of a dynamical marker map $\psi$ in the following.

Let $\tau^*_\lambda$ denote the twig for $\Sigma_\lambda$, cut off to start at the first $q$ points in the forward orbit of the critical value $-\lambda^2/4$, namely $\mathrm{P}^{i}_\lambda(-\lambda^2/4)$ $i=1, \ldots, q$, i.e. 
\begin{equation}\label{truncatedtwig}
\tau_\lambda^*:=\bigcup_{j=0}^{q-1}{\tau^j}_{|t\geq\frac{s(j)}{q}},
\end{equation}
where $s(j)$ is the representative of $j/p \in \Z_{q}$ in $\{2, \ldots, q+1\}$ as in Equation (\ref{eqn_lj}). Notice that if $\phi_{\lambda,\sigma}(v_1)=\lambda$ and $v_2 \in \tau_1$, then
$x\in \tau_\lambda\setminus \tau^*_\lambda$, by Equation (\ref{eqn_lj}) and the discussion preceding it.

Let $\Sigma_\lambda^*= \Sigma_\lambda\setminus \tau_\lambda^*$, see Figure \ref{attractingdomains}.

\begin{figure}
\begin{center}
    \includegraphics[width=9cm]{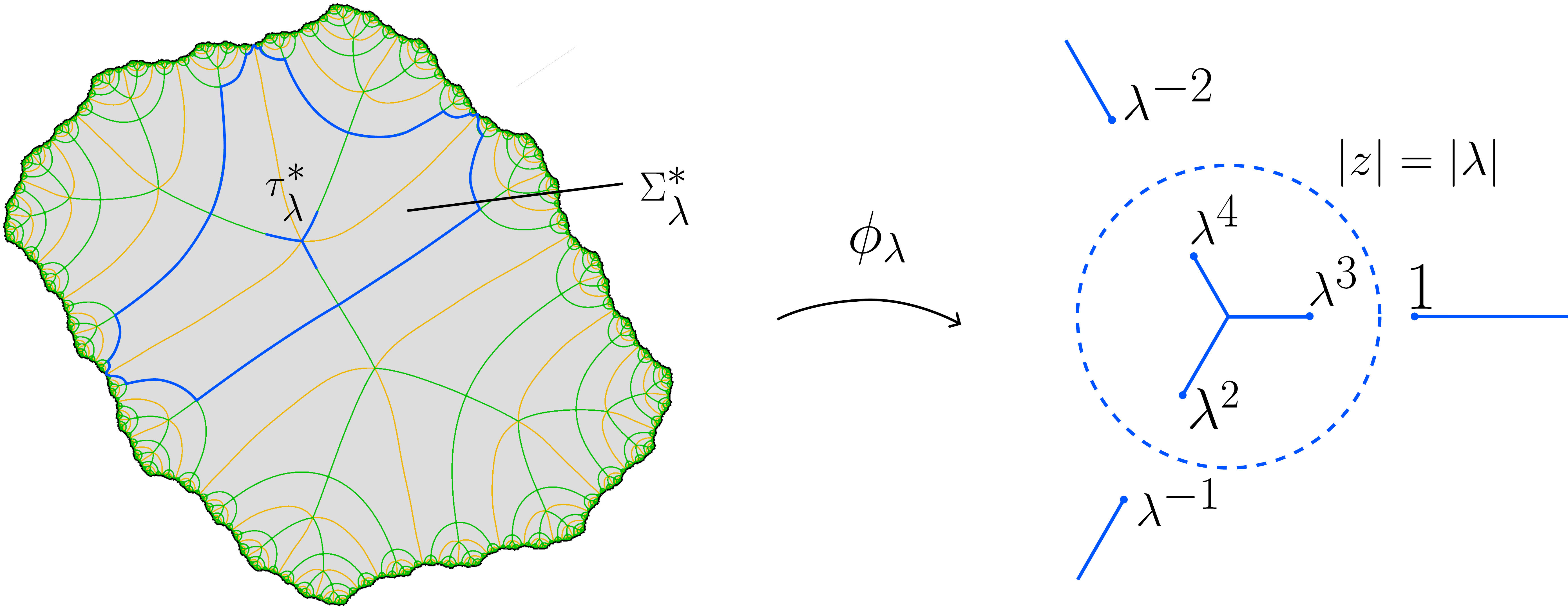}
\end{center}
	\caption{Illustration of the cut-off twig $\tau_\lambda^*$ and the domain $\Sigma^*_\lambda$.}
    \label{attractingdomains}
\end{figure}

\subsubsection{Definition of dynamical marking relative to $\mathrm{P}_\lambda$}\label{sect:dynmarkrelPl}

For the following, fix a $\lambda \in\D_{-\omega}$. For $x\in A_\lambda\setminus (\tau_\lambda^*\cup \{-\lambda^2/4\})$,
\begin{itemize}
 \item if $x\in \Sigma^*_\lambda\setminus \{-\frac{\lambda^2}{4}\}$ 
 let $U=U_\lambda(x):=\Sigma_\lambda\setminus \upsilon_{\{\mathrm{P}^{-i}_\lambda(x)\}_{i=1}^q}$, where $\upsilon_{\{\mathrm{P}^{-i}_\lambda(x)\}_{i=1}^q}$ is the cycle of leaves (or branches) containing $x$, cut off to end at the first $q$ points in the backwards orbit of $x$ in $\Sigma_\lambda$,  $\{\mathrm{P}^{-i}_\lambda(x)\}_{i=1}^q$, instead of at 0. See the proof of Proposition \ref{prop:dynmarkinglambdabylambda} for a precise parametrization. 
 \item If $x\in A_\lambda\setminus \Sigma_\lambda$ let $l=l(x)\geq 1$ be minimal such that $\mathrm{P}_\lambda^{l}(x)\in \Sigma_\lambda$ and set $U=U_\lambda(x):=\mathrm{P}_\lambda^{-l}(\Sigma_\lambda)$.
 \end{itemize}
In either case, the set $U_\lambda(x)$ is open, connected, simply connected, forward invariant under $\mathrm{P}_\lambda$ and $0, -\lambda^2/4, x \in U_\lambda(x)$. 

\begin{definition}[Dynamical marking relative to $\mathrm{P}_\lambda$]\label{def:dynmarkinglambda}
Let $\lambda\in \D_{-\omega}$, $\sigma \in \mathcal{R}^\lambda$ and let $g\in \Gamma_{\lambda,\sigma}$ with a fixed point at $z_0$ of eigenvalue $\lambda$. A \emph{dynamical marking} of $g$ \emph{relative} to $\mathrm{P}_\lambda$ (written rel $\mathrm{P}_\lambda$) is a pair $(x,\psi)$ where $x\in A_\lambda\setminus (\tau^*_\lambda\cup \{-\lambda^2/4\})$ and $\psi: U\to V$ is a holomorphic conjugacy, 
\begin{equation}\label{functeqnattracting}
\psi\circ \mathrm{P}_\lambda=g\circ \psi,
\end{equation} 
with $\psi(-\frac{\lambda^2}{4})$ and $\psi(x)$ equal to the critical values of $g$, and where $U=U_{\lambda}(x)$ is as defined above.
\end{definition}
Note that $-\frac{\lambda^2}{4}$ is excluded as a possible dynamical marking, since $v_1\neq v_2$ always for quadratic rational maps. Points in $\tau^*_\lambda$ also never occur as dynamical markings rel $\mathrm{P}_\lambda$, see Equation (\ref{eqn_2on1twig}) and the preceding discussion. 

Also note that the domain $U$ of the dynamical marker map $\psi$ depends on the choice of $p/q$, since it is constructed from the $p/q$ star for $\mathrm{P}_\lambda$. However, the dynamical marking $x$ does not depend on $p/q$, and the definition could be made with a different choice of domain; the present choice is convenient for our purpose. For this reason and to ease the notation, the dependence on $p/q$ will be implicitly understood and suppressed in the notation. Also, when $\mathrm{P}_\lambda$ is understood from the context, we will omit "rel $\mathrm{P}_\lambda$". 

A dynamical marking induces a labelling  $v_1=\psi(-\frac{\lambda^2}{4})$ and $v_2=\psi(x)$ of critical values of $g$. We will say that $g$ is \emph{dynamically marked} by $(x,\psi)$ if $(x,\psi)$ is a dynamical marking of $g$. It follows from the required properties of $U$ and $\psi$ that $\psi(0)=z_0$ and that $V$ is an open, connected, simply connected, forward invariant subset of the attracting basin $A_{\lambda,\sigma}$ of $g$ at $z_0$.

The following lemma implies that distinct equivalence classes have distinct markings. 
\begin{lemma}\label{lemma:pullback}
Let $\lambda \in \D_{-\omega}$ and $\sigma,\sigma' \in\mathcal{R}^\lambda$. If $g\in \Gamma_{\lambda,\sigma}$ is dynamically marked by $(x,\psi)$ and $g'\in \Gamma_{\lambda,\sigma'}$ is dynamically marked by $(x,\psi')$, both rel $\mathrm{P}_\lambda$,
then $\psi'\circ\psi^{-1}$ extends to a M\"obius conjugacy of $g$ to $g'$, in particular $\Gamma_{\lambda,\sigma}=\Gamma_{\lambda,\sigma'}$ and equivalently $\sigma=\sigma'$. 
\end{lemma}
The lemma is proved by a standard pull-back argument, we will not prove it here. 

As another consequence of Lemma \ref{lemma:pullback}, if a map $g$ is dynamically marked by $(x,\psi)$ then $\psi$ is unique up to automorphisms of $g$. In view of this property, we will speak of $x$ as a dynamical marking of $g$. 
 
Moreover, if $x$ is a dynamical marking of $g\in \Gamma_{\lambda,\sigma}$, then $x$ is a conformal invariant of $\Gamma_{\lambda,\sigma}$, i.e. $x$ is a dynamical marking of any $f\in \Gamma_{\lambda,\sigma}$: if $(x,\psi)$ is a marking of $g$, then $(x, H\circ\psi)$ is a marking of $f$, where $H$ is a M\"{o}bius transformation conjugating $g$ to $f$. For this reason we will also speak of $x$ as a dynamical marking of the equivalence class $\Gamma_{\lambda,\sigma}$ or of its parameter $\sigma\in \mathcal{R}^\lambda$.

\subsubsection{Construction of dynamical markings}\label{subsect:dynconj}
 Consider the map $z\mapsto \lambda^2/z$ on $\C^*$, it is an involution that preserves the circle $|z|=|\lambda|$, and with fixed points $\lambda$ and $-\lambda$. 
 Let $\mathcal{U}_\lambda$ denote the component of the pre-image by $\phi_\lambda$ of the disc $|z|<|\lambda|$, containing 0.
 Then $I_\lambda:=\phi^{-1}_{\lambda}\circ \frac{\lambda^2}{\phi_{\lambda}}:\Sigma^*_\lambda\to \Sigma^*_\lambda$ is an involution preserving $\partial \mathcal{U}_\lambda$, with fixed points $\phi_\lambda^{-1}(\lambda)=-\lambda^2/4$ and $\phi_\lambda^{-1}(-\lambda)$. 
Notice that $\Sigma^*_\lambda$ is the largest domain in $\Sigma_\lambda$ on which $I_\lambda$ can be defined, since the start points of the truncated twig $\tau^*_\lambda$ (i.e. ${\tau^j}|_{t=\frac{s(j)}{q}}$, see Equation (\ref{truncatedtwig})) are exactly the images of the start points of branches ($\overline{\tau^j}|_{t=\frac{-(q-\kappa(j))}{q}}$, see  Equation (\ref{eqn_nj})) under the involution $I_\lambda$.

We elaborate on the construction of dynamical markings in the proof of the following proposition.
\begin{prop}\label{prop:dynmarkinglambdabylambda}
Let $\lambda\in \D_{-\omega}$. For each $\sigma\in \mathcal{R}^\lambda$ there exists an $x\in 
A_\lambda\setminus (\tau^*_\lambda\cup \{-\frac{\lambda^2}{4}\})
$, such that $\Gamma_{\lambda,\sigma}$ is dynamically marked by $x$ rel $\mathrm{P}_\lambda$. 
Either
\begin{itemize}
\item $x\notin \Sigma_{\lambda}$ and $x$ is the unique marking of $\sigma$, or
\item $x\in \Sigma^*_\lambda\setminus \{-\frac{\lambda^2}{4}\}$ and $x$, $I_\lambda(x)$ are the only markings of $\sigma$.
\end{itemize}
\end{prop}
\begin{proof}
Let $\lambda\in \D_{-\omega}$, $\sigma\in \mathcal{R}^\lambda$ and $g=g_{\lambda,\sigma}\in \Gamma_{\lambda,\sigma}$.
Suppose $(x,\psi)$ is a dynamical marking of $g$, then $\phi_{\lambda,\sigma}\circ\psi$ is a linearizing coordinate for $\mathrm{P}_\lambda$. Thus, the marker map $\psi$ preserves the foliations in the sense that it maps (visible) leaves in $U\subset A_\lambda$ to (visible) leaves in $V\subset A_{\ls}$. In particular, the critical value  dynamically labeled $v_1=\psi(-\lambda^2/4)$ is always a visible critical value of $g$. Since linearizing coordinates are unique up to normalization, $\phi_{\lambda,\sigma}$ can be normalized so that $\phi_{\lambda,\sigma}\circ\psi=\phi_{\lambda}$ on $U$, by letting $\phi_{\lambda,\sigma}(v_1)=\phi_{\lambda,\sigma}(\psi(-\lambda^2/4))=\phi_{\lambda}(-\lambda^2/4)$.

Thus, we must choose a labeling of the critical values of $g$, so that $v_1$ is a visible critical value of $g$ and we normalize $\phi_{\lambda,\sigma}$ by $\phi_{\lambda,\sigma}(v_1)=\lambda$. 
There are now two cases:

\paragraph{There is one critical value of $g$ in the $p/q$-star $\Sigma_{\lambda,\sigma}$.}
In other words, $g$ has one visible critical value, $k=1$ and $v_2\notin \Sigma_{\lambda,\sigma}$. The map 
\[\psi:=\phi_{\ls}^{-1}\circ\phi_\lambda: \Sigma_{\lambda}\to \Sigma_{\ls},\]
is a conformal isomorphism, with $\psi(0)=z_0$, $\psi(\frac{-\lambda^2}{4})=v_1$ and
\begin{equation}\label{eqn:etafuncteqn}
g\circ\psi = \psi\circ\mathrm{P}_{\lambda}.
\end{equation}
Let $l\geq 1$ be minimal such that $g^{l}(v_2)\in \Sigma_{\lambda,\sigma}$ and set $U=\mathrm{P}_\lambda^{-l}(\Sigma_\lambda)\subset A_{\lambda}$ and $V=g^{-l}(\Sigma_{\ls}) \subset A_{\lambda,\sigma}$, then $v_2\in V$ and $c_2\notin V$. The conjugacy $\psi$ can be extended by means of the functional equation (\ref{eqn:etafuncteqn}) to a conformal conjugacy, $\psi:U\to V$, which coincides with the definition of $\psi$ on $\Sigma_{\lambda}$. Then $\psi$ is a dynamical marker map of $g$ and $x:=\psi^{-1}(v_2)\in A_\lambda\setminus \Sigma_\lambda$ is a dynamical marking of $g$, hence of the class ${\bfnf}_{\lambda,\sigma}$ and parameter $\sigma$.

\paragraph{Both critical values of $g$ are in the $p/q$-star $\Sigma_{\lambda,\sigma}$,}
i.e. $v_2\in \Sigma_{\lambda,\sigma}$ so there are two visible critical values.

Let $U=\phi_\lambda^{-1}(\phi_{\lambda,\sigma}(\Sigma_{\lambda,\sigma}))\cap \Sigma_\lambda$ and $V=\Sigma_{\ls}$. Then the conformal isomorphism  
\[\psi:=\phi_{\ls}^{-1}\circ\phi_\lambda: U\to V,\]
with $\psi(0)=z_0$ and $\psi(\frac{-\lambda^2}{4})=v_1$ obeys the functional equation ($\ref{eqn:etafuncteqn}$). Thus, it is a dynamical marker map of $g$ with $x:=\psi^{-1}(v_2)$ as a dynamical marking of $g$, and of the class ${\bfnf}_{\lambda,\sigma}$.

There are now two cases.
If the critical values are on distinct cycles of branches  (i.e. $v_2\in \Sigma_{\lambda,\sigma}\setminus \tau_1$ and $k=2$), then  $x\in \Sigma_{\lambda}\setminus \tau_\lambda$ and 
\[U=\Sigma_{\lambda}\setminus \bigcup_{j=0}^q\phi_\lambda^{-1}\left(\wht_2^j|_{t\leq \frac{-(q-\kappa(j))}{q}}\right),\]
where $\kappa(j)$ is as defined in Equation (\ref{eqn_nj}). See an example of this case in Figure \ref{fig_dynmark}.

If, on the other hand, the critical values are on the same cycle of branches, then $k=1$ and $v_2\in \tau_1$, more precisely, $v_2\in \tau_1^j|_{t<\frac{s(j)}{q}}$,
for some $j\in\Z_q$ (see Equation \ref{eqn_2on1twig} and the discussion preceding it). This means that $x\in \tau_\lambda^j|_{t<\frac{s(j)}{q}}$, in particular $x\in \Sigma^*_\lambda\setminus \{-\frac{\lambda^2}{4}\}$ and 
\[U=\Sigma_{\lambda}\setminus \bigcup_{j=0}^q\phi_\lambda^{-1}\left(\phi_{\ls}(c_2) \cdot \wht_1^j|_{t\leq \frac{-(q-\kappa(j))}{q}}\right),\]
where $\kappa(j)$ and $s(j)$ are as defined in Equations (\ref{eqn_nj}) and (\ref{eqn_lj}). 

Note that $\bigcup_{j=0}^q\phi_\lambda^{-1}\left(\wht_2^j|_{t\leq \frac{-(q-\kappa(j))}{q}}\right)$ (or $\bigcup_{j=0}^q\phi_\lambda^{-1}\left(\phi_{\ls}(c_2) \cdot \wht_1^j|_{t\leq \frac{-(q-\kappa(j))}{q}}\right)$ respectively) is just a specific notation for $\upsilon_{\{\mathrm{P}^{-i}_\lambda(x)\}_{i=1}^q}$, the cut-off cycle of leaves (or branches) referred to for the definition of $U_\lambda(x)$, using the parametrization from Equation (\ref{eqn:tautilde}).
\\

This means that for any $\lambda\in \D_{-\omega}$, $\sigma \in \mathcal{R}^\lambda$, $g\in \Gamma_{\lambda,\sigma}$ there exists an $x\in A_\lambda\setminus (\tau^*_\lambda\cup \{-\lambda^2/4\})$ such that $g$ is dynamically marked by $x$ rel $\mathrm{P}_\lambda$.

If there is only one visible critical value of $g$ (i.e. $v_2 \notin \Sigma_{\lambda,\sigma}$ and $x=\psi^{-1}(v_2)\notin \Sigma_{\lambda}$), then there is no choice of critical value for the prescribed normalization of $\phi_{\lambda,\sigma}$, so $\psi$ and $x$ are uniquely determined by $g$. If, on the other hand, there are two visible critical values of $g$ (i.e. $v_2\in\Sigma_{\lambda,\sigma}$ and  
 $x=\psi^{-1}(v_2)\in \Sigma^*_\lambda \subset \Sigma_{\lambda}$), then the choice of $v_1$ for the normalization of $\phi_{\lambda,\sigma}$ determines the marking $x:=\psi^{-1}(v_2)$ of $g$. Changing the choice of critical value for the normalization of $\phi_{\ls}$, corresponds to changing the normalization by post-composing $\phi_{\ls}$ with $z\mapsto z\frac{\lambda}{\phi_{\lambda}(x)}$. With this new normalization $g$ is then dynamically marked by $I_\lambda(x)$.

Moreover, if $I_\lambda(x)\neq x$ then $\psi$ is uniquely determined by $g$ and a marking $x$ of $g$. Lastly, if $I_\lambda(x)= x$, equivalently $\phi_\lambda(x)=-\lambda$, then $g$ has an automorphism ($\sigma$ is in the symmetry locus in ${\bf rat}_2$) and $\psi$ is only unique up to an automorphism of $g$, i.e. a M\"obius conjugacy exchanging $v_1$ and $v_2$. 
\end{proof}

\begin{figure}
\begin{center}
     \includegraphics[width=13cm]{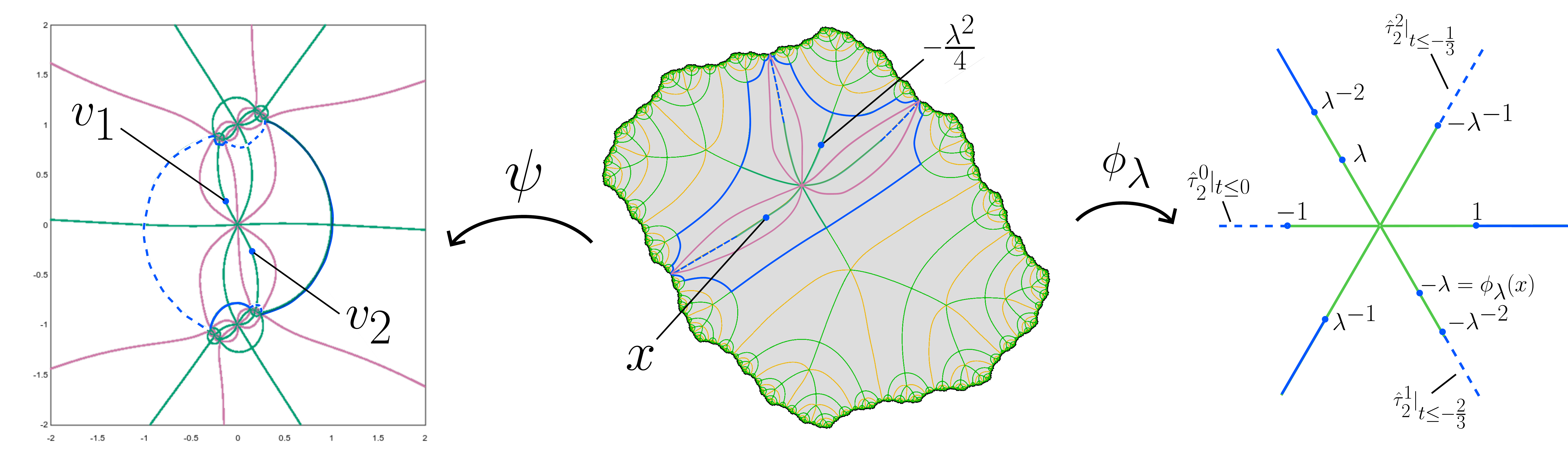}
\caption{Illustrates the dynamical marking $(x,\psi)$ for $G_{\lambda,0}$ with $\lambda=\frac{1}{2}\omega_{1/3}$. This map has $k=2$, that is, its star has two cycles of wires, and they both land on $q$-cycles. The domain $U$ and range $V$ of $\psi$ are bounded by blue curves. The left picture shows $G_{\lambda,0}$, the middle $\mathrm{P}_{\lambda}$ and the right picture illustrates how to construct $U$ from the co-domain of the linearizing coordinate. 
} 
\label{fig_dynmark}
\end{center}
\end{figure}

For the remainder of the discussion, $\psi: U\to V$ denotes a dynamical marker map for $g=g_{\ls}$, with domain and range constructed as described above. If there is need to emphasize that $\psi$ is a dynamical marker map for $g_{\ls}$, we will use indices $\psi_{\ls}$, $U_{\ls}$ and $V_{\ls}$,  otherwise we will suppress this dependence in the notation.

\subsection{A model for $\mathcal{R}^{\lambda}$ recording dynamical marking}\label{chap:gkmodel}
Each $\mathcal{R}^\lambda=\C\setminus M^\lambda$ is isomorphic to $\C\setminus \overline{\D}$.
We introduce a specific model for $\mathcal{R}^{\lambda}$, with $\lambda\in \D_{-\omega}$, which records the dynamical marking of $\sigma\in \mathcal{R}^\lambda$ relative to $\mathrm{P}_{\lambda}$. It is inspired by a model introduced by Goldberg and Keen in \cite{gold90}. Their formulation is quite different from ours, since they worked with marked critical points, so that they study critically marked slices $\per(\lambda)^{cm}$, which are twofold branched coverings of $\per(\lambda)$, parametrized by the normal form $G_{\lambda,a}$. (See also \cite{mil93} for a discussion of the different versions of moduli space and slices herein.)

The model space will be based on the $p/q$-star $\Sigma_\lambda$ for the basin $A_\lambda$. The involution $I_\lambda: \Sigma^*_\lambda\to \Sigma^*_\lambda$ induces an equivalence relation on $\Sigma^*_\lambda$.

\begin{definition}\label{def:equivrelattpq}
  Two points $z_1,z_2 \in \Sigma^*_{\lambda}$ are called equivalent
  modulo $\lambda$, written $z_1 \sim_{\lambda} z_2$, if $z_2=I_\lambda(z_1)$.
\end{definition}

The model space, denoted $\Delta^\lambda$, is defined as the quotient 
\[\Delta^{\lambda}:=(A_\lambda\setminus \tau^*_\lambda)/_{\sim_{\lambda} \text{ on }  \Sigma^*_\lambda}.\] 
Clearly $\Delta^\lambda$ is a Riemann surface, isomorphic to the unit disk. Let $\pi_\lambda$ denote the projection map $\pi_\lambda: A_\lambda\setminus \tau^*_\lambda \to \Delta^\lambda$ and let $\Delta^{\lambda *}={\Delta^\lambda}\setminus \pi_{\lambda}(-\lambda^2/4)$. Then $\Delta^{\lambda *} \cong \D^*$.

It follows from Proposition \ref{prop:dynmarkinglambdabylambda} and Lemma \ref{lemma:pullback} that an injective map 
\[\chi^{\lambda}:\mathcal{R}^{\lambda}\to \Delta^{\lambda *}\] can be uniquely defined by letting $\chi^{\lambda}(\sigma)=\pi_\lambda(x)$, where $x\in A_\lambda\setminus (\tau^*_\lambda\cup \{-\frac{\lambda^2}{4}\})
$ is a dynamical marking of $\sigma\in \mathcal{R}^{\lambda}$ relative to $\mathrm{P}_\lambda$. Note that in cases where there are two dynamical markings, $x$ and $I_\lambda(x)$, they get identified by $\pi_\lambda$, whence the map $\chi^{\lambda}$ is well-defined.

The construction of dynamical marker maps $\psi$ via linearizing coordinates shows that $\chi^\lambda$ is holomorphic, and results from \cite{gold90} imply that $\chi^\lambda$ is surjective, whence an isomorphism. This isomorphism is canonical, in the sense that it is \emph{the} isomorphism on $\mathcal{R}^\lambda$, which records dynamical markings rel $\mathrm{P}_{\lambda}$. This discussion and the results of \cite[Lemma 3.1 and Thm. 3.3]{gold90} (in our setting) are collected in the following proposition.

\begin{prop}
\label{th:gk}
  The map $\chi^{\lambda}:\mathcal{R}^{\lambda}\to \Delta^{\lambda *}$
  is an isomorphism, which records dynamical markings, that is $x\in \pi_\lambda^{-1}(\chi^{\lambda}(\sigma))$ is a dynamical marking of $\sigma$ rel $\mathrm{P}_{\lambda}$. 
  Moreover, the point $\pi_{\lambda}(-\lambda^2/4)$ corresponds to $\infty\in\CC$, in the sense that $(\lambda,\sigma)\to(\lambda,\infty)$ if and only if $\chi^{\lambda}(\sigma)\to \pi_{\lambda}(-\lambda^2/4)\in \Delta^\lambda$.
\end{prop}

Since dynamical marker maps are composed of linearizing coordinates, that also vary analytically with the parameter $\lambda$, and $\Delta^\lambda$ is an analytically varying family of Riemann surfaces, the map $\chi^{\lambda}$ is analytic in $\lambda$.

Note that even though the construction is based on the
$p/q$-star $\Sigma_\lambda$ for the basin $A_\lambda$, the resulting model space $\Delta^{\lambda *}$ and isomorphism $\chi^{\lambda}$ do not depend on $p/q$. In fact, $\Delta^{\lambda}$ is canonically isomorphic to the quotient obtained by identifying only points on $\partial \mathcal{U}_\lambda$, i.e.  
$(A_\lambda\setminus \mathcal{U}_\lambda)/_{\sim_{\lambda} \text{ on } \partial \mathcal{U}_\lambda}$, where the isomorphism  
is induced by the injection of $A_\lambda\setminus \mathcal{U}_\lambda$ into $A_\lambda\setminus \tau^*_\lambda. $
And, since the dynamical marking $x$ does not depend on $p/q$ (only the domain of the marker map $\psi$ does), the map $\chi^{\lambda}$ can be given without reference to $p/q$, as long as a dynamical marking $x$ or $I_\lambda(x)$ in $A_\lambda\setminus \mathcal{U}_\lambda$ is chosen for its definition.

Because of this, and in accordance with our general convention for notation, we have suppressed the dependence on $p/q$ in the notation.

We now proceed to prove the needed properties of stars in infinitely connected quadratic basins, resulting in Proposition \ref{prop:eigenvalueconv} and its Corollaries \ref{cor:fatcycles} and \ref{cor:extracycles}.

\subsection{Characterization of landing points of wires}\label{sect:landingwire}
As usual, let $\lambda\in\D_{-\omega}$, $\sigma\in \mathcal{R}^\lambda$ and $g=g_{\ls}\in{\bfnf}_{\lambda,\sigma}$. Suppose that $x\in 
A_{\lambda}\setminus (\tau^*_\lambda\cup \{-\frac{\lambda^2}{4}\})
$ is a dynamical marking rel $\mathrm{P}_{\lambda}$ of $g$, with corresponding dynamical marker map $\psi:U\to V$, and $v_1$, $v_2$ a corresponding dynamical labelling of the critical values, i.e. $v_1$ is visible, $v_1=\psi(-\lambda^2/4)$ and $v_2=\psi(x)$.

A $q$-cycle of wires $\gamma$ of $g$ separates $V$, the range of the dynamical marker map $\psi$, into $q$ simply connected components. We will say that a cycle of wires $\gamma$ is \emph{non-separating} if $v_1$ and $v_2$ are in the same component of $V\setminus \gamma$ and \emph{separating} otherwise. There is always at most one non-separating $q$-cycle of wires.

\begin{lemma}\label{lemma:qwirecycle}
For $\lambda\in\D_{-\omega}$ and $\sigma\in \mathcal{R}^\lambda$ let $g=g_{\ls}\in{\bfnf}_{\lambda,\sigma}$.
A cycle of wires $\gamma$ of $g$ lands on a fixed point if and only if the cycle $\gamma$ is non-separating. 
\end{lemma}
The proof will use the Denjoy-Wolff Theorem, which we recall here: 
\begin{denjoy*}
Let $f:\D\to\D$ be a holomorphic map, not conjugate to $z\mapsto \e^{i\theta}z$, then there exists $p\in \overline\D$ such that $f^n(z)\to p$, as $n\to\infty$, for all $z\in \D$. In fact, the convergence is uniform on compact subsets of $\D$.
\end{denjoy*}

\begin{proof}[Proof of Lemma \ref{lemma:qwirecycle}]
By Lemma \ref{lemma:wiresland} any $q$-cycle of wires of $g$ lands, either on a fixed point or a $q$-cycle of $g$. 

First assume that the $q$-cycle of wires $\gamma$ of $g$ lands on a fixed point $\zeta$. Then $\CC\setminus\overline{\gamma}$ consists of $q$ topological disks, denoted $D^0, ...,  D^{q-1}$ and  named counterclockwise according to their cyclical order around $z_0$, so that the critical point $c_1$ is in $D^0$.
The critical value $v_1=g(c_1)$ is in $D^p$, since the wires are $p/q$ rotated by $g$. The pre-image $g^{-1}(\gamma)$ has two components, $\gamma$ itself and $\gamma'$, where $\overline{\gamma'}\cap \overline{\gamma}=\emptyset$ and $\gamma'$ lands on the pre-image $\zeta'\neq\zeta$ of $\zeta$. The leaves in $\gamma'$ emanate from $\zeta'$ and end on $z_0'$, the other pre-image of $z_0$.

Hence, the set $\CC\setminus g^{-1}(\overline{\gamma})=\CC\setminus (\overline{\gamma}\cup \overline{\gamma'})$ consists of $2(q-1)$ topological disks and one doubly connected component $A(\gamma, \gamma')\subset D^0$, which contains $c_1$. Since $g: A(\gamma, \gamma')\to D^p$ is a proper, branched covering of degree 2, it follows from the Riemann-Hurwitz Formula, that the number of branch points in $A(\gamma, \gamma')$ is 2, whence $v_2$ is also in $D^p$. But then $\gamma$ is non-separating.

Now assume that the cycle of wires $\gamma$ is non-separating. Suppose that $\psi:U\to V$ is a dynamical marker map for $g$, and $v_1$, $v_2$ a corresponding dynamical labelling of the critical values. We make an alteration of the co-domain $V$, by removing the ends of the strips of the star $\Sigma_{\ls}$, and all their pre-images, in order to construct a subset which is compactly contained in the basin $A_{\ls}$.
Let $l\geq 0$ be minimal so that $g^l(v_2)\in \Sigma_{\ls}$, as in the construction of the domain and range of $\psi$, and choose
$R>\max\{|\lambda^l\cdot\phi_{\ls}(v_2)|,|\lambda|\}$,
then the subset
\[V(R):=g^{-l}\left(\Sigma_{\ls}\cap\phi_{\ls}^{-1}(\D_R)\right)\subset V\]
is connected, simply connected, forward invariant under $g$ and $v_1,v_2\in V(R)$. The closure $\overline{V(R)}$ is forward invariant under $g$ as well, and by choosing $R<\max\{|\lambda^{l-1}\cdot\phi_{\ls}(v_2)|,1|\}$ if necessary, so that there is at most one critical point in $\overline {V(R)}$, we make sure that $\overline {V(R)}$ is also simply connected.

Hence the complement $W=\CC\setminus \overline{V(R)}$ is open, simply connected and forward invariant under branches of the inverse $g^{-1}$.
Since $v_1,v_2 \in V(R)$, it follows from the Riemann-Hurwitz Formula that the pre-image $g^{-1}(\overline{V(R)})$ is doubly connected, whence the complement $\CC\setminus g^{-1}(\overline{V(R)})=g^{-1}(W)\subset W$ consists of two open, simply connected components, $W_1,W_2\subset W$, each mapped isomorphically to $W$ by $g$. Note that the ends and landing points of the cycle of wires $\gamma$ must be contained in 
$g^{-1}(W)$, by the compactness of $V(R)$ relative to $A_{\ls}$. Moreover:
\paragraph{Claim.}  All landing points of $\gamma$ are in the same component of $g^{-1}(W)$, this component will be denoted $W_1$.
\begin{proof}[Proof of claim]
The cycle of wires $\gamma$ is non-separating, so $v_1$ and $v_2$ can be connected by a curve $\ell$ in $V(R)$, which does not cross $\gamma$. For simplicity, consider a continuous curve that does not self-intersect, then the preimage $g^{-1}(\ell)$ is a simple closed curve in $g^{-1}(V(R))$ containing both critical points $c_1$ and $c_2$, and separating $\CC$ into two components, each compactly containing a component of $g^{-1}(W)$. Since the curve $\ell$ does not cross $\gamma$, neither does its pre-image, whence all landing points of $\gamma$ are in the same component of $g^{-1}(W)$.
 \end{proof}
 Let $f_1$ denote the branch of the inverse of $g$ from $W$ to $W_1$, i.e. $f_1: W \overset{\cong}\to W_1\subset W$. Since all landing points of $\gamma$ are in $W_1$, they constitute a periodic cycle for $f_1$. Let $\Phi:W\to\D$ denote a Riemann map for $W$, and consider the conjugate map $f_\Phi=\Phi\circ f_1\circ\Phi^{-1}:\D\to\D$. Since $W_1$ is a proper subset of $W$, it follows that $f_\Phi(\D)\subsetneq \D$, whence $f_\Phi$ is not conjugate to a rotation. But then $f_1$ can have no periodic cycles in $W$ of period $n>1$, since that would correspond to periodic cycles for $f_\Phi$ in $\D$ of period $n>1$, in contradiction to the Denjoy-Wolff Theorem. So the wires in $\gamma$ land together on a fixed point $\zeta$ of $g$. 
\end{proof}
\begin{lemma}\label{lemma:distinctcyclesdistinctlanding}
For $\lambda\in\D_{-\omega}$ and $\sigma\in \mathcal{R}^\lambda$ let $g=g_{\ls}\in{\bfnf}_{\lambda,\sigma}$.
Two distinct cycles of wires $\gamma\neq \gamma'$ of $g$ land on distinct periodic cycles of $g$.
\end{lemma}
\begin{proof}
Recall that $\tau_i$ denotes the twig of $\Sigma_{\ls}$ associated to the critical value $v_i$, when it is visible. When there are two distinct cycles of wires ($k=2$), then there are two visible critical values and two twigs with $\tau_1 \cap \tau_2=\{z_0\}$.
There is always at most one non-separating cycle of wires, whence from Lemma \ref{lemma:qwirecycle} at most one cycle can land at a fixed point. Therefore, the only case to check is when both cycles of wires land at genuine $q$-cycles of $g$. If $\gamma$ and $\gamma'$ land together on the q-cycle $\langle z\rangle$ of $g$, then the wires must land in pairs. Consider the union of branches $\overline{\tau_1^p}\cup \overline{\tau_2^p}$, which contains $v_1$, $v_2$ and $z_0$. The pre-image $g^{-1}(\overline{\tau_1^p}\cup \overline{\tau_2^p})$ contains a simple closed curve, containing the fixed point $z_0$, its other pre-image, and branches $\tau_1^0$ and $\tau_2^0$. The complement in $\CC$ of this simple closed curve contains $\overline{\gamma}\setminus \{z_0\}$ and $\overline{\gamma'}\setminus \{z_0\}$. By construction there is always an odd number of wires between any branch from $\tau_1$ and any branch from $\tau_2$ (which is immediately seen for the lines $\wtt_{i}^j$ and $\wtga^j$ in the universal covering), whence there is an odd number of wires in each component of this complement. But then the wires can't land in pairs, whence the two cycles of wires $\gamma$ and $\gamma'$ of $g$ land on distinct periodic cycles of $g$. 
\end{proof}

\subsection{Modulus estimates and subhorocyclic convergence}\label{sect:modulus}
The distance between adjacent lines in $\wtt_{1}/L$ is $\frac{2\pi\sin\theta}{q|L|}$, see Figure \ref{fig_visheight}. For $k=2$ we consider also the distance between $\wtt_{1}/L$ and $\wtt_{2}/L$.
This number depends on $p/q$, but neither on the representative $g\in {\bfnf}_{\lambda,\sigma}$, nor on the normalization of $\phi_{\ls}$ or the choice of labeling of critical values. We will also call it the \emph{visible height} of $v_2$ rel $\mathrm{P}$ (or of the dynamical marking $x$) for $g\in {\bfnf}_{\lambda,\sigma}$ (or simply for $(\lambda,\sigma)$) and denote it by $H=H_{p/q}(\ls)$, see Figure \ref{fig_visheight}. For $k=1$ we set $H=0$. Notice that $H=d(\wtt_{1}/L,\frac{\log\circ\phi_{\ls}(v_2)}{L})$.
\begin{figure}
\begin{center}
     \includegraphics[width=12cm]{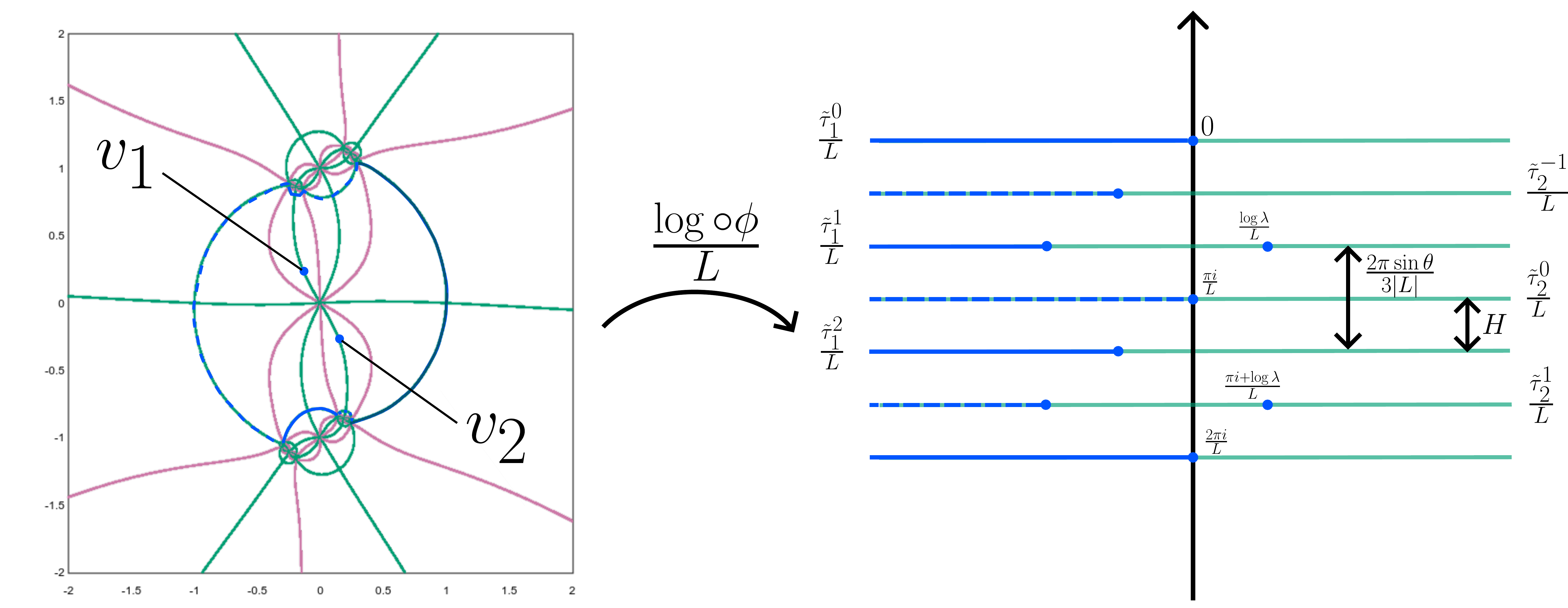}
\caption{Illustrates the visible height $H$ for $G_{\lambda,0}$ with $\lambda=\frac{1}{2}\omega_{1/3}$. This map has $k=2$, that is, its star has two cycles of wires, both landing on $q$-cycles and both wide ($H=\frac{\pi\sin\theta}{3|L|}$).} 
\label{fig_visheight}
\end{center}
\end{figure}

\paragraph{\emph{Associated} modulus for a cycle of wires $\gamma$.}
A wire $\gamma^j$ in a cycle of wires $\gamma$ belongs to the strip $U^j$ of the star. The quotient $U^j/g_{\ls}^q$ is conformally isomorphic to the cylinder $C=\widetilde U^j/(z\mapsto z+L)$, whence it has modulus $\operatorname{mod}(U^j/g_{\ls}^q)=\operatorname{mod}(C)$. Strips in the same cycle, are mapped conformally to each other, so their quotients have the same moduli. We will say that a cycle of wires $\gamma$ has \emph{associated modulus} $m$ if $m=\operatorname{mod}(U^j/g_{\ls}^q)$. 

\begin{itemize}
\item When $k=1$ there is one cycle of wires, of associated modulus 
\[m=m(\lambda)=\frac{2\pi \sin\theta}{q|L|}=\frac{\pi}{q^2 r_\lambda},\]
where $\theta$ is the angle between $L$ and $2\pi i$. See also Equations (\ref{eqn_mlambdaattrpar}) and  (\ref{eqn:rlambda}).

\item When $k=2$ there are two cycles of wires, $\gamma_1$ and $\gamma_2$ of associated moduli $m_1$ and $m_2$, where $m_1+m_2 = m(\lambda)$. We will use the convention $m_1\geq m_2$. By the definition of visible height 
\[m_2=H\leq\frac{\pi\sin\theta}{q|L|},\]
and thus
\[m_1=m(\lambda)-H=\frac{2\pi \sin\theta}{q|L|}-H\geq\frac{\pi\sin\theta}{q|L|}.\]
\end{itemize}

\paragraph{A \emph{wide} $q$-cycle of wires.} A $q$-cycle of wires is called \emph{wide} if it has associated modulus: $m\geq\frac{\pi\sin\theta}{q|L|}$. Note that $g_{\ls}$  always has one wide $q$-cycle of wires, and two if and only if $k=2$ and $H=\frac{\pi\sin\theta}{q|L|}$, see Figure \ref{fig_visheight}.

Recall that $\gamma_\lambda$ denotes the $q$-cycle of wires, and $\tau_\lambda$ denotes the twig, of the $p/q$-star $\Sigma_\lambda$ for $\mathrm{P}_{\lambda}$ and let $U_\lambda^j$ denote the jth strip of $\Sigma_\lambda$. Also recall that $S^p_\lambda$ denotes the critical value sector of $A_\lambda$: the connected component of $A_\lambda\setminus \gamma_\lambda$ that contains $-\lambda^2/4$, see Figure \ref{critvaluesector}.
\begin{lemma}\label{lemma:nonsepwirecycle}
For $\lambda\in\D_{-\omega}$ and $\sigma\in \mathcal{R}^\lambda$ let $g=g_{\ls}\in{\bfnf}_{\lambda,\sigma}$. Suppose $x\in 
A_{\lambda}\setminus (\tau^*_\lambda\cup \{-\frac{\lambda^2}{4}\})
$ is a dynamical marking of $g$ rel $\mathrm{P}_{\lambda}$ and let $v_2=\psi(x)$ denote the corresponding critical value of $g$. Then,
\begin{itemize}
\item there exists a non-separating $q$-cycle of wires $\gamma$ of $g$ if and only if \\ $x\in S^p_\lambda\cup U^p_\lambda\cup U_\lambda^{p-1}$, and 
\item there exists a wide non-separating $q$-cycle of wires $\gamma$ of $g$ if and only if $x\in \overline{S^p_\lambda}$.
\end{itemize}
\end{lemma}

\begin{proof}
Recall that $\tau_i$ is the twig of $\Sigma_{\ls}$ associated to the critical value $v_i$, when it is visible. It follows from the definition of $\psi$ that $\psi^{-1}(\tau_1)\subseteq \tau_\lambda$, with equality when $x\notin \tau_\lambda$.

If $v_2\notin\Sigma_{\lambda,\sigma}$ or $v_2\in\tau_1$ then $k=1$, i.e. there is one cycle of wires $\gamma$ of $g$, which is then automatically wide. In this case $\psi(\gamma_{\lambda})=\gamma$ and $\gamma$ is non-separating if and only if $x\in S^p_\lambda$. 

If $v_2\in\Sigma_{\lambda,\sigma}\setminus \tau_1$, then $k=2$ and there exists a non-separating cycle of wires if and only if branches $\tau_{2}^p$ and $\tau_{1}^p$ are adjacent to each other, with respect to the cyclic ordering. This is the case if and only if $x\in U^p_\lambda\cup U_\lambda^{p-1}$, since $\psi^{-1}(\tau_1)=\tau_\lambda$. Further, there exists a wide non-separating cycle of wires if and only if $d(\wtt_{1}^p/L, \wtt_{2}^p/L)=H$, which in turn is equivalent to $x\in \overline{S^p_\lambda}$.
\end{proof}

If a cycle of wires $\gamma$ of $g$ is non-separating let $\zeta=g(\zeta)$ denote the landing fixed point and $\mu$ its eigenvalue.
If $\gamma$ is separating let $\langle z\rangle$ denote the landing $q$-cycle of $g$ and $\rho$ its eigenvalue. Recall that $r_\lambda=r_{\lambda,p/q}$ is the radius of the circle through $\log\lambda$, and tangent to the imaginary axis at $2\pi ip/q$ (cf. Equation \ref{eqn:rlambda}). Similarly, let $r_\mu=r_{\mu,-p/q}$ denote the radius of the circle through $\log\mu$, and tangent to the imaginary axis at $-2\pi ip/q$, and let $r_\rho=r_{\rho,0}$ denote the radius of the circle through $\log\rho$, and tangent to the imaginary axis at 0, for specified choices of $\log\mu$ and $\log\rho$ respectively.

\begin{lemma}\label{lemma:wiremodulus}
For $\lambda\in\D_{-\omega}$ and $\sigma\in \C$ let $g=g_{\lambda,\sigma}\in{\bfnf}_{\lambda,\sigma}$. Let $\gamma=\gamma_{\ls}$ be a cycle of wires of $g$ of associated modulus $m=m(\gamma)$. 

\begin{enumerate}
\item If the landing cycle of $\gamma$ is a repelling fixed point $\zeta$ of eigenvalue $\mu$, then there is a choice of $\log\mu$ so
\[m\leq\frac{\pi}{q^2 r_\mu}.\]
\item If the landing cycle of $\gamma$ is a repelling $q$-cycle $\langle z\rangle$ of eigenvalue $\rho$, then there is a choice of $\log\rho$ so
\[m\leq\frac{\pi}{r_\rho}.\]
 
\end{enumerate}
\end{lemma}
Note that Lemma \ref{lemma:wiremodulus} also holds in the case of simply connected basins ($\sigma\in M^\lambda$). 
The proof is an application of Bers inequality. We recall it here, for the benefit of the reader, inspired by the presentation in \cite{mil99}.

Consider a flat torus $\T=\C/\Lambda$, where $\Lambda\subset\C$ is a two-dimensional lattice, equipped with the induced Euclidean metric from $\C$. Let $A\subset \T$ be an embedded annulus. The central curve of $A$ lifts by the universal covering map $\C\to\T$, to a curve segment which joins a point $z_0\in\C$ to a translate $z_0+w$ by the lattice element $w\in\Lambda$. The lattice element $w$ is called the \emph{winding number} of $A$ in $\T$, and $A\subset \T$ is called an \emph{essentially embedded annulus} if $w\neq 0$.

\begin{bers*}
If the flat torus $\T=\C/\Lambda$ contains several pairwise disjoint, essentially embedded annuli $A_i$, then the annuli have the same winding number $w\in\Lambda$, and
\[\sum\operatorname{mod}(A_i)\leq \frac{\operatorname{area(\T)}}{|w|^2}.\] 
\end{bers*}

\begin{proof}[Proof of Lemma \ref{lemma:wiremodulus}]
For both cases, we consider quotient tori for the dynamics of $g^q$, at the relevant repelling (fixed or)  periodic points. Let $\gamma^j$ denote a wire in the cycle of wires $\gamma$ and let $U^j$ denote the strip containing $\gamma^j$. Since wires and strips of the star $\Sigma_{\ls}$ are forward invariant under $g^q$, wires $\gamma^j$ descend to closed curves and strips $U^j$ descend to essentially embedded annuli in the quotient tori. 

For case 1. consider the quotient torus $T_{\mu^q}$, for the dynamics of $g^q$, at the landing fixed point $\zeta$ with eigenvalue $\mu$ for $g$. The torus $T_{\mu^q}$ is conformally isomorphic to $\C/\Lambda$, where $\Lambda$ is the lattice generated by $2\pi i$ and $\log\mu^q$, for some choice of $\log\mu^q$.

Since wires $\gamma^j$ in $\gamma$ are $p/q$-rotated around $z_0$, they are $-p/q$-rotated around $\zeta$. Let $\phi_{rep}$ denote a linearizing coordinate for $g$ at the repelling fixed point $\zeta$, and let $\check{\gamma}^j$ denote a lift of $\phi_{rep}(\gamma^j)$ to the exponential function, where lifts are ordered horizontally so $\check{\gamma}^j$ is below $\check{\gamma}^i$ for $j<i$. The complete fiber of $\phi_{rep}(\gamma^j)$ (and the lift of $\gamma^j/g^q$ to the universal cover of the torus $T_{\mu^q}$) is $\{\check{\gamma}^{j+nq}:=\check{\gamma}^{j}+n2\pi i, \, n\in\Z\}$. Choose $\log\mu$ so that $\check{\gamma}^j+\log\mu=\check{\gamma}^{j-p}$, then $M=q\log\mu+p2\pi i$ is the winding number of the annuli $A_j=U^j/g^q$ in $T_{\mu^q}$ and our choice of logarithm of $\mu^q$. Moreover, since the strips $U^j$ are all pairwise disjoint, so are the annuli $A_j$. 

Thus Bers inequality yields: 
\[qm=\sum_{j\in \Z_{q}} \operatorname{mod}(U^{j}/g^q)\leq \frac{\operatorname{area(T_{\mu^q})}}{|M|^2}.\] 
Rewriting this in terms of the horocyclic radius $r_\mu=\frac{|M|}{2q \sin(\theta_M)}$, where $\theta_M$ is the angle between $M$ and $2\pi i$, proves case 1.:
\[qm=\sum_{j\in \Z_{q}} \operatorname{mod}(U^{j}/g^q)\leq \frac{\operatorname{area(T)}}{|M|^2}=\frac{2\pi\sin(\theta_M)}{|M|}=\frac{\pi}{q\: r_\mu}.\]

For case 2. consider similarly the quotient torus $T_{\rho}$, for the dynamics of $g^q$, at one of the periodic points of $g$ in the $q$-cycle $\langle z \rangle$. The torus $T_{\rho}$ is conformally isomorphic to $\C/\Lambda$, where $\Lambda$ is the lattice generated by $2\pi i$ and a logarithm of $\rho$. We choose the logarithm $R=\log\rho=\Log\rho+n2 \pi i$, where $n\in\Z$ is chosen so that $R$ is the winding number of $U^j/g^q$ in $T_{\rho}$.

Thus Bers inequality yields
\[m=\operatorname{mod}(U^{j}/g^q)\leq \frac{\operatorname{area(T_{\rho})}}{|R|^2}=\frac{2\pi\sin(\theta_{R})}{|R|}=\frac{\pi}{r_{\rho}},\] 
where $\theta_R$ is the angle between $R$ and $2\pi i$. 
\end{proof}

For $\lambda\in\D_{-\omega}$ and $\sigma \in \mathcal{R}^{\lambda}$ any   $g=g_{\lambda,\sigma}\in{\bfnf}_{\lambda,\sigma}$ has an attracting fixed point with two critical values in its attracting basin. Hence by the Fatou-Shishikura inequality all other cycles are repelling (see for example \cite{mil99}). In particular, all landing cycles are repelling, whence Lemma \ref{lemma:wiremodulus} applies.

The following Proposition is an extension of Theorem B from \cite{pet99} to the case of non-simply connected basins.
\begin{prop}\label{prop:eigenvalueconv}
For $\lambda\in\D_{-\omega}$ and $\sigma \in \mathcal{R}^{\lambda}$ let  $g=g_{\lambda,\sigma}\in{\bfnf}_{\lambda,\sigma}$. Suppose $x\in 
A_{\lambda}\setminus (\tau^*_\lambda\cup \{-\frac{\lambda^2}{4}\})
$ is a dynamical marking of $g$ rel $\mathrm{P}_{\lambda}$ and $H$ is the visible height of $x$. 

\begin{enumerate}
\item If $x\in \overline{S^p_{\lambda}}$ then 
\[
r_{\mu}\leq k \, r_{\lambda},
\]
where $k$ is the number of cycles of wires.
More precisely for $k=2$
\[
\frac{\pi}{q^2 r_{\mu}}\geq \frac{\pi}{q^2 r_{\lambda}}-H.
\]

\item 
If $x\notin \overline{S^p_{\lambda}}$ then 
\[
r_{\rho}\leq k \, q^2 r_{\lambda},
\]
where $k$ is the number of cycles of wires.
More precisely for $k=2$
\[
\frac{\pi}{ r_{\rho}}\geq \frac{\pi}{q^2 r_{\lambda}}-H.
\]
\end{enumerate}
For $k=2$ ($x\in 
\Sigma_{\lambda}\setminus \tau_\lambda$) we further have:
\begin{enumerate}
\item[3.] If $x\in U^p_{\lambda}\cup U^{p-1}_{\lambda}$ then
\[r_{\mu}\leq \frac{\pi}{q^2H} \quad \text{and} \quad r_{\rho} \leq \frac{\pi}{H}.\]
\item[4.]
For $x\in \Sigma_{\lambda}\setminus U^p_{\lambda}\cup U^{p-1}_{\lambda}$ let $\rho$, $\rho'$ denote the eigenvalues  of the landing $q$-cycles $\langle z\rangle\neq \langle z' \rangle$, of the two (separating) cycles of wires of $g$. Then 
\[ r_{\rho}\leq \frac{\pi}{H} \quad \text{and} \quad r_{\rho'}\leq \frac{\pi}{H}.\]
\end{enumerate}
\end{prop}
Note that in cases 3. and 4. sharper inequalities hold for the eigenvalue of the landing cycle pertaining to the wide cycle of wires, these are contained in the situation $k=2$ of cases 1. and 2. Proposition \ref{prop:eigenvalueconv} has the following immediate corollaries.
\begin{cor}\label{cor:fatcycles}
For $\lambda_n\in\D_{-\omega}$ and $\sigma_n \in \mathcal{R}^{\lambda_n}$ let  $g_n=g_{\lambda_n,\sigma_n}\in{\bfnf}_{\lambda_n,\sigma_n}$. Suppose $x_n\in 
A_{\lambda_n}\setminus (\tau^*_{\lambda_n}\cup \{-\frac{\lambda_n^2}{4}\})
$ is a dynamical marking of $g_n$ rel $\mathrm{P}_{\lambda_n}$ and let $v_n=(v_2)_n=\psi_n(x_n)$ denote the corresponding critical value of $g_n$. Let $\lambda_n\to\opq$ subhorocyclicly. 
\begin{enumerate}
\item If $x_n\in \overline{S^p_{\lambda_n}}$ then there exist repelling fixed points $\zeta_n=g_n(\zeta_n)$ with eigenvalues $\mu_n\to \omega_{-p/q}$ subhorocyclicly.

\item 
If $x_n\notin \overline{S^p_{\lambda_n}}$ then there exist repelling $q$-cycles $\langle z\rangle_n$ of $g_n$ with eigenvalues $\rho_n\to 1$ subhorocyclicly. 
\end{enumerate}
\end{cor}

\begin{cor}\label{cor:extracycles}
For $\lambda_n\in\D_{-\omega}$ and $\sigma_n \in \mathcal{R}^{\lambda_n}$ let  $g_n=g_{\lambda_n,\sigma_n}\in{\bfnf}_{\lambda_n,\sigma_n}$. Suppose $x_n\in 
\Sigma_{\lambda_n}\setminus (\tau^*_{\lambda_n}\cup \{-\frac{\lambda_n^2}{4}\})
$ is a dynamical marking of $g_n$ rel $\mathrm{P}_{\lambda_n}$ and let $v_n=(v_2)_n=\psi_n(x_n)$ denote the corresponding critical value of $g_n$. Let $\lambda_n\to\opq$ subhorocyclicly and suppose $H_n\to\infty$. 
\begin{enumerate}
\item If $x_n\in U^p_{\lambda_n}\cup U^{p-1}_{\lambda_n}$ then
there exist repelling fixed points $\zeta_n=g_n(\zeta_n)$ with eigenvalue $\mu_n\to \omega_{-p/q}$ subhorocyclicly and repelling $q$-cycles $\langle z\rangle_n$ of $g_n$ with eigenvalue $\rho_n\to 1$ subhorocyclicly.

\item
If $x_n\in U^j_{\lambda_n}$, $j\neq p, p-1$, then there exist 
distinct repelling $q$-cycles $\langle z\rangle_n\neq \langle z' \rangle_n$ of $g_n$ with eigenvalues $\rho_n\to 1$, $\rho'_n\to 1$ subhorocyclicly. 

\end{enumerate}
\end{cor}

\begin{proof}[Proof of Proposition \ref{prop:eigenvalueconv}]
For a $g$ as in the proposition, there is always at least one wide cycle of wires. We let $m_1$ denote the associated modulus of such a cycle:
\[m_1=m(\lambda)-H=\frac{\pi}{q^2r_\lambda}-H \geq \frac{\pi}{k q^2r_\lambda},\]
where the visible height $H=0$ if $k=1$ and $H\leq \frac{\pi}{2 q^2 r_{\lambda}}$ by definition.

Suppose $x\in \overline{S^p_{\lambda}}$. Then by Lemma \ref{lemma:nonsepwirecycle} there exists a wide non-separating cycle of wires.
 By Lemma \ref{lemma:qwirecycle} this cycle of wires lands on a fixed point $\zeta$ of $g$, whence by Lemma \ref{lemma:wiremodulus} the eigenvalue $\mu$ of $\zeta$ satisfies 
\[\frac{\pi}{k q^2r_\lambda} \leq \frac{\pi}{q^2r_\lambda}-H = m_1\leq\frac{\pi}{q^2 r_\mu},\]
which shows case 1.

Suppose $x\notin \overline{S^p_{\lambda}}$. Then by Lemma \ref{lemma:nonsepwirecycle} any wide cycle of wires is separating, and therefore lands on a $q$-cycle $\langle z\rangle$ of $g$ by Lemma \ref{lemma:qwirecycle}. 
By Lemma \ref{lemma:wiremodulus} the eigenvalue $\rho$ of $\langle z\rangle$ satisfies 
\[\frac{\pi}{k q^2r_\lambda} \leq \frac{\pi}{q^2r_\lambda}-H = m_1\leq \frac{\pi}{r_\rho},\]
which shows case 2.

Suppose $k=2$ and $x\in U^p_{\lambda}\cup U^{p-1}_{\lambda}$, then there are two cycles of wires. By Lemma \ref{lemma:nonsepwirecycle} there exists a non-separating cycle of wires, which lands on a fixed point $\zeta$ of $g$ by Lemma  \ref{lemma:qwirecycle}. The other cycle of wires is necessarily separating, and therefore lands on a $q$-cycle $\langle z\rangle$ of $g$ by Lemma \ref{lemma:qwirecycle}. This situation is illustrated in the bottom left picture in Figure \ref{fig_stars}. The associated modulus of either cycle is greater than or equal to $H$, whence by Lemma \ref{lemma:wiremodulus}
\[ H\leq\frac{\pi}{q^2 r_\mu} \quad \text{and} \quad H \leq \frac{\pi}{r_\rho}\]
which shows case 3. 

Suppose $k=2$ and $x\in \Sigma_{\lambda}\setminus U^p_{\lambda}\cup U^{p-1}_{\lambda}$, then there are two cycles of wires, both separating by Lemma \ref{lemma:nonsepwirecycle}, whence they land on distinct $q$-cycles of $g$ by Lemmas \ref{lemma:qwirecycle} and \ref{lemma:distinctcyclesdistinctlanding}, this situation is illustrated in the bottom right picture in Figure \ref{fig_stars}. Let $\rho$, $\rho'$ denote the eigenvalues  of the landing $q$-cycles $\langle z\rangle\neq \langle z' \rangle$. Then case 4. follows from Lemma \ref{lemma:wiremodulus} as in case 3.
\end{proof}

\begin{proof}[Proof of Corollaries \ref{cor:fatcycles} and \ref{cor:extracycles}]
Since 
\begin{eqnarray}
\nonumber \lambda_k\to\opq \text{ subhorocyclicly } & \Leftrightarrow & r_{\lambda_k}\to 0,\\
\nonumber \mu_k\to\omega_{-p/q} \text{ subhorocyclicly } & \Leftrightarrow & r_{\mu_k}\to 0,\\
\nonumber \rho_k\to 1 \text{ subhorocyclicly } & \Leftrightarrow & r_{\rho_k}\to 0,
\end{eqnarray}
the statements follow immediately from Proposition \ref{prop:eigenvalueconv}.
\end{proof}

\section{Dynamical markings rel $\mathrm{P}$}\label{sectdynmarkrelPproofs}
 In this section we will prove Propositions \ref{prop:markinglambdarelP} and \ref{prop:parabolicmarkings}, and show uniqueness of dynamical markings rel $\mathrm{P}$.

The parametrization $\Phi^\lambda: \mathcal{A}^*_{M} \to \mathcal{R}^{\lambda}$ from Proposition \ref{prop:markinglambdarelP}, and hence the map $\mathbf{\Phi}$ from Definition \ref{defParameter}, will be defined through a dynamical conjugacy between the quadratic polynomials $\mathrm{P}$
and $\mathrm{P}_{\lambda}$, defined on a suitable subset of the closure of the parabolic basin $\overline{A}$ into the attracting basin $A_\lambda$. 

Consider the subset $\mathcal{A}_{m(\lambda)}\subset \overline{A}$ defined in Equation (\ref{def:Xim}) and with $m(\lambda)=\frac{2\pi\sin\theta}{q|L|}$ or $m(\lambda)=m(\omega)=\infty$, for $\lambda=\omega$, as defined in Equation (\ref{eqn_mlambdaattrpar}). Recall that 
$\gamma_\lambda=\bigcup_{j\in\Z_q} \gamma^j_\lambda\cup \{0\}$ 
denotes the $q$-cycle of wires for the $p/q$-star $\Sigma_\lambda$ for $\mathrm{P}_{\lambda}$ and let $\Gamma= \bigcup_{n\in \N}\mathrm{P}^{-n}(\gamma_\lambda \setminus\{0\})$.

\begin{prop}\label{prop:maphlambda}
Suppose $\lambda\in \D_{-\omega}$. There exists a bijection 
\[h_\lambda: \mathcal{A}_{m(\lambda)}\to A_\lambda\setminus \Gamma,\]
with $h_\lambda(0)=0$ and $h_\lambda(-\omega/2)=-\lambda/2$, which is univalent on the interior of the domain, such that 
\begin{equation}\label{functeqnhlambda}
	h_\lambda\circ\mathrm{P} = \mathrm{P}_{\lambda}\circ h_\lambda,
\end{equation}
and so that for any $x\in \mathcal{A}^*_{m(\lambda)}$, the restriction of $h_\lambda$ to $\mathcal{U}_{m(\lambda)}(x)$ is a homeomorphism onto its image. The map $h_\lambda$ is unique.
Moreover, for every $x\in \mathcal{A}_{M}$, $h_\lambda$ depends analytically on the parameter $\lambda\in D_M$.
\end{prop}
\begin{figure}[htbp]
  \begin{center}
 \includegraphics[width=12.5cm]{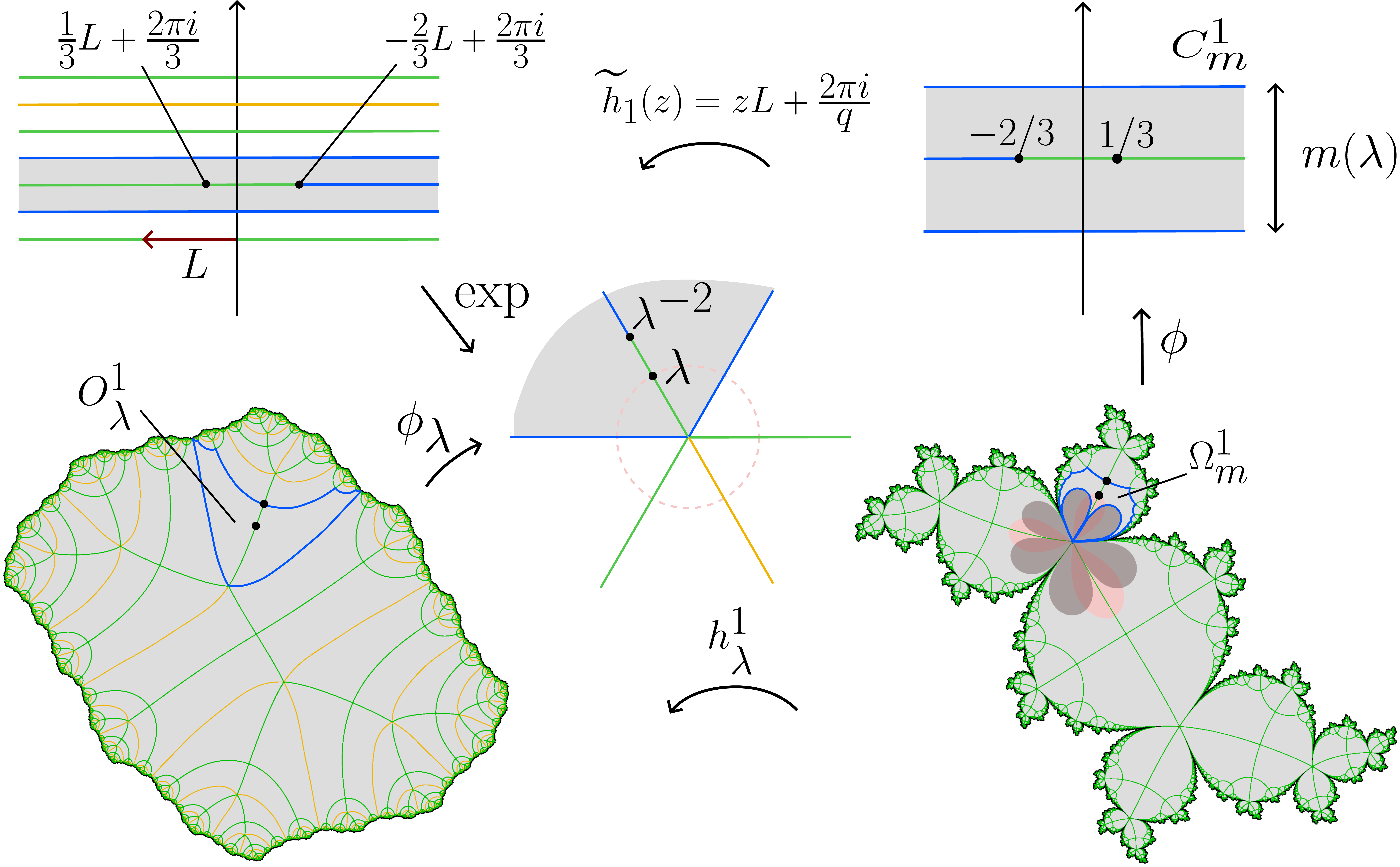}
  \end{center}
  \caption{An illustration of the definition of the map $h_\lambda$. The figure shows $h_\lambda^1$ for $p/q=1/3$. The domains $\Omega_m^1$ and $O_\lambda^1$ are highlighted with blue contours.  \label{fig:hlambda} }
\end{figure}
\begin{proof}[Proof of Proposition \ref{prop:maphlambda}]
The map $h_\lambda$ will first be defined as a map onto $\Sigma_{\lambda}\setminus \gamma_\lambda$, the reader is recommended to consult Figure \ref{fig:hlambda}.
Let the Fatou coordinate $\phi$ of $\mathrm{P}$ and the linearizing coordinate $\phi_\lambda$ of $\mathrm{P}_{\lambda}$ be normalized as described in Section \ref{results} and \ref{sect_stardef} respectively.
Recall that the linearizing coordinate $\phi_\lambda$ of $\mathrm{P}_{\lambda}$ is univalent on the star $\Sigma_{\lambda}$, we let $\phi^{-1}_\lambda$ denote the inverse of the restriction $\phi_\lambda|_{\Sigma_{\lambda}}$. 
The cycle of wires $\gamma_\lambda$ separates $\Sigma_\lambda$ into $q$ components, denote these $O^j_{\lambda}$, $j\in \Z_q$, numbered so that $\tau_\lambda^j\subset O^j_{\lambda}$. Also let 
$O_\lambda:=\bigcup_{j\in \Z_q} O^j_{\lambda}=\Sigma_{\lambda}\setminus \gamma_\lambda$ and $\widehat{O}_\lambda:=O_\lambda\cup \{0\}$.

Each $O^j_{\lambda}$ is forward invariant under $\mathrm{P}^q_{\lambda}$ and corresponds to a slit horizontal strip, $C^j_{m}$, of height $m(\lambda)$ in the normalized log-linearizing coordinate $\frac{\log\circ\phi_\lambda}{L}$, see Figure \ref{fig:hlambda}. 
To be precise:
\[C^j_{m}=\left\{z\in \C : |\Im(z)|< m(\lambda)/2\right\}\setminus \left\{z\in\R : z \leq \frac{-(q-\kappa(j))}{q}\right\},\]where $\kappa(j)=\frac{j}{p} \in\Z_q$ using the representative in $\{1, ..., q\}$ as in Equation (\ref{eqn_nj}). 
Define for $j\in\Z_q$ affine maps $\tilde h^j:\C\to\C$ by $\tilde h^j(z)=zL+ \frac{j}{q}2\pi i$. Then the map $\phi^{-1}_\lambda\circ \exp\circ \tilde h^j:C^j_{m}\to O^j_{\lambda}$ is an isomorphism that conjugates $z\mapsto z+1$ to $\mathrm{P}_{\lambda}^q$.

On the other hand, the parabolic basin of 0 for $\mathrm{P}$ has $q$ components, named $B^0,..., B^{q-1}$ (cf Section \ref{results} and Figure \ref{Fatrabbit}). Let $\Omega^j_{m}$ denote the component of $\phi^{-1}(C^j_{m})$ in $B^j\cap \mathcal{A}_{m(\lambda)}$, with 
 the parabolic fixed point 0 on the boundary. It is a (double) tile in the so-called checkerboard-tiling of the basin, with a pair of sepals of height $m(\lambda)/2$ removed (see Figure \ref{fig:hlambda}). Also, let
$\Omega_{m}=\bigcup_{j\in \Z_q} \Omega^j_{m}$ and $\widehat{\Omega}_{m}=\Omega_{m}\cup \{0\}$.
For every $j\in\Z_q$ the map 
\[h^j_{\lambda}:= \phi^{-1}_\lambda\circ \exp \circ \, \tilde h^j\circ \phi \: : \Omega^j_{m}\to O^j_{\lambda}\]
is an isomorphism, that conjugates $\mathrm{P}^q$ to $\mathrm{P}_{\lambda}^q$. Since the quotient of $\Omega^j_{m}$ by $\mathrm{P}^q$ is a cylinder of modulus $m(\lambda)$, the collection of maps $h^j_{\lambda}$ are the unique maps that embed the $q$ cylinders conformally in the quotient torus of 0 for $\mathrm{P}_{\lambda}^q$, as required. Thus the constructed map $h_\lambda$ will be unique.

Each $h^j_{\lambda}$ extends to 0 by setting $h^j_{\lambda}(0)=0$, 
such that $h^j_{\lambda}$ is continuous on every restricted domain $\mathcal{U}_{m(\lambda)}(x)\cap (\Omega^j_{m}\cup\{0\})$. We remark that $h^j_{\lambda}$ is not continuous at 0 in $\Omega^j_{m}\cup\{0\}$, since this domain has three accesses to 0, whereas the range $O^j_{\lambda}$ only has one (the other two map to landing points of wires), see again Figure \ref{fig:hlambda}.

Next, the $q$ maps $h^j_{\lambda}$ are pasted together to obtain 
\[h_{\lambda}: \widehat{\Omega}_{m}\to \widehat{O}_\lambda,\]
which conjugates $\mathrm{P}$ to $\mathrm{P}_{\lambda}$.
By successive lifting of $h_\lambda\circ\mathrm{P}$ to $\mathrm{P}_{\lambda}$, each time choosing the lift that coincides with $h_\lambda$ in the previous domain, $h_{\lambda}$ extends to a bijection:
\[h_{\lambda}: \mathcal{A}_{m(\lambda)} \to A_\lambda\setminus\Gamma.\]
The domain $\mathcal{A}_{m(\lambda)}$ of $h_\lambda$ is connected, but neither open nor closed, $h_\lambda$ obeys the functional equation (\ref{functeqnhlambda}), is holomorphic in $\inter(\mathcal{A}_{m(\lambda)})$
and continuous on any $\mathcal{U}_{m(\lambda)}(x)$.
Moreover, by the properties of Fatou and linearizing coordinates, $h_\lambda$ depends analytically on the parameter $\lambda\in D_M$.
\end{proof}
Note that $h_\lambda$ is not continuous in $\mathcal{A}_{m(\lambda)}$, however, it follows from the construction in the proof that the inverse $h_\lambda^{-1}$ is continuous onto $\mathcal{A}_{m(\lambda)}$.

\subsection{Adjusted domains for dynamical markings rel $\mathrm{P}$}\label{markingRbyPadjustdomain}
Suppose $\lambda\in \D_{-\omega}\cup\{\omega\}$ and $\sigma \in \mathcal{R}^\lambda$ and let $g\in \Gamma_{\lambda,\sigma}$ with a fixed point at $z_0$ of eigenvalue $\lambda$.
Recall that a dynamical marking of $g$ rel $\mathrm{P}$ is a pair $(x,\psi)$ where $x\in \mathcal{A}^*_{m(\lambda)}$ and $\psi: \mathcal{U}\to V$ is a homeomorphism, so that $\psi\circ \mathrm{P}=g\circ \psi$, where $\mathcal{U}=\mathcal{U}_{m(\lambda)}(x)=\mathrm{P}^{-n}(\widehat{\mathcal{P}}_{0})\cap \mathcal{A}_{m(\lambda)}$ for 
$n=n(x)$ minimal so that $\mathrm{P}^{n}(x)\in \widehat{\mathcal{P}}_{0}$. Moreover, $\psi$ is holomorphic on $\inter(\mathcal{U})$,  and $v_1=\psi(-\frac{\omega^2}{4})$ and $v_2=\psi(x)$ are the critical values of $g$. 

Let $\H_{-n/q}$ denote the right halfplane $\{z=x+iy\in\C : \, x>-n/q\}$, then $\phi(\inter(\mathcal{U}))$ is the horizontal band in $\H_{-n/q}$ of height $m(\lambda)$, and centered on $\R$.

We will slightly adjust the domains of dynamical marker maps. As mentioned before, the choice of domain for the definition of dynamical marker maps is not canonical, the one we have for Definition \ref{def:dynmarkingpara} is chosen so that marker maps become homeomorphisms and for simplicity of the presentation. The 
adjustment matches the domains with the domains of the dynamical marker maps rel $\mathrm{P}_\lambda$ and is otherwise inconsequential. 

Let $\widehat\Omega_{m(\lambda)}\subset \mathcal{A}_{m(\lambda)}$ as in the definition of $h_\lambda$ in the proof of Proposition \ref{prop:maphlambda} (see also Figure \ref{fig:hlambda}). For $\lambda=\omega$ we will use the notation $\Omega^j$, $\Omega$ and $\widehat\Omega$, instead of $\Omega^j_\infty$, $\Omega_\infty$ and $\widehat\Omega_\infty$, for the tiles at 0 in $\mathcal{A}$ without sepals removed. Let $0\leq l\leq n$ be minimal so that $\mathrm{P}^{l}(x)\in \widehat\Omega_{m(\lambda)}$ and redefine $\mathcal{U}$ so that $\mathcal{U}=\mathcal{U}_{m(\lambda)}(x):=\mathrm{P}^{-n}(\widehat{\mathcal{P}}_{0})\cap \mathrm{P}^{-l}( \widehat\Omega_{m(\lambda)})$, see Figure \ref{fig_reducedD} for an example. 
\begin{figure}[htbp]
  \begin{center}
 \includegraphics[width=12cm]{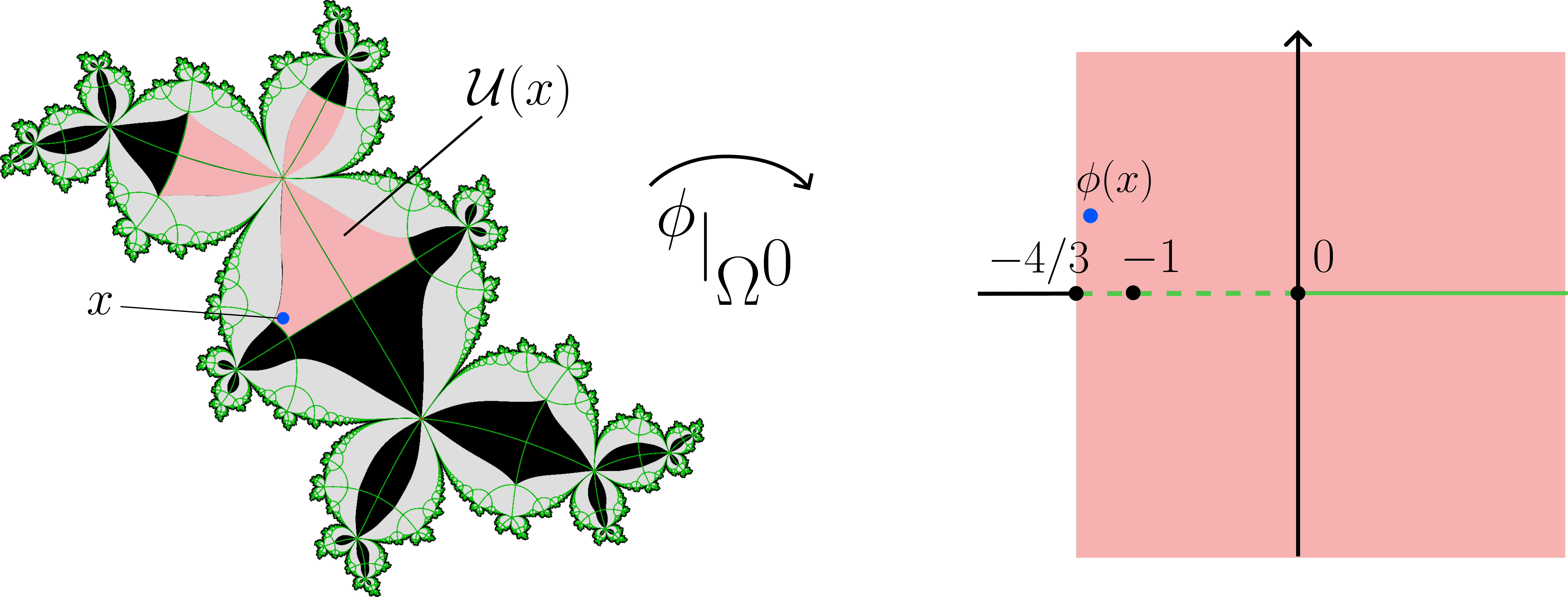}
  \end{center}
  \caption{An example of a reduced domain with $n=4$ and $l=0$. The original domain $\mathrm{P}^{-4}(\widehat{\mathcal{P}}_{0})$ for the marker map $\psi$ is shown in black and the reduced domain $\mathcal{U}(x)=\mathrm{P}^{-4}(\widehat{\mathcal{P}}_{0})\cap \widehat\Omega$ is shown in pink. The figure also shows the image of $\mathcal{U}(x)\cap \Omega^0$ in the Fatou coordinate.\label{fig_reducedD} }
\end{figure}
Then $\phi(\inter(\mathcal{U}))$ is now contained in the band in $\H_{-n/q}$ of height $m(\lambda)$ and centered on $\R$.

With this redefinition of the domain we retain its important properties: by construction $\mathcal{U}$ is forward invariant and connected, but neither open nor closed and  $0, -\frac{\omega^2}{4}, x\in \mathcal{U}$.  It follows from the properties of $\mathcal{U}$ and $\psi$, that $\psi(0)=z_0$, $V=\psi(\mathcal{U})$ is neither open nor closed, and that $V$ is connected, forward invariant and
\[
V\subseteq\begin{cases}
	A_{\lambda,\sigma} & \text{ for } \lambda\in \D_{-\omega}\\
	A_{\lambda,\sigma} \cup \bigcup_{k=0}^{l}g^{-k}(z_0) & \text{ for } \lambda= \omega,
\end{cases}
\]
where $A_{\lambda,\sigma}$ is the attracting or parabolic basin of $g$ at $z_0$, and with $l\leq n=n(x)$ from the definition of $\mathcal{U}$.
\subsection{Parametrization of $\mathcal{R}$ by dynamical markings rel $\mathrm{P}$}\label{markingRbyP}
We consider again $(\lambda,\sigma)\in \mathcal{R}$, in particular let $\lambda\in \D_{-\omega}$ and $\sigma \in \mathcal{R}^\lambda$ and let $g\in \Gamma_{\lambda,\sigma}$ with a fixed point at $z_0$ of eigenvalue $\lambda$.

\begin{proof}[Proof of Proposition \ref{prop:markinglambdarelP}]
For $M>1/q^2$ consider $(\lambda,x)\in D_M\times \mathcal{A}^*_{M}$ and set $\sigma=(\chi^\lambda)^{-1}\circ \pi_\lambda\circ h_\lambda(x)\in \mathcal{R}^{\lambda}$. Then, by the properties of the model of $\mathcal{R}^{\lambda}$, see Proposition \ref{th:gk}, $\sigma$ is dynamically marked by $h_\lambda(x)$ rel $\mathrm{P}_\lambda$. Let $g\in \Gamma_{\lambda,\sigma}$ and let $\psi_\lambda:U_\lambda\to V_\lambda$ denote the corresponding marker map rel $\mathrm{P}_\lambda$, with its dynamical labeling of the critical values of $g$, $v_1=\psi_\lambda(-\lambda^2/4)$ and $v_2=\psi_\lambda(h_\lambda(x))$. Recall that the domain of $\psi_\lambda$ is $U_\lambda=U_\lambda(h_\lambda(x))$ as defined in Section \ref{sect:dynmarkrelPl}, and that $U_\lambda\subset\Sigma_\lambda$
 if $h_\lambda(x)\in \Sigma^*_\lambda\setminus \{-\frac{\lambda^2}{4}\}$
 and $U_\lambda=\mathrm{P}_\lambda^{-l}(\Sigma_\lambda)$ if $h_\lambda(x)\in A_\lambda\setminus \Sigma_\lambda$, with $l=l(x)\geq 1$ minimal such that $\mathrm{P}_\lambda^{l}(h_\lambda(x))\in \Sigma_\lambda$ and equivalently minimal so that $\mathrm{P}^{l}(x)\in \widehat\Omega_{m(\lambda)}$.

Let $\mathcal{U}=\mathcal{U}_{m(\lambda)}(x)$ denote the (restricted) domain in $\mathcal{A}_{m(\lambda)}$ as defined above, then $h_\lambda(\mathcal{U})\subset U_\lambda$ and we can define $\psi:=\psi_\lambda\circ h_\lambda: \mathcal{U}\to V$, where $V=\psi(\mathcal{U})$. By the properties of $\psi_\lambda$ and $h_\lambda$ and since $\psi(-\omega^2/4)=\psi_\lambda(h_\lambda(-\omega^2/4))=\psi_\lambda(-\lambda^2/4)=v_1$ and $\psi(x)=\psi_\lambda\circ h_\lambda(x)=v_2$, it follows that $g$ is dynamically marked by $(x,\psi)$ rel $\mathrm{P}$. 
The uniqueness of $h_\lambda$ (cf. Proposition \ref{prop:maphlambda}) and of dynamical markings rel $\mathrm{P}_\lambda$ (cf. Lemma \ref{lemma:pullback}) imply that $\sigma$ is the unique parameter in $\mathcal{R}^{\lambda}$ which is marked by $x$ rel $\mathrm{P}$ and $\psi$ is unique up to automorphisms of $g$.

Moreover, if $x\notin\Omega_{m}$, then $h_\lambda(x)\notin \Sigma_\lambda$, whence $x$ is the unique marking of $\sigma$ rel $\mathrm{P}$. If $x\in\Omega_{m}\setminus \partial\mathcal{P}_{1/q}$ then $h_\lambda(x)$ and $I_\lambda(h_\lambda(x))$ are the unique markings of $\sigma$ rel $\mathrm{P}_\lambda$, however $I_\lambda(h_\lambda(x))$ is not in the image $h_\lambda(\mathcal{A}^*_{M})$, whence $x$ is the unique marking of $\sigma$ rel $\mathrm{P}$.

If $x\in \partial\mathcal{P}_{1/q}$ then $x'\in \partial\mathcal{P}_{1/q}$, so that $h_\lambda(x')=I_\lambda\circ h_\lambda(x)$, is the only other  marking of $\sigma$ rel $\mathrm{P}$.
With reference to the parabolic basin alone, the condition becomes: if $x\in \partial\mathcal{P}_{1/q}^{p+j}$ then $x'\in \partial\mathcal{P}_{1/q}^{p-j}$, both $p+j$ and $p-j$ taken modulo $q$, with $\Im(\phi(x'))=-\Im(\phi(x))$ is the only other marking of $\sigma$ rel $\mathrm{P}$.

Lastly, the partial parametrization $\Phi^\lambda: \mathcal{A}^*_{M} \to \mathcal{R}^{\lambda}$ is thus defined by $\Phi^\lambda=(\chi^\lambda)^{-1}\circ \pi_\lambda\circ h_\lambda$ on $\mathcal{A}^*_{M}$. 
It follows immediately from the definition and properties of the maps $\chi^\lambda$, $\pi_\lambda$ and $h_\lambda$ (see Section \ref{chap:gkmodel} and Proposition \ref{prop:maphlambda}), that $\Phi^\lambda$ is holomorphic on $\inter(\mathcal{A}^*_{M})$ and 
the restriction to 
$\mathrm{P}^{-n}(\widehat{\mathcal{P}}_{0})\cap \mathcal{A}^*_{M}$ is continuous for every $n$. 
Since the constituent maps in $\Phi^\lambda$ depend analytically on the parameter $\lambda$, so does $\Phi^\lambda$, that is, for every $x\in \mathcal{A}^*_{M}$, $\lambda\mapsto \Phi^\lambda(x)$ is analytic. 
\end{proof}

The proof shows how not every $\sigma\in \mathcal{R}^{\lambda}$ is dynamically marked rel $\mathrm{P}$, only those that are marked rel $\mathrm{P}_\lambda$ by a point in $A_\lambda\setminus \Gamma$. Moreover, we remark that if $g\in \Gamma_{\lambda,\sigma}$ is dynamically marked by $x\in\Omega$, then $h_\lambda(x)\in \Sigma_\lambda$ (equivalently $v_2\in \Sigma_{\ls}$). In this case $|\Im(\phi(x))|=H$, where $H=H_{p/q}(\lambda,\sigma)$ is the visible height of $v_2$ rel $\mathrm{P}$ defined in Section \ref{sect:modulus}, cf. Figures \ref{fig_visheight} and \ref{fig:hlambda}.

\subsection{Dynamical markings rel $\mathrm{P}$ on $\mathcal{R}^\omega$}\label{sectdynmarkrelPonRomega}
Dynamical markings rel $\mathrm{P}$ for parameters $\sigma\in\mathcal{R}^\omega\setminus \mathcal{D}^\omega$ can be constructed using Fatou coordinates, in the same way that linearizing coordinates are used to construct dynamical markings rel $\mathrm{P}_\lambda$ for parameters $\sigma\in\mathcal{R}^\lambda$. 
Note that when $\lambda=\omega$, we have $m(\omega)=\infty$ so we shall use the notation $\mathcal{U}=\mathcal{U}(x)\subset \mathcal{A}= \mathcal{A}_\infty$ for the domain $\mathcal{U}_{m(\omega)}(x)=\mathcal{U}_\infty(x)$ of the dynamical marker map in this case.  

\begin{proof}[Proof of Proposition \ref{prop:parabolicmarkings}]
Recall that $\phi$ is a Fatou coordinate for $\mathrm{P}$. Let $g\in \Gamma_{\omega,\sigma}$, with a fixed point at $z_0$ of eigenvalue $\omega$, and let $\varphi:=\varphi_{\omega,\sigma}$ denote a Fatou coordinate for $g$, with relative normalization between components $B^j(g)$ so that $\varphi\circ g=1/q+\varphi$. Let $\mathcal{P}(g)=\bigcup_{j\in\Z_q}\mathcal{P}^j(g)$ denote a maximal attracting flower for $g$ as defined for $\mathrm{P}$ in Section \ref{results}, then there is at least one critical point of $g$ in its boundary. Label the critical points of $g$ so that $c_1\in\partial \mathcal{P}(g)$, then $v_1=g(c_1)$ and $v_2$ is the other critical value of $g$. We normalize $\varphi$ so that $\varphi(v_1)=1/q$. Let $\varphi_j^{-1}$
 denote the inverse of the restriction of $\varphi$ to each $\mathcal{P}^j(g)$.
 
We first define $\psi$ on each petal $\mathcal{P}^j_0$ of $\mathrm{P}$ onto $\mathcal{P}^j(g)$ as $\psi:=\varphi^{-1}\circ\phi$, and then extend continuously to 0 by $\psi(0)=z_0$, so that $\psi$ is defined from $\widehat{\mathcal{P}}_0$ onto $\widehat{\mathcal{P}}(g)=\mathcal{P}(g)\cup \{z_0\}$. Either $v_2\in \widehat{\mathcal{P}}(g)$ and we let $x=\psi^{-1}(v_2)$ and $\mathcal{U}=\widehat{\mathcal{P}}_0$. Or there is a minimal $n>0$ so that $g^n(v_2)\in \widehat{\mathcal{P}}(g)$, in which case we define $\psi$ on $\mathrm{P}^{-n}(\widehat{\mathcal{P}}_0)$ by iterated use of the functional equation $\psi\circ\mathrm{P}=g\circ\psi$. Then we let $x=\psi^{-1}(v_2)$ and $\mathcal{U}=\mathcal{U}(x):=\mathrm{P}^{-n}(\widehat{\mathcal{P}}_0)\cap \mathrm{P}^{-l}( \widehat\Omega)$, where $0\leq l\leq n$ is minimal so that $\mathrm{P}^{l}(x)\in \widehat\Omega$. In both cases, $x=\psi^{-1}(v_2)\in \mathcal{A}^*$ is a dynamical marking of $\sigma$ rel $\mathrm{P}$, with corresponding dynamical marker map $\psi:\mathcal{U}\to V$, where $V=\psi(\mathcal{U})$. 
Moreover, $x\notin S^p$ by results in \cite{u10}, see the discussion preceding Proposition \ref{prop:parabolicmarkings}.

Again, by the properties of dynamical marker maps, $\varphi\circ \psi$ is a Fatou coordinate for $\mathrm{P}$, and we have chosen the relative normalization so that $\varphi\circ \psi=\phi$, by letting $\varphi(v_1)=1/q$. The only other choice lies in the labeling of the critical values of $g$ when $x\in \partial \mathcal{P}_{1/q}$, which leads to the statement.  
\end{proof}

For $\lambda\in \D_{-\omega}$ and for $\sigma \in \mathcal{R}^\lambda$ we have from Proposition \ref{prop:markinglambdarelP} and its proof that if $g\in \Gamma_{\lambda,\sigma}$ is dynamically marked by $(x,\psi)$ rel $\mathrm{P}$, then $\sigma$ is the unique parameter in $\mathcal{R}^{\lambda}$ dynamically marked by $x$ and the marker map $\psi$ is unique up to automorphisms of $g$.
A more technical pullback argument shows similarly for $\lambda= \omega$:

\begin{prop}\label{thm:injective1}
	Let $\sigma, \sigma' \in \mathcal{R}^\omega\setminus \mathcal{D}^\omega$, so that $g\in \Gamma_{\omega,\sigma}$ is dynamically marked by $(x,\psi)$ and $g'\in\Gamma_{\omega,\sigma'}$ is dynamically marked by $(x,\psi')$, both rel $\mathrm{P}$, for some $x \in \mathcal{A}^*$.  Then $\psi'\circ\psi^{-1}$ extends to a M\"obius conjugacy of $g$ to $g'$, in particular $\Gamma_{\omega,\sigma}=\Gamma_{\omega,\sigma'}$ and equivalently $\sigma=\sigma'$.
\end{prop}\label{thm:injective}
Together with the discussion of uniqueness in the proof of Proposition \ref{prop:markinglambdarelP} this shows our claim in Section \ref{sectdynmarkrelP}, about the uniqueness of dynamical markings rel $\mathrm{P}$. 

In the situation of Proposition \ref{thm:injective1}, the dynamical markings induce dynamical labels $v_1=\psi(-\omega^2/4)$, $v_2=\psi(x)$ and $v'_1=\psi'(-\omega^2/4)$, $v'_2=\psi'(x)$ of the critical values of $g$ and $g'$ respectively. Since $g$ and $g'$ are both dynamically marked by $x$ rel $\mathrm{P}$, the marker maps $\psi$ and $\psi'$ have the same domain, $\mathcal{U}=\mathcal{U}(x)=\mathrm{P}^{-n}(\widehat{\mathcal{P}}_{0})\cap \mathrm{P}^{-l}(\widehat{\Omega})$ for $n\geq 0$ and $0 \leq l\leq n$ minimal so $\mathrm{P}^n(x)\in \widehat{\mathcal{P}}_{0}$ and $\mathrm{P}^l(x)\in \widehat{\Omega}$. As usual let $V=\psi(\mathcal{U})$ and $V'=\psi'(\mathcal{U})$ denote the ranges of $\psi$ and $\psi'$ respectively. 

Let $\Psi_0:=\psi'\circ\psi^{-1}:V\to V'$. Then, by the properties of dynamical markings, $\Psi_0$ is a homeomorphism, holomorphic on $\inter(V)$ and it conjugates $g$ to $g'$ where it is defined.  
For the proof of Proposition \ref{thm:injective1} we need the following Lemma.
\begin{lemma}\label{lemma:qcextinj}
 Let $\sigma, \sigma' \in \mathcal{R}^\omega\setminus \mathcal{D}^\omega$, so that $g\in \Gamma_{\omega,\sigma}$ is dynamically marked by $(x,\psi)$ and $g'\in\Gamma_{\omega,\sigma'}$ is dynamically marked by $(x,\psi')$, both relative to $\mathrm{P}$, for some $x \in \mathcal{A}^*$. The map $\Psi_0=\psi'\circ\psi^{-1}: V\to V'$ can be continued to a $K$--quasiconformal map
  $\widehat \Psi_0:\CC\to\CC$, so that $\widehat\Psi_0=\Psi_0$ on
  $V$.
\end{lemma}

The situation in Proposition \ref{thm:injective1} and Lemma \ref{lemma:qcextinj} is illustrated in Figure
\ref{fig:qcextension1}. We first prove Proposition \ref{thm:injective1}, assuming Lemma \ref{lemma:qcextinj}, and then prove Lemma \ref{lemma:qcextinj}.
\begin{figure}[htbp]
 \begin{center}
    \includegraphics[width=13cm]{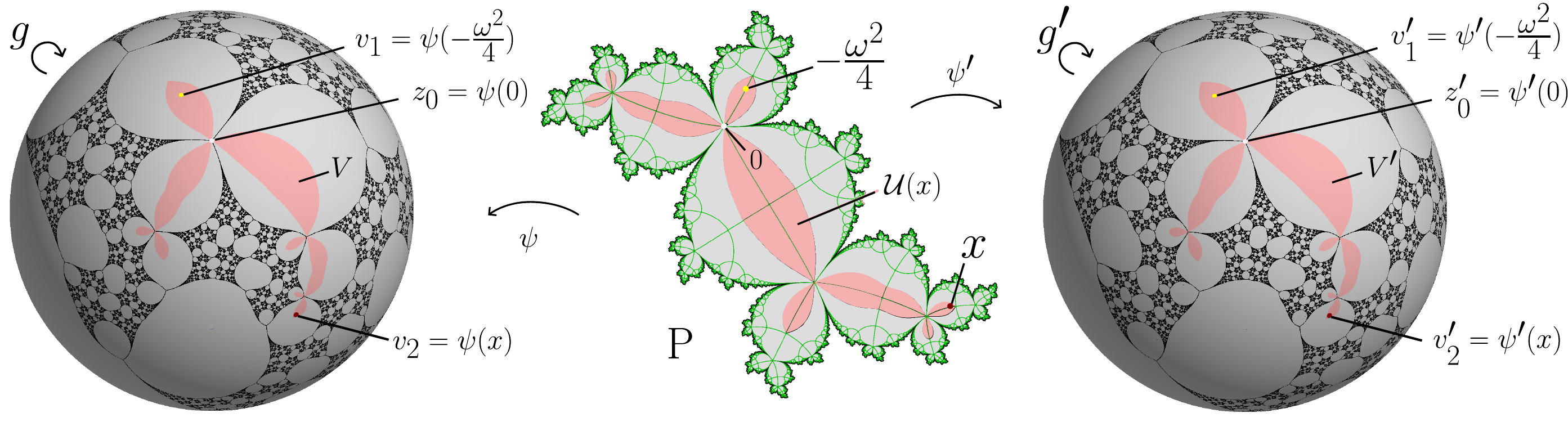}
  \end{center}
  \caption{\label{fig:qcextension1} Illustrates dynamical marking rel $\mathrm{P}_{\omega_{1/3}}$ for two maps $g\in \Gamma_{\omega_{1/3},\sigma}$ and $g'\in\Gamma_{\omega_{1/3},\sigma'}$, that are both marked by the same $x$, as in Proposition \ref{thm:injective1}. The middle shows $\mathrm{P}_{\omega_{1/3}}$ with the domain $\mathcal{U}(x)$ of $\psi$ and $\psi'$ in black, and the left and right pictures illustrate $g$ and $g'$ with the ranges $V$ and $V'$ in pink. }
\end{figure}

\begin{proof}[Proof of Proposition \ref{thm:injective1}]
As introduced above, let  $\Psi_0:=\psi'\circ\psi^{-1}:V\to V'$.
Because of the dynamical labeling induced by the marker maps, we have $\Psi_0(v_1)=v_1'$ and $\Psi_0(v_2)=v_2'$.
    
The next step is to construct a sequence of $K$--quasiconformal maps
$(\widehat\Psi_n:\CC\to\CC)_n$, with the property that they are conformal conjugacies of $g$ to $g'$ on larger and larger domains. 
The first map in the sequence is the $K$--quasiconformal map $\widehat \Psi_0:\CC\to\CC$ from Lemma \ref{lemma:qcextinj}, which is a conjugacy of $g$ to $g'$ on $V$ and conformal on $\inter(V)$, since $\widehat\Psi_0=\Psi_0$ on $V$.

For $n\in\N_0$, let $V_n=g^{-n}(V)$ and $V'_n=(g')^{-n}(V)$. Since $\widehat\Psi_0(v_1)=\Psi_0(v_1)=v_1'$ and $\widehat\Psi_0(v_2)=\Psi_0(v_2)=v_2'$, maps $\widehat\Psi_n$ can be defined by iterated lifting. To be precise, for $n\in\N$, the map
$\widehat\Psi_n:\CC\to\CC$ is defined as the lift of $\widehat\Psi_{n-1}\circ
g:\CC \to \CC$ to $g': \CC \to \CC$, where the
lift that coincides with $\widehat\Psi_{n-1}$ on $V_{n-1}$ is chosen.

Hence $\widehat\Psi_n$ conjugates $g$ to $g'$ on $V_n$ and
the restriction to $\inter(V_n)$ is conformal.
The map $\widehat\Psi_n$ is still (at worst) $K$--quasi conformal, since
pull--back by holomorphic maps can change the orientation
of the complex structure defined by $\widehat\Psi_{n-1}$, but not the
dilatation.

Thus, there exists a sequence of $K$--quasiconformal maps $(\widehat\Psi_n)_n$,
where any two members in the sequence coincide on their common domain
of conformality, and they all coincide on $V$. 

Let $A(g)$ denote the parabolic basin of the parabolic fixed point $z_0$ of $g$ (and $A(g')$ the basin of $z_0'$ for $g'$). Since $\bigcup_{n\in \N} V_{n}=A(g)\cup \bigcup_{n=0}^{\infty}g^{-n}(z_0)$ and $\bigcup_{n\in \N} V'_{n}=A(g')\cup \bigcup_{n=0}^{\infty}(g')^{-n}(z'_0)$ it follows that $\widehat\Psi_n$ converges pointwise everywhere on $A(g)\cup \bigcup_{k=0}^{\infty}g^{-k}(z_0)$
to a limiting map $\Psi_{\infty}:\CC\to\CC$,
where the restriction of $\Psi_{\infty}$ to any
domain $V_N$ is equal to $\widehat\Psi_n$, for all $n\geq N$, hence the restriction
$\Psi_{\infty}:A(g) \to A(g')$ is a holomorphic map.

Since $\widehat\Psi_n=\Psi_0$ in $V$ for all $n\in\N$, the sequence $(\widehat\Psi_n)_n$ is normal (cf. \cite[Satz II.5.1]{lehto65}) and any convergent subsequence converges locally uniformly to a $K$--quasiconformal limit function (cf. \cite[Satz II.5.3]{lehto65}). 
Consider a subsequence which converges to such a $K$--quasiconformal limit function $\Psi:\CC\to\CC$. Then, by the
pointwise convergence of the full sequence, $\Psi=\Psi_{\infty}$ on $A(g)\cup \bigcup_{k=0}^{\infty}g^{-k}(z_0)$.  
Thus, $\Psi$ is conformal everywhere, except possibly on the
Julia set $J=\CC\setminus A(g)$ of $g$.

The maps $g$ and $g'$ have no critical points in their Julia sets, so they are so-called critically non-recurrent rational maps (maps so that any critical point in the Julia set is non-recurrent), whence the Hausdorff dimension of the Julia sets are smaller than 2, $HD(J(g))<2$ and $HD(J(g'))<2$ (see \cite{ur94}). This
implies that the Julia sets $J(g)$ and $J(g')$ have Lebesgue
measure 0 and it follows from Weyl's lemma that $\Psi$ is in fact
conformal on $\CC$.

Furthermore, by construction $\Psi$ is a conjugacy of $g$ to 
$g'$ on the basin $A(g)$, hence $\Psi$ is a conjugacy
on $\CC=\overline{A(g)}$. So $\Psi_0:=\psi'\circ\psi^{-1}:V\to V'$ extends to a M\"obius conjugacy of $g$ to $g'$, in particular $\Gamma_{\omega,\sigma}=\Gamma_{\omega,\sigma'}$ and equivalently $\sigma=\sigma'$.
\end{proof}

\begin{proof}[Proof of lemma \ref{lemma:qcextinj}]
We first extend $\Psi_0$ by hand to a neighborhood of the fixed point $z_0$ of $g$. For the following construction the reader is recommended to consult Figures \ref{fig:qcextension2} and \ref{fig:qcextension3}.

Let $\varphi$ and $\varphi'$ denote the extended Fatou coordinates for $g$ and $g'$ respectively, normalized so $\varphi(v_1)=\varphi'(v'_1)=\frac{1}{q}$. 
Let $N>|\Im(\varphi(v_2)|=|\Im(\varphi'(v_2')|$ when $v_2\in A(g)$ and $N>0$ when $v_2$ is a pre-imge of $z_0$, and let $\mathcal{S}$ denote the collection of $2q$ sepals at $z_0$, i.e. $\mathcal{S}$ is the collection of connected components of $\varphi^{-1}(\{z=x+iy: |y|> N\})$ with $z_0$ on the boundary. Each component of $\mathcal{S}$ and its boundary are invariant under $g^q$. Let $\mathcal{S}'$ denote the collection of sepals for $g'$ similarly defined, and let $W=\mathcal{S}\cup V$ and $W'=\mathcal{S}'\cup V'$. Since 
$v_2\notin \mathcal{S}$ and $v'_2\notin \mathcal{S}'$,
$\Psi_0$ immediately extends to a holomorphic conjugacy of $g$ to $g'$ on $W$, as well as on $\partial \mathcal{S}$. See Figure \ref{fig:qcextension2}.
  \begin{figure}[htbp]
  \centering
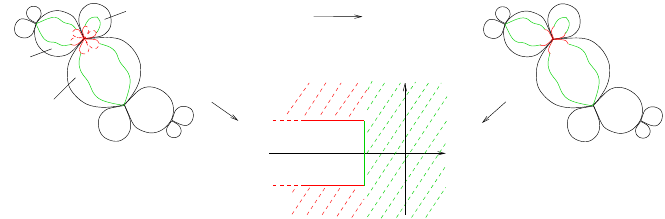
  \caption{\label{fig:qcextension2} A sketch showing (for $p/q=1/3$) the system of sepals $\mathcal{S}$ (red) and $V$ (green), which together form $W$ (and $\mathcal{S}'$, $V'$ $W'$ for $g'$).}
\end{figure}

There are $q$ repelling directions at the fixed point $z_0$ of $g$, 
let $r_j$ denote the repelling direction between the
basin components $B^j=B^j(g)$ and $B^{j+1}=B^{j+1}(g)$ for $g$. Let
$\mathcal{P}_{rep,j}$ be a repelling petal for $g$ for the direction
$r_j$, small enough so that the intersection
$\mathcal{P}_{rep,j}\cap \partial W$ is contained in $\partial\mathcal{S}$, and hence consists of two backward invariant (under $g^q$) curves, denoted $\gamma_{j,1}$ and $\gamma_{j,2}$ and labeled so that $\gamma_{j,1}\subset B^j$. To ease notation let $\gamma_{1}:=\gamma_{j,1}$ and $\gamma_{2}:=\gamma_{j,2}$. See Figure \ref{fig:qcextension3}. Repelling petals $\mathcal{P}'_{rep,j}$ and curves $\gamma'_{1}:=\gamma'_{j,1}$ and $\gamma'_{2}:=\gamma'_{j,2}$ are similarly defined at the fixed point $z_0'$ for $g'$. 
\begin{figure}[htbp]
  \centering
  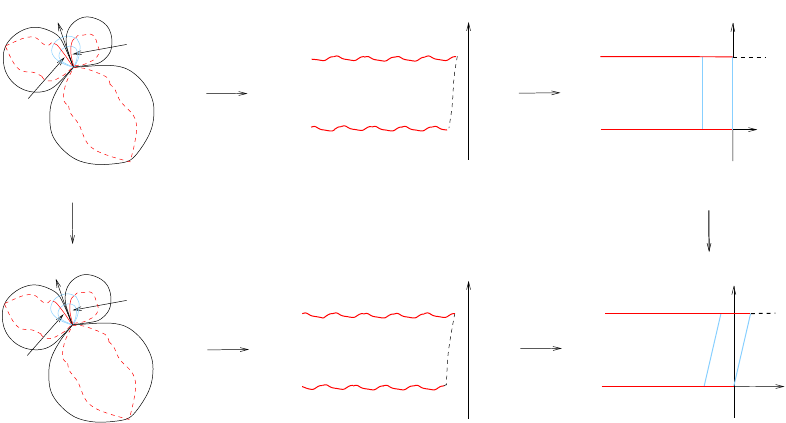
  \caption{\label{fig:qcextension3}A sketch of the extension of
    $\Psi_0$ to a neighborhood of $z_0$. Fundamental domains for the repelling petals $\mathcal{P}_{rep,1}$ and $\mathcal{P}'_{rep,1}$ shown in blue.}
\end{figure}
Then $W\cup \bigcup_{j\in\Z_q}\mathcal{P}_{rep,j}$ and $W'\cup \bigcup_{j\in\Z_q}\mathcal{P}'_{rep,j}$ are neighborhoods of $z_0$ and $z_0'$ respectively.

We will extend $\Psi_0$ quasiconformally on each repelling petal $\mathcal{P}_{rep,j}$, between two adjacent sepal boundary curves $\gamma_{1}$ and $\gamma_{2}$, so that the extension agrees with $\Psi_0$ on the curves $\gamma_{1}$ and $\gamma_{2}$.

Let $\varphi_{rep}:=\varphi_{rep}^{j}$ denote a repelling Fatou coordinate  for the
direction $r_j$ for $g$, i.e. it conjugates $g^q$ to translation by 1. Then the curves $\gamma_{1}$ and $\gamma_{2}$ are mapped by $\varphi_{rep}$ to 
two 1-periodic curves
$\widehat\gamma_{1}$ and $\widehat\gamma_{2}$ bounding a one--sided infinite
strip $C$, or equivalently to two simple closed curves in the
corresponding \`Ecalle cylinder. These two simple closed curves bound
a cylinder with finite modulus embedded in the
\`Ecalle cylinder, denote the modulus $B$. The cylinder is uniformized to a straight cylinder, with
the same modulus and circumference. The uniformizing map induces
a uniformizing map of the strip $C$, to a horizontal strip, $\tilde
C$, we denote this map by $\eta$. By construction $\eta$ commutes with translation by 1. The same holds for $g'$, where we analogously define $\varphi'_{rep}$, $\widehat\gamma'_{1}$, $\widehat\gamma'_{2}$, $C'$, $\tilde C'$ and $\eta'$. See again Figure \ref{fig:qcextension3}.

The curves $\widehat\gamma_{1}$, $\widehat\gamma_{2}$, $\widehat\gamma'_{1}$ and $\widehat\gamma'_{2}$ map to straight, horizontal lines $\tilde\gamma_1=\eta(\widehat\gamma_1)$, $\tilde\gamma_2 =
\eta(\widehat\gamma_2)$, $\tilde\gamma'_1=\eta'(\widehat\gamma'_1)$ and $\tilde\gamma'_2 =
\eta'(\widehat\gamma'_2)$. We will normalize so that
$\tilde\gamma_1\subset \R$ and $\tilde\gamma'_1\subset \R$. The restrictions $\Psi_0:\gamma_1\to
\gamma_1'$ and $\Psi_0:\gamma_2\to \gamma_2'$ descend
to real--analytic diffeomorphisms $\widetilde\Psi_1:\tilde\gamma_1\to
\tilde\gamma_1'$ and $\widetilde\Psi_2:\tilde\gamma_2\to
\tilde\gamma_2'$, which are in particular quasi--symmetric. The original map $\Psi_0$
conjugates $g$ to $g'$, so the functions
$\widetilde\Psi_1$ and $\widetilde\Psi_2$ commute with 
$z\mapsto z+1$, see Figure \ref{fig:qcextension3}.

Next we consider fundamental domains for $z\mapsto z+1$ in the strips $\tilde C$ and $\tilde C'$. Let the uniformizing maps $\eta$ and $\eta'$ be further normalized so that $0\in\tilde\gamma_1$ and $i B\in\tilde\gamma_2$
and so that $\tilde\Psi_1(0)=0$. 
Let $F$ denote the fundamental domain in $\tilde C$ bounded
by the segment $i[0,B]$ in the imaginary axis and the real segment
$[-1,0]$. Let $B_2=\widetilde\Psi_2(i B)\in \tilde\gamma_2'$ and let $F'$ denote the fundamental domain in $\tilde C'$,
which is the quadrilateral determined by the three corners $(-1,0,B_2)$. Then $\widetilde\Psi_1$ maps the lower edge of $F$ to the lower edge of $F'$ and $\widetilde\Psi_2$ maps the upper edge of $F$ to the upper edge of $F'$.

The next objective is to construct a quasiconformal
map $\tilde\Psi_0$ between the fundamental domains $F$ and $F'$, which agrees with $\widetilde\Psi_1$ and $\widetilde\Psi_2$ on the lower and upper edges. We first define $\tilde\Psi_0$ on $\partial F$. On the lower edge we of course let $\tilde\Psi_0=\widetilde\Psi_1$ and on the upper edge $\tilde\Psi_0=\widetilde\Psi_2$. On the right edge, $\tilde\Psi_0$ is defined as the affine map, sending the segment
$i[0,B]$ to the segment $[0,1]B_2$ and on the left edge $\tilde\Psi_0$ is defined as the affine map, sending the segment $-1+i[0,B]$ to the segment $-1+[0,1]B_2$. The definitions agree at the corners, so that 
$\tilde\Psi_0:\partial F\to\partial F'$ defined in this way is a $k$--quasi--symmetric map for some $k\geq 1$. By construction $\tilde\Psi_0:\partial F\to\partial F'$ commutes with $z\mapsto z+1$ on the left edge.

By \cite[Satz II.6.3]{lehto65}) the $k$--quasisymmetric map $\tilde\Psi_0:\partial F\to\partial F'$ can be extended to a $K'$--quasiconformal map $\tilde\Psi_0: F\to F'$, which coincides with the
  quasisymmetric map on the boundary, and where $K'$ only depends
  on $k$. 
  
Now let the strips $\tilde C$ and $\tilde C'$ be cut
  off (at the right end) at the segments $i[0,B]$ and $[0,1]B_2$ respectively. Using the map $z\mapsto z+1$, $\tilde\Psi_0$ can be extended to a
  $K'$--quasiconformal map $\tilde\Psi_0:\tilde C\to \tilde C'$, by letting $\tilde\Psi_0(z)=\tilde\Psi_0(z+l)-l$ for $z\in \tilde C$ and with $l\in\N$ so that $z+l\in F$.  
  
  Finally, $\tilde\Psi_0:\tilde C\to \tilde C'$ is lifted to a $K'$ quasiconformal map $\widehat\Psi_0:\mathcal{P}_{rep,j}\setminus W\to\mathcal{P}_{rep,j}'\setminus W'$, which by construction coincides with $\Psi_0$ on the boundaries $\gamma_1$ and $\gamma_2$. We repeat the procedure, first for every repelling direction $r_j$ at $z_0$ to extend $\Psi_0$ to a quasiconformal map on the neighborhood $W\cup \bigcup_{j\in\Z_q}\mathcal{P}_{rep,j}$ of $z_0$, and then similarly to neighborhoods of all the pre-images of $z_0$ in $V$. Let $D\supset V$ and $D'\supset V'$ denote the new domains, thickened in this way. Thus we have constructed a $K''$-quasiconformal extension $\widehat\Psi_0: D\to D'$, where $K''$ is the maximal of the finitely many quasiconformal dilatations involved in the extension. Moreover, by construction $D$ and $D'$ are Jordan domains whose boundaries are quasicircles.

  Hence, by \cite[Satz II.8.1--II.8.3]{lehto65},
  $\widehat\Psi_0:D\to D'$ extends to a $K$-quasi-conformal map of the
  entire Riemann sphere $\widehat\Psi_0:\CC\to \CC$, with maximal dilatation $K$ depending only on $K''$, $D$ and $D'$, and so that
  $\widehat\Psi_0=\Psi_0$ on $V$.
\end{proof}

\section{Proofs of Main Theorems}\label{sectMainProofs}

We will now consider $x$-horizontal sequences for $x\in \mathcal{A}^*$. That is, let $M>1/q^2$ so that $x\in \mathcal{A}^*_M$, let $(\lambda_k)_k\subset D_M$ be a sequence with $\lambda_k\to\omega$ subhorocyclicly, then $(\lambda_k,\sigma_k)_k=(\lambda_k,\Phi_x(\lambda_k))_k$ is an $x$-horizontal sequence.
Let $g_k=g_{\lambda_k,\sigma_k}\in {\bfnf}_{\lambda_k,\sigma_k}$ and let $\psi_k$ denote the dynamical marker map of $g_k$ rel $\mathrm{P}$ corresponding to $x$. The domain $\mathcal{U}_{m(\lambda_k)}(x)$ of $\psi_k$ depends on $\lambda_k$, but for $\lambda_k \in D_M$, all marker maps $\psi_k$ are defined at least on $\mathcal{U}_M(x)$. We use this domain as the common domain for the maps $\psi_k$ in the following, so that $\psi_k:\mathcal{U}_{M}(x)\to V_k$, where $V_k:=\psi_k(\mathcal{U}_{M}(x))$, denotes the dynamical marker map rel $\mathrm{P}$, corresponding to the marking $x$, of the map $g_k$.

Recall that $B^j$ denotes the components of the immediate basin of 0 for $\mathrm{P}$, labeled counterclockwise and so that $B^0$ contains the critical point $-\omega/2$. Let $\mathcal{U}^j=\mathcal{U}_M(x)\cap B^j$ and $V^j_k=\psi_k(\mathcal{U}^j)$. The domains $\mathcal{U}^j$ and $V^j_k$ are isomorphic to disks, forward invariant under $\mathrm{P}^q$ and $g_k^q$ respectively, and each $\psi_k$ is univalent on $\mathcal{U}^j$, where it conjugates $\mathrm{P}^q$ to $g_k^q$.

The next lemma deals with rescalings of sequences of maps $G_{\lambda_k, a_k} \in {\bfnf}_{\lambda_k,\sigma_k}$ with $\sigma_k \to \infty$, which we will need for the proofs of both Theorems \ref{thm:convergence} and \ref{thm:convergenceunbound}. Recall the normal form $G_T(z)=z+\frac{1}{z}+T$. We will use the notation $B_T$ to denote the parabolic basin of $\infty$ for $G_T$. When the critical point 1 is in $B_T$, we let $B_T(1)$ denote the component of $B_T$ which contains 1.

\begin{lemma} [Dynamics at infinity]\label{lemma:dynamicsinfinity}
Let $x \in \mathcal{A}^*_{M}$, and let $\lambda_k\in D_M$ be a sequence converging to $\opq$ subhorocyclicly. 
Let $g_k=G_{\lambda_k, a_k}$ where $\frac{(\lambda_k-2)^2-a_k^2}{\lambda_k^2}=\sigma_k=\Phi_x(\lambda_k)$ and with corresponding dynamical labeling $\psi_k(-\opq/2)=1$.

Assume that $g^i_{k}\to \infty$ locally uniformly on $\C^*$ when $0<i<q$, and $g_{k}^q\to G_T$, $T\in\CC$, locally uniformly on $\C^*$. Then $T\in\C$, $1\in B_T$ and:\\
Either there is a smallest $N=mq+\ell>0$ so $g_k^N(-1)\in V^0_k$, in which case:
\begin{itemize}
\item[\bf{A.}] if $0<\ell<q$ then $G_T^{m'}(-1)=0$ for some $0<m'\leq m$
\item[\bf{B.}] if $\ell=0$ and $m>1$, then $G_T^{m'}(-1)=0$ for some $m' < m$
\item[\bf{C.}] if $\ell=0$ and $m=1$ then $x\in B^p$ and $G_T(-1)\in B_T(1)$. 
\end{itemize}
Or, there is a smallest $N=mq+\ell>0$ so that $g_k^N(-1)=0$, in which case:
\begin{itemize}
\item[\bf{D.}] $G_T^{m'}(-1)=0$ for some $m'\leq m$.
\end{itemize}
\end{lemma}  
We remind the reader that the assumptions on the convergence of $g_k^i$ and $g_k^q$ are indeed relevant, cf. Lemma \ref{BHClemma}. 
\begin{proof}
As described above, let $\psi_k:\mathcal{U}_{M}(x)\to V_k$ denote the dynamical marker map corresponding to the marking $x$ rel $\mathrm{P}$ of $g_k$, and with the corresponding dynamical labeling $\psi_k(-\omega^2/4)=g_k(1)$ (equivalently $\psi_k(-\omega/2)=1$) and $\psi_k(x)=g_{k}(-1)$.

Either $-\omega/2\in \mathcal{U}^0$ (for $x\notin \widehat{\Omega}$) or $-\omega/2\in \partial \mathcal{U}^0$ (for $x\in \widehat{\Omega}$) in which case the domain $\mathcal{U}^0$ can be slightly extended to include $-\omega/2$, in such a way that $\mathcal{U}^0$ is still a disk, and so that $\psi_k$ is still univalent and conjugates $\mathrm{P}^q$ to $g_k^q$. Hence, we can assume that $-\omega/2\in \mathcal{U}^0$.

Notice that by the properties of the marker map $\psi_k$ and its domain $\mathcal{U}_{M}(x)$ we have $V^0_k:=\psi_k(\mathcal{U}^0)\subset \C^*$ for all $k$, and  
\begin{equation}\label{eqnGk1}
    g_k^q(1)=g_k^{q}(\psi_k(-\omega/2))=\psi_k(\mathrm{P}^q(-\omega/2))\in V^0_{k},
\end{equation}
\begin{equation}\label{eqnGkminus1}
    g_k^q(-1)=g_k^{q-1}(g_k(-1))=g_k^{q-1}(\psi_k(x))=\psi_k(\mathrm{P}^{q-1}(x))\in V_{k},
\end{equation}    
for all $k$.

It follows from the general theory of univalent functions (passing to a subsequence if necessary) that the sequence of functions $(\psi_k)_k$ converges uniformly on compact subsets of $\mathcal{U}^0$ to a limiting function $\psi_\infty: \mathcal{U}^0\to \CC$, which is either univalent or constant. Moreover, $\psi_\infty(-\omega/2)=1$ because $\psi_k(-\omega/2)=1$ for all $k$. Since $\psi_k\circ \mathrm{P}^q=g_k^q\circ \psi_k$ on $\mathcal{U}^0$ and since $g_k^q\to G_T$ locally uniformly on $\C^*$ by assumption, $\psi_\infty\circ \mathrm{P}^q=G_T\circ \psi_\infty$ on $\mathcal{U}^0$, as long as $\psi_\infty(z)\neq 0,\infty$. Therefore, we have (in both cases) $2+T=G_T(1)=G_T(\psi_\infty(-\omega/2))=\psi_\infty(\mathrm{P}^q(-\omega/2))$.

Assume $\psi_\infty$ is constant. Since $\psi_\infty(-\omega/2)=1$, $\psi_\infty$ is the map $z\mapsto 1$, whence $T=-1$ because $2+T=\psi_\infty(\mathrm{P}^q(-\omega/2))=1$. Then by Lemma \ref{BHClemma} there are $q$-cycles $\langle z\rangle_k$ of $g_k$ with eigenvalue $\rho_k\to 0$. But $\sigma_k\in \mathcal{R}^{\lambda_k}$, which means that both critical points of $g_k$ are in the basin of attraction of the fixed point at $\infty$, with eigenvalue $\lambda_k\to\omega$. By classical results of Fatou and Julia, each attracting or parabolic basin contains at least one critical point, hence $g_k$ can have no other attracting cycles, for a contradiction.

So $\psi_\infty$ is univalent. The univalence of $\psi_\infty$ implies (as a consequence of the equivalence between Riemann maps and pointed disks, given by the Carath\'eodory kernel Theorem, see f.ex. \cite[Thm. 5.1]{CTM}) that the sequence of pointed disks $(V^0_{k}, g_k^q(1))$ converges in the Carath\'eodory topology to a pointed disk $(V^0_{\infty}, G_T(1))$, where $V^0_{\infty}\subset \C^*$ since $V^0_{k}\subset \C^*$ for all $k$. In particular, $T\in\C$ since $2+T=G_T(1)\in V^0_{\infty}$.

Since $V^0_{\infty}\subset \C^*$ we have $\psi_\infty\circ\mathrm{P}^q=G_T\circ\psi_\infty$ on $\mathcal{U}^0$, and it follows that $V^0_{\infty}$ is a subset of $B_T$, the parabolic basin of $\infty$ for $G_T$, in particular $1\in V^0_{\infty}\subset B_T$.

From \cite[Thm. 5.3]{CTM} if the distance $d_k(g_k^q(1),w_k)$ is uniformly bounded (in $k$) in the hyperbolic metric $d_k$ on $V^0_{k}$, and $w_k\to w$, then $w \in V^0_\infty$. In particular, this implies that if there exists $N>0$ so that $g_k^N(-1)\in V^0_{k}$ (equivalently $\mathrm{P}^{N-1}(x)\in \mathcal{U}^0)$ and $g_k^N(-1)$ converges to some point $w$, then $w \in V^0_\infty$, since the hyperbolic distance 
\begin{equation}\label{hypdistance}
d_k(g_k^q(1),g_k^{N}(-1))=d(\mathrm{P}^q(-\opq/2),\mathrm{P}^{N-1}(x)) \text{ for all } k,
\end{equation}
where $d$ denotes the hyperbolic distance on $\mathcal{U}^0$.

If $x\in \inter(\mathcal{U}_{M}(x))$, then there exists a least $N>0$ so that $g_k^N(-1)\in V^0_{k}$. If $N<q$ then by assumption $g_k^N(-1)\to \infty$ in contradiction to Equation (\ref{hypdistance}). So $N=mq+\ell$ for some $m\geq 1$ and $0\leq \ell<q$. We remark that in all cases, either $g_k^{mq}(-1)\to G^m_T(-1)$ or there exists a minimal $0<m'<m$ so that $g_k^{m'q}(-1)\to 0=G^{m'}_T(-1)$.

Part A: If $0<\ell<q$ then $g_k^N(-1)=g_k^\ell\circ (g_k^q)^m(-1)$, and $g_k^N(-1)\not\to\infty$, again because of Equation (\ref{hypdistance}).
On the other hand $g_k^\ell\to \infty$ locally uniformly on $\C^*$, whence there is a smallest $0<m'\leq m$ so that $g_k^{qm'}(-1)\to 0=G_T^{m'}(-1)$.

Part B: If $\ell=0$ and $m>1$ then there is a smallest $m'<m$ so that $g_k^{qm'}(-1)\to 0=G_T^{m'}(-1)$. If not, then $g_k^N(-1)=g_k^{qm}(-1)\to G_T^{m}(-1)\in V^0_\infty$, by the hyperbolic argument cf. Equation (\ref{hypdistance}). Due to the way the domain $\mathcal{U}_{M}(x)$ is constructed, $\mathrm{P}^{qm-1}(x)$ is at least $qm-1$ levels deeper (with respect to iteration by $\mathrm{P}$) than $\partial \mathcal{U}^0$. Consequently, the two preimages of $G_T^m(-1)$ by $G_T$ belong to $V^0_\infty$, $G_T^{-1}(\{G_T^m(-1)\})\in V^0_\infty$; they correspond to the two points of $g_k^{-q}(\{g_k^{mq}(-1)\})$ in $V_k^0$. 

One of these pre-images, call it $y$, must be in the forward orbit of -1, i.e. $y=G_T^{m-1}(-1)\in V^0_\infty$. But then $g_k^{q(m-1)}(-1)\in V^0_{k}$ for $k$ sufficiently large, in contradiction to the assumption that $n=qm$ was the first such iterate.

Part C: If $\ell=0$ and $m=1$ then $g_k^{q}(-1) \in V^0_{k}$ (equivalently $\mathrm{P}^{q-1}(x)\in \mathcal{U}^0$). First note that if $x\in B^p$, then $\mathrm{P}^{q-1}(x)\in \mathcal{U}^0$ and hence $g_k^{q}(-1) \in V^0_{k}$ by Equation (\ref{eqnGkminus1}), and again by the hyperbolic argument above, $g_k^{q}(-1)\to G_T(-1)\in V^0_\infty\subset B_T(1)$.

To see that this is the only possibility, assume to the contrary that 
$x\notin B^p$ whence $\mathrm{P}^{-1}(\{x\})\cap B^0=\emptyset$. In this case $\psi_k$, and consequently $\psi_\infty$, would extend to a univalent conjugacy on all of $B^0$. Either $\mathrm{P}^{q-1}(x)=\mathrm{P}^q(-\omega/2)$ whence $g_k^q(-1)=g_k^q(1)$ for all $k$ by Equations (\ref{eqnGk1}) and (\ref{eqnGkminus1}), implying $G_T(-1)=G_T(1)$ for a contradiction. Or $\mathrm{P}^{q-1}(x)\neq \mathrm{P}^q(-\omega/2)$ and there are two distinct pre-images of $\mathrm{P}^{q-1}(x)$ by $\mathrm{P}^q$ in $B^0$, which, since $\psi_\infty$ is a univalent conjugacy on $B^0$  and $\psi_\infty(\mathrm{P}^{q-1}(x))=G_T(-1)$, would imply that there are two distinct pre-images of $G_T(-1)$ by $G_T$, for a contradiction.

Part D: If $x\in \partial \mathcal{U}_{M,}(x)$, then there exists a smallest $N=qm+\ell>0$ so that $g_k^N(-1)=0$, and the result is obtained in much the same way as before. If $0<\ell<q$ then $g_k^N(-1)=g_k^\ell\circ g_k^{qm}(-1)=0$, but on the other hand $g_k^\ell\to \infty$ locally uniformly on $\C^*$, whence there is a smallest $m'\leq m$ such that $g_k^{qm'}(-1)\to 0=G_T^{m'}(-1)$.
If $\ell=0$, then either there is a smallest $m'<m$ such that $g_k^{qm'}(-1)\to 0=G_T^{m'}(-1)$ or $0=g_k^N(-1)=g_k^{qm}(-1)\to G_T^m(-1)$, whence $G_T^m(-1)=0$.
\end{proof}

\begin{proof}[Proof of Theorem \ref{thm:convergence}]  
Let $x$, $M$, and $(\lambda_k)_k$, with $\lambda_k\to\omega$, be as in the Theorem and consider the $x$-horizontal sequence $(\lambda_k,\sigma_k)=(\lambda_k,\Phi_x(\lambda_k))$. 
Any map $f_k=f_{\lambda_k,\sigma_k}\in {\bfnf}_{\lambda_k,\sigma_k}$ has an attracting fixed point with eigenvalue $\lambda_k$ and $f_k$ is dynamically marked by $x\in \mathcal{A}^*_M\setminus S^p$ rel $\mathrm{P}$, whence it is dynamically marked by $x_k=h_{\lambda_k}(x)$ rel $\mathrm{P}_\lambda$. Since $x\notin S^p$, $x_k=h_{\lambda_k}(x)\notin \overline{S^p_{\lambda_k}}$, so by Corollary \ref{cor:fatcycles}(2) there exist repelling $q$-cycles $\langle z\rangle_k$ of $f_k$ with eigenvalue $\rho_k\to 1$.

We first show that $(\sigma_k)_k$ is bounded. 
Assume to the contrary that $(\sigma_k)_k$ is unbounded, then, on passing to a subsequence if necessary, $\sigma_k\to\infty$ and the fixed point eigenvalues $\lambda_k, \mu_k, \nu_k$ tend to $\opq, \omega_{-p/q}, \infty$ (cf. Section \ref{sect:divergenceinfty}). We can without loss of generality choose representatives $f_k= G_{\lambda_k, a_k}$, where $a_k\in \C$ so that $a_k^2=(\lambda_k-2)^2-\sigma_k\lambda_k^2$ and with dynamical labeling $\psi_k(-\omega^2/4)=G_{\lambda_k, a_k}(1)$ and $\psi_k(x)=G_{\lambda_k, a_k}(-1)$, where $\psi_k: \mathcal{U}_{M}(x)\to V_k$ denotes the dynamical marker map for $f_k$ rel $\mathrm{P}$, as 
described above.

It then follows from Lemma \ref{BHClemma}, passing to a further subsequence if necessary, that 
\[
f_k^i \to
\begin{cases}
\infty & \text{ for } 1\leq i<q \\
G_T  & \text{ for } i=q
\end{cases}
\quad \text{ locally uniformly on } \C^* \text{ as } \lambda_k\to\opq,
\]
for some $T\in\CC$.
Lemma \ref{lemma:dynamicsinfinity} then implies that $T\in\C$ and therefore $T=0$ because of the last part of Lemma \ref{BHClemma}, since there are repelling $q$-cycles $\langle z\rangle_k$ of $f_k$ with eigenvalues $\rho_k\to 1$.

But the map $G_0$ has a double parabolic fixed point at $\infty$, its basin consists of two fixed components, each containing a critical point, 1 or -1, both having infinite forward orbit. So no iterate of -1 by $G_0$ is 0 or in the basin component containing 1, which is in contradiction to all of the possible cases in Lemma \ref{lemma:dynamicsinfinity}. Therefore $(\sigma_k)_k$ must be  bounded.

Since the sequence $(\sigma_k)_k$ is bounded, a subsequence can be extracted for which $(\lambda_k,\sigma_k)\to (\opq,\sigma)$ for some $\sigma\in\C$. We prove that any limit point $\sigma$ of such a convergent subsequence is in $\mathcal{R}^{\opq}$ and is dynamically marked by $x$ rel $\mathrm{P}$.

Without loss of generality, we choose the representatives $f_k= G_{\lambda_k, a_k}\in {\bfnf}_{\lambda_k,\sigma_k}$ as before.
Then $a_k\to a$ as $k\to\infty$, for $a\in \C$ with $a^2=(\omega-2)^2-\sigma\omega^2$, and $f_k$ converges uniformly on $\CC$ to $f=G_{\omega, a}\in {\bfnf}_{\omega,\sigma}$.

From the uniform convergence $f_k\to f$, the functional equation $\psi_{k}\circ\mathrm{P}=f_k\circ\psi_{k}$ and the general theory of univalent maps, it follows that the sequence of marker maps $\psi_k$ converges to a limiting map $\psi_\infty$, with $\psi_\infty(0)=\infty$, $\psi_\infty(-\opq/2)=1$ and $\psi_\infty(x)=f(-1)$. The convergence is locally uniform on each component of $\inter(\mathcal{U}_{M}(x))$ and pointwise on $\mathcal{U}_{M}(x)$. Moreover, the limiting map $\psi_\infty$ is injective, univalent on $\inter(\mathcal{U}_{M}(x))$, and obeys the functional equation:
\begin{equation}\label{eqn:funceqninf}
\psi_{\infty}\circ\mathrm{P}=f\circ\psi_{\infty} \quad \text{on } \mathcal{U}_{M}(x).
\end{equation}

To see that $\sigma \in \mathcal{R}^{\opq}$, we argue in a way similar to the proof of Lemma \ref{lemma:dynamicsinfinity}. Recall that $\mathcal{U}^j=\mathcal{U}_M(x)\cap B^j$ and $V_k^j=\psi_k(\mathcal{U}^j)$ and let $V^j_\infty$ denote the Carath\'eodory limit of $V_k^j$ as in the proof. Now, either there exists $N>1$ so that $\mathrm{P}^{N-1}(x)=0$, and equivalently $f_k^N(-1)=\infty$ for all $k$, whence also $f^{N}(-1)=\infty$. Or there exists $N>0$ so that $\mathrm{P}^{N-1}(x)\in \mathcal{U}^0$, equivalently $f_k^N(-1)\in V_k^0$ for all $k$, whence $f^{N}(-1)\in  V^0_\infty$. Since $V^0_\infty$ is a subset of the parabolic basin of $\infty$ for $f$, it follows that $\sigma \in \mathcal{R}^{\omega}$ in both cases.

The limit map $\psi_{\infty}$ can be extended to the larger domain $\mathcal{U}(x)=\mathrm{P}^{-n}(\widehat{\mathcal{P}}_0)\cap \mathrm{P}^{-l}(\widehat{\Omega})$, as defined in Section \ref{sectdynmarkrelPonRomega}, in such a way that the functional equation (\ref{eqn:funceqninf}) still holds on $\mathcal{U}(x)$. Namely, for each compact set $L\subset\inter(\mathcal{U}(x))$ there is an $M'>M$ such that $L\subset \inter(\mathcal{U}_{M'}(x))$ and since $\lambda_k\to\omega$ subhorocyclicly, $\lambda_k\in D_{M'}$ for $k$ sufficiently large. Then $\psi_k$ converges locally uniformly on $\inter(\mathcal{U}_{M'}(x))$ to a limit map that extends $\psi_\infty$ to $\mathcal{U}_{M'}(x)$.

So we have $\psi_\infty:\mathcal{U}(x)\to\CC$, which is injective, univalent in $\inter(\mathcal{U}(x))$, conjugates $\mathrm{P}$ to $f$ and such that $\psi_\infty(-\omega^2/4)=f(1)$ and $\psi_\infty(x)=f(-1)$. It remains to show that $\psi_\infty$ is continuous on $\mathcal{U}(x)$, in order to establish that $\psi_\infty$ is a dynamical marker map for $f$, which then implies that $f$ is dynamically marked by $x$.

Since $\psi_\infty$ is univalent on $\inter(\mathcal{U}(x))$, it is only necessary to show continuity at points in $\partial \mathcal{U}(x)\cap \mathcal{U}(x)$. Consider first a sequence $z_{k'}\in \mathcal{U}(x)$ so that $z_{k'}\to 0$, then $z_{k'}$ is eventually in $\mathcal{U}(x)\cap \Omega$, we can without loss of generality assume it is in a tile $\Omega^j$. Recall that $\phi$ denotes the Fatou coordinate for $\mathrm{P}$ and let $\varphi$ denote the Fatou coordinate for $f$ as in Section \ref{sectdynmarkrelPonRomega}. They are normalized so $\phi(-\omega/2)=0=\varphi(1)$. Recall from Sections \ref{markingRbyP} and \ref{sectdynmarkrelPonRomega} that $\H_0\subseteq \phi(\inter(\mathcal{U}(x))\subseteq \H_{-n/q}$, where $n=n(x)$ is minimal so that $\mathrm{P}^n(x)\in \widehat{\mathcal{P}}_0$.

Then, since $\psi_\infty$ conjugates  
$\mathrm{P}^q$ on $\mathcal{U}^j$ to $f^q$ on $V_\infty^j$, and because of the chosen normalization, $\varphi\circ \psi_\infty=\phi$ on $\mathcal{U}^j$, so that $\varphi\circ\psi_\infty(z_{k'})=\phi(z_{k'})\to\infty$ in $\H_{-n/q}$ as $k'\to\infty$. Hence $\psi_\infty (z_{k'})\to z$, where $z$ is $\infty$ or one of its pre-images under $f$, and since $\psi_\infty(\Omega^j\cap \mathcal{U}(x))$ is forward invariant under $f^q$, it follows that $z=\infty$. Thus, $\psi_\infty(z_{k'})\to \infty=\psi_\infty(0)$ as $z_{k'}\to 0$.

Next, suppose $z\in \mathcal{U}(x)$ is a pre-image of 0, i.e. there exists a minimal  $l>0$ such that  $\mathrm{P}^{l}(z)=0$, and consider a sequence $\mathcal{U}(x)\ni z_{k'}\to z$, which is necessarily contained in $\inter(\mathcal{U}(x))$ for $k'$ sufficiently big. There are $q$ components of  $\inter(\mathcal{U}(x))$ meeting at $z$, we can again assume, without loss of generality, that the sequence $(z_{k'})_{k'}$ is contained in one of them. Thus, it follows from the continuity at 0 and the functional equation (\ref{eqn:funceqninf}) that $\psi_\infty(z_{k'})\to z_\infty \in f^{-l}(\{\infty\})$ as $z_{k'}\to z$. To see that $z_\infty=\psi_\infty(z)$, let $D(z)$ be a sufficiently small disc around $z$, so that the pre-image by $\phi$ of the positive real axis has $q$ components in $\mathcal{U}(x)\cap D(z)$, with $z$ as the only common boundary point. Let $Y$ denote the closure of this pre-image:
\[Y=\overline{\phi^{-1}(\R_+)\cap \mathcal{U}(x) \cap D(z)}.\] 
Note that $Y$ is closed, connected and that $Y\setminus \inter(\mathcal{U}(x))=\{z\}$. Let $Y_k=\psi_k(Y)$ (where $k$ here is the index from the convergence $\lambda_k\to\omega$), for context we remind the reader that $Y_k$ is a pre-image at $\psi_k(z)$ of a twig for the $p/q$ star for $f_k$. In particular, $Y_k$ is closed and connected, since $\psi_k$ is a homeomorphism, whence on passing to a subsequence if necessary, $Y_k$ converges in the Hausdorff topology to a closed and connected $Y_\infty$. Since $\psi_k(z)\to\psi_\infty(z)$ it follows that $\psi_\infty(z)\in Y_\infty$, and it is the only point in $Y_\infty$ which is not in the parabolic basin of $\infty$ for $f$. It follows that for a sequence $Y \ni z_{k'}\to z$, $Y_\infty \ni \psi_\infty(z_{k'})\to \psi_\infty(z)$, whence $z_\infty=\psi_\infty(z)$.

Since $\psi_\infty$ is continuous, it is a dynamical marker map for $f=G_{\omega, a}\in {\bfnf}_{\omega,\sigma}$ and thus the limit point $\sigma$ is dynamically marked by $x$ rel $\mathrm{P}$. Assume $\sigma'$ is the limit point of another convergent subsequence of $(\sigma_k)_k$. But then $\sigma'=\sigma$ by Proposition \ref{thm:injective1}, whence $\sigma_k\to\sigma$, where $\sigma$ is the unique parameter in $\mathcal{R}^\omega\setminus\mathcal{D}^\omega$ dynamically marked by $x$ rel $\mathrm{P}$.
\end{proof} 

\begin{proof}[Proof of Theorem \ref{thm:convergenceunbound}]
Let $x$, $M$, and $(\lambda_k)_k$, with $\lambda_k\to\omega$, be as in the Theorem and consider the $x$-horizontal sequence $(\lambda_k,\sigma_k)=(\lambda_k,\Phi_x(\lambda_k))$. Any map $f_k=f_{\lambda_k,\sigma_k}\in {\bfnf}_{\lambda_k,\sigma_k}$ has an attracting fixed point with eigenvalue $\lambda_k$ and $f_k$ is dynamically marked by $x\in \mathcal{A}^*_M\cap S^p$ rel $\mathrm{P}$, whence it is dynamically marked by $h_{\lambda_k}(x)$ rel $\mathrm{P}_\lambda$. Since $x\in S^p$, $h_{\lambda_k}(x)\in S^p_{\lambda_k}$, so by Corollary \ref{cor:fatcycles}(1) there exist repelling fixed points $\zeta_k=f_k(\zeta_k)$ with eigenvalue $\mu_k\to \omega_{-p/q}$. Since $\lambda_k\to\opq$, it follows from the discussion in Section \ref{sect:divergenceinfty} that $\sigma_k\to\infty$.

Now consider representatives $f_k= G_{\lambda_k, a_k}$, dynamically marked by $x$ rel $\mathrm{P}$, where $a_k\in \C$ so that $a_k^2=(\lambda_k-2)^2-\sigma_k\lambda_k^2$ and with dynamical labeling $\psi_k(-\omega^2/4)=f_k(1)$ and $\psi_k(x)=f_k(-1)$, where $\psi_k: \mathcal{U}_{M}(x)\to V_k$ denotes the corresponding dynamical marker map. We consider again $\psi_k$ restricted to $\mathcal{U}_{M}(x)$ so that all $\psi_k$ have the same domain, as described in the beginning of this section, and used in the proofs of Lemma \ref{lemma:dynamicsinfinity} and Theorem \ref{thm:convergence}.

It follows from Lemma \ref{BHClemma} that $f_k^i \to \infty$ locally uniformly on $\C^*$ for $0<i<q$, in particular for $i=1$. It also follows from Lemma \ref{BHClemma}, that any subsequence of the sequence $(f_k)_k$ has a further subsequence so that $f_k^q \to G_T$ locally uniformly on $\C^*$, for some $T\in\CC$, as $\lambda_k\to\opq$. 

We consider now such a rescaled limit $G_T$, for a subsequence so that $f_k^q \to G_T$. Then it follows from Lemma \ref{lemma:dynamicsinfinity} that $T\in \C$ and in each of the cases $\sigma=1-T^2\in \mathcal{R}^{1}$.

Recall the surjective, continuous map $h_1: \overline{S^0}\cap \mathcal{A} \to \widetilde{A}_1$, introduced in Section \ref{mainresult2rescaling}, with $h_1(0)=0$ and $h_1(-\omega/2)=-1/2$, and such that 
\begin{equation}
h_1\circ\mathrm{P}^q = \mathrm{P}_{1}\circ h_1 \quad \text{on } \overline{B^0}\cap \mathcal{A}.
\end{equation}
By construction $h_1$ is a homeomorphism on $\overline{B^0}\cap \mathcal{A}$ and holomorphic on $B^0$. 

The proof now uses arguments similar to the proof of Lemma \ref{lemma:dynamicsinfinity} and Theorem \ref{thm:convergence}, and we will do it in less detail here. Let $\widetilde{\mathcal{U}^0}=\mathcal{U}_M(x)\cap \overline{B^0}$. In particular, recall that, passing to a subsequence if necessary, the sequence of marker maps $\psi_k$ converges on $\widetilde{\mathcal{U}^0}$ to a limiting map $\psi_\infty: \widetilde{\mathcal{U}^0}\to \C^*$ with $\psi_\infty(0)=\infty$ and $\psi_\infty(\mathrm{P}^q(-\omega/2))=G_T(1)$. The convergence is locally uniform on $\mathcal{U}^0$ and pointwise on $\widetilde{\mathcal{U}^0}$, and the limiting map is injective and univalent on $\mathcal{U}^0$. Moreover, $\psi_\infty$ is a conjugacy of $\mathrm{P}^q$ to $G_T$ and  $\psi_\infty(\mathcal{U}^0)=V^0_\infty\subset B_{T}$. 
By arguments similar to the proof of Theorem \ref{thm:convergence}, $\psi_\infty$ is continuous on $\widetilde{\mathcal{U}^0}$.

Let $\widetilde{V^0_\infty}=\psi_\infty(\widetilde{\mathcal{U}^0})$, $U_1=h_1(\mathcal{U}^0)$, 
$\widetilde{U_1}=h_1(\widetilde{\mathcal{U}^0)}$ and let $h_1^{-1}:\widetilde{U_1}\to \widetilde{\mathcal{U}^0}$ denote the inverse of the restriction of $h_1$ to 
$\widetilde{\mathcal{U}^0}$.
Then $\psi:=\psi_\infty\circ h_1^{-1}:\widetilde{U_1}\to \widetilde{V^0_\infty}$ is a homeomorphism, holomorphic on $U_1$, conjugates $\mathrm{P}_1$ to $G_T$, and has 
$\psi(0)=\psi_\infty(0)=\infty$ and 
$\psi(-1/4)=\psi_\infty(\mathrm{P}^q(-\omega/2))=G_T(1)$. 

If $x\in B^p$ then $\mathrm{P}^{q-1}(x)\in \mathcal{U}^0$ (equivalently $f^q_k(-1)\in V_k^0$), which corresponds to case C of Lemma \ref{lemma:dynamicsinfinity}. In this case $h_1(\mathrm{P}^{q-1}(x))\in U_1$ and 
\[\psi(h_1(\mathrm{P}^{q-1}(x))=\psi_\infty(\mathrm{P}^{q-1}(x))=G_T(-1)\in V^0_\infty\subset B_T(1).\]

If $x\in S^p\setminus B^p$, then there are two cases. Either there is a smallest $N=mq+\ell>q$ such that $\mathrm{P}^{N-1}(x)\in \mathcal{U}^0$ (equivalently $f_k^N(-1)\in V_k^0$), which corresponds to case A or B of Lemma \ref{lemma:dynamicsinfinity}. Or there is a smallest $N=mq+\ell\geq q$ such that $\mathrm{P}^{N-1}(x)=-\lambda$ (and hence $\mathrm{P}^{N}(x)=0$), and equivalently $f_k^N(-1)=0$ for all $k$, which corresponds to case D of Lemma \ref{lemma:dynamicsinfinity}.
In each case there exists a minimal $1\leq m'\leq m$ such that $f_k^{qm'}(-1)\to G^{m'}_T(-1)=0$, and then  $f_k^{qn'}(-1)\to G^{n'}_T(-1)$ for $n'<m'$, in particular $\psi_k(\mathrm{P}^{q-1}(x))=f_k^{q}(-1)\to G_T(-1)$, where $G_T(-1)$ is a pre-image of 0 by $G_T^{m'-1}$.

As in the definition of $h_1$ (in Section \ref{mainresult2rescaling}) let $z_0$ denote an iterated pre-image of 0 (by $\mathrm{P}$) in $\widetilde{\mathcal{U}^0}\subset \overline{B^0}$. Every such pre-image of 0 separates $S^0$ into $q$ connected components, recall that $C(z_0)$ denotes the union of components not containing $-\omega/2$ and let $\overline{C}(z_0)=C(z_0)\cup \{z_0\}$, see Figure \ref{fig_h1}. Recall that $h_1(z)=h_1(z_0)$ for all $z\in \overline{C}(z_0)$.

Since $x\in S^p\setminus B^p$, $\mathrm{P}^{q-1}(x)\in \mathcal{U}_M(x)\cap \overline{C}(z_0)$ for some $C(z_0)$ at an iterated pre-image $z_0\in \widetilde{\mathcal{U}^0}$ of 0, whence $h_1(\mathrm{P}^{q-1}(x))=h_1(z_0)$, see Figure \ref{fig_h1}. 

\paragraph{Claim.} The restriction of $\psi_k$ to $\mathcal{U}_M(x)\cap \overline{C}(z_0)$ converges (pointwise) to the constant map $\psi_\infty|_{\mathcal{U}(x)\cap \overline{C}(z_0)}\equiv \psi_\infty(z_0)$.
  \begin{proof}[Proof of claim]
If not, then there would be a component $C'$ of $\inter(\mathcal{U}_M(x))\cap C(z_0)$, so that (a subsequence of) $\psi_k|_{C'}$ converges locally uniformly to a univalent map $\psi_\infty|_{C'}$. By construction of the maps $\psi_k$ and by the properties of Carath\'eodory convergence, $\psi_\infty(C')\subset\C^*$ and $\psi_\infty(C')\cap V_\infty^0=\emptyset$. The component $C'$ has a center denoted $z_1$: the point in $C'$ which first iterates to the critical point $-\omega/2$, that is $\mathrm{P}^{N_1}(z_1)=-\omega/2$ for $N_1=m_1q+\ell_1\leq N=mq+\ell$, and equivalently $f_k^{N_1}(\psi_k(z_1))=1$. Since $\psi_\infty(z_1)$ is contained in the open set $\psi_\infty(C')$, which does not intersect $V_\infty^0$, it follows that $\psi_\infty(z_1)\notin\widetilde{V^0_\infty}$.
By arguments similar to those in Lemma \ref{lemma:dynamicsinfinity} there either exists $m_1'\leq m_1$ so that $G_T^{m_1'}(\psi_\infty(z_1))=0$ or $G_T^{m_1}(\psi_\infty(z_1))=1$. But by construction of $\mathcal{U}_M$, we both have $G_T^{-m_1'}(\{0\})\subset\widetilde{V^0_\infty}$ and $G_T^{-m_1}(\{1\})\subset V^0_\infty$, for a contradiction. 
\end{proof}
The claim implies that $\psi_\infty(z_0)=\psi_\infty(\mathrm{P}^{q-1}(x))$ and since $\psi_k(\mathrm{P}^{q-1}(x))=f_k^q(-1)$ we also have that  $\psi_\infty(\mathrm{P}^{q-1}(x))=G_T(-1)$. It follows that 
\[\psi(h_1(\mathrm{P}^{q-1}(x))=\psi(h_1(z_0))=\psi_\infty(z_0)=G_T(-1)\in \widetilde{V^0_\infty}\setminus V^0_\infty.\]

In both cases ($x\in B^p$ or $x\in S^p\setminus B^p$) this shows that $\psi:=\psi_\infty\circ h_1^{-1}:\widetilde{U_1}\to \widetilde{V^0_\infty}$ is a dynamical marker map rel $\mathrm{P}_1$ of $G_T$, whence the rescaled limit $G_T$ (and its parameter $\sigma=1-T^2$) is dynamically marked rel $\mathrm{P}_1$ by $h_1(\mathrm{P}^{q-1}(x))$. By uniqueness of dynamical markings 
we then have convergence of the full sequence $f_k^q\to G_T$ with the required properties.
\end{proof}

\paragraph{Acknowledgements.} \quad The author would like to thank Carsten Lunde Petersen and Adam Epstein for many helpful discussions and suggestions, and for input to formulations of some of the key concepts. The author thanks Arnaud Ch\'eritat and Christian Henriksen for contributing their beautiful pictures. Furthermore, the author is very grateful to Harvard Mathematics Department and the Institute for Mathematical Sciences at Stony Brook University for their hospitality, and the Carlsberg Foundation for their generous support, via the grants 2010\_01\_0785 and 2011\_01\_0759, during the initial development of the concepts and ideas of this paper.

\end{document}